\documentclass[usenames,dvipsnames,letterpaper,11pt]{article} 
\usepackage{setspace}
\usepackage{amsmath}
\usepackage[T1]{fontenc}
\usepackage[margin=1in]{geometry}
\usepackage[numbers, semicolon]{natbib}
\setcitestyle{authoryear,open={(},close={)},aysep={}} 
\usepackage{amsfonts}
\usepackage{float}
\usepackage{graphicx}
\usepackage{tikz}
\usetikzlibrary{arrows,shapes,snakes,automata,backgrounds,petri,fit, positioning,patterns}
\usepackage{mathrsfs}
\usepackage{caption}
\usepackage{subcaption}
\usepackage[titles,subfigure]{tocloft}
\usepackage{layouts}
\usepackage{hyperref}
\usepackage[normalem]{ulem}
\usepackage{authblk}
\usepackage{bm}
\usepackage{hyperref}
\DeclareMathOperator*{\argmax}{arg\,max}
\newcommand{\simiid}{\stackrel{\textnormal{i.i.d.}}{\sim}}
\usepackage[ruled,procnumbered]{algorithm2e}
\usepackage{amsthm}
\usepackage{lscape}
\allowdisplaybreaks

\newtheorem{thm}{Theorem}
\newtheorem{defi}{Definition}

\newtheorem{lemma}{Lemma}

\newtheorem{model}{Setting}

\newtheorem{conjecture}{Conjecture}

\theoremstyle{remark}
\newtheorem{remark}{Remark}

\newcommand{\rr}{\mathbb R}

\newcommand{\p}{\mathbb P}
\newcommand{\ind}{\mathbf{1}}
\newcommand{\di}[1]{\mathop{d#1}}
\newcommand{\Xno}[1]{X_{\text{-}{#1}}}

\newcommand{\ndata}[1]{\mathbf{#1}}
\DeclareMathOperator{\Var}{Var}
\DeclareMathOperator{\Cov}{Cov}
\DeclareMathOperator{\Corr}{Corr}

\DeclareMathOperator{\Unif}{Unif}

\DeclareMathOperator{\Bern}{Bern}

\DeclareMathOperator{\trace}{tr}
\DeclareMathOperator{\TV}{TV}
\DeclareMathOperator*{\argmin}{argmin}
\newcommand{\e}{\mathbb E}

\newcommand{\ci}{\mathrel{\perp\mspace{-10mu}\perp}}

\newcommand{\eqd}{\stackrel{d}{=}}
\newcommand{\cid}{\stackrel{d}{\to}}

\newcommand{\cip}{\stackrel{p}{\to}}
\newcommand{\cas}{\stackrel{\textnormal{a.s.}}{\to}}

\usepackage{etoolbox}

\makeatletter
\AtBeginEnvironment{procedure}{\let\c@algocf\c@procedure}
\makeatother

\title{A Power Analysis of the\\ Conditional Randomization Test and Knockoffs}
\date{}
\author{Wenshuo Wang}
\author{Lucas Janson}
\affil{Department of Statistics, Harvard University}

\begin{document}

\maketitle

\begin{abstract}
In many scientific problems, researchers try to relate a response variable $Y$ to a set of potential explanatory variables $X = (X_1,\dots,X_p)$, and start by trying to identify variables that contribute to this relationship. In statistical terms, this goal can be posed as trying to identify $X_j$'s upon which $Y$ is conditionally dependent. 
Sometimes it is of value to simultaneously test for each $j$, which is more commonly known as variable selection. The conditional randomization test (CRT) and model-X knockoffs are two recently proposed methods that respectively perform conditional independence testing and variable selection by, for each $X_j$, computing any test statistic on the data and assessing that test statistic's significance by comparing it to test statistics computed on synthetic variables generated using knowledge of $X$'s distribution. 
Our main contribution is to analyze their power in a high-dimensional linear model where the ratio of the dimension $p$ and the sample size $n$ converge to a positive constant. We give explicit expressions of the asymptotic power of the CRT, variable selection with CRT $p$-values, and model-X knockoffs, each with a test statistic based on either the marginal covariance, the least squares coefficient, or the lasso. One useful application of our analysis is the direct theoretical comparison of the asymptotic powers of variable selection with CRT $p$-values and model-X knockoffs; in the instances with independent covariates that we consider, the CRT provably dominates knockoffs. We also analyze the power gain from using unlabeled data in the CRT when limited knowledge of $X$'s distribution is available, and the power of the CRT when samples are collected retrospectively.
\end{abstract}

{\small {\bf Keywords.} Conditional randomization testing, knockoffs, Benjamini--Hochberg, model-X, retrospective sampling, approximate message passing. }

\section{Introduction}

\subsection{The conditional randomization test and model-X knockoffs}
Analyzing the statistical relationship between random variables lies at the heart of many practical problems. For example, in clinical trials, doctors aim to determine whether a certain treatment has any effect on the patient's health. In genome-wide association studies (GWAS), researchers seek to understand which genes directly contribute to a trait of interest. Many such modern statistical problems are set up in high dimensions, partly because scientific advances have allowed us to easily collect a large number of covariates. 

\citet{candes2018panning} proposed two methods to identify important variables: the conditional randomization test (CRT) for testing conditional independence, and model-X knockoffs, or simply knockoffs, for variable selection while controlling the false discovery rate (FDR).
Coined in \citet{candes2018panning}, ``model-X'' refers to a framework where inference is conducted by making as many assumptions on the distribution of $X$ (covariates) as possible and as few assumptions on the conditional distribution of $Y$ (outcome) given $X$ as possible. While the CRT and knockoffs have gained the interest of many researchers, there has been limited theoretical work on their power. 

\subsection{Our contribution}
This article analyzes the CRT for single hypothesis testing, its generalization for multiple hypothesis testing, and knockoffs for variable selection \citep{candes2018panning}. 
We mainly study the asymptotic performance of the CRT and knockoffs with different choices of popular test statistics. 
Power analysis is beneficial for various reasons. First, it is useful for determining how many samples one would like to collect beforehand in order to achieve a certain target power in a given experiment. Second, 
it allows direct comparison between methods in infinitely many data-generating distributions without the need for any simulations.
For CRT and knockoffs, power analysis is particularly important for two reasons: (a) both methods act as wrappers around a chosen test statistic, and our theory can be used to choose among test statistics according to their power for a given data-generating distribution; (b) both methods provide particularly easy ways to leverage unlabeled data and also apply directly to retrospectively sampled data, though the impact of the unlabeled data or the retrospective sampling scheme on power has not been studied.

Our results are in the setting of high-dimensional linear regression, where the ratio of the numbers of observations and covariates converges to a positive constant and the covariates follow a multivariate Gaussian distribution $\mathcal N(0,\Sigma)$. Our main results are:
\begin{enumerate}
    \item We give explicit expressions for the asymptotic power of the CRT when the 
    test statistic is derived from the marginal covariance, ordinary least square (OLS) coefficient, or the lasso \citep{tibshirani96regression}. We also prove bounds on and conjecture an exact expression for the asymptotic power of the CRT with marginal covariance test statistic when finite unlabeled samples are incorporated.
    
    \item When  $\Sigma=I$, i.e., the covariates are independent and identically distributed (i.i.d.) Gaussian random variables, we characterize the asymptotic power of the Benjamini--Hochberg (BH) procedure and the adaptive $p$-value thresholding (AdaPT) \citep{lei2018adapt} procedure applied to the CRT $p$-values given by the aforementioned three statistics. In the same setting, we also show all of these procedures asymptotically control the FDR at the nominal level.
    \item When $\Sigma=I$, we derive the asymptotic power of knockoffs using statistics derived from the marginal covariance, the OLS coefficient, or the lasso. We show that knockoffs is asymptotically equivalent to applying the AdaPT procedure to CDF-transformed knockoff statistics, and thus we can directly compare knockoffs' power with our earlier results on the CRT. 
    \item We extend the above three contributions with the marginal covariance test statistic to retrospectively sampled data, showing that the resulting effective signal strength is an explicit function of the marginal second moment of the retrospectively sampled $Y$.
\end{enumerate}

We demonstrate that our asymptotic power expressions are quite accurate in finite samples, and the CRT and knockoffs can achieve power close to oracle methods.



\subsection{Related work}
Since \citet{candes2018panning} introduced the CRT and model-X knockoffs, subsequent works have studied their robustness \citep{robust-knockoffs,berrett2019conditional,huang2019relaxing,barber2019construction}, computation \citep{tansey2018holdout,SB-EC-LJ-WW:2020,liu2020constructing}, and application to, e.g., neural networks \citep{deeppink2018}, time series \citep{ipad2018}, graphical models \citep{li2019nodewise}, and biology \citep{sesia2018gene,katsevich2019multilayer,sesia2020multi,bates2020causal,sesia2020controlling,chia2020interpretable,Katsevich2020.08.13.250092}.
Regarding the power of these methods,
\citet{weinstein2017power} analyzed the power of a knockoffs-inspired procedure that is only valid when all the covariates are i.i.d.; our work studies (in addition to the CRT) the original model-X knockoffs procedure, which is valid for any covariate distribution although we study it in a setting with i.i.d. covariates. \citet{liu2019power} provided a condition under which knockoffs' power goes to $1$ and FDR goes to $0$; our work exactly characterizes the asymptotic power when it is non-trivial (strictly between $0$ and $1$). \citet{katsevich2020theoretical} studied the CRT under low-dimensional asymptotics, while our work focuses on the high-dimensional regime, although we include a short note on the power of the CRT in low dimensions in Appendix~\ref{sec:scalar-CRT}. During the preparation of our manuscript, \citet{weinstein2020power} independently 
quantified the asymptotic power of knockoffs with the lasso coefficient difference statistic, a result which is quite similar to our Theorem~\ref{theorem:lasso-fdr-power-simple}, though that paper does not study the CRT or other statistics for knockoffs as we do.

There have been a number of other works on the asymptotic power of other methods that test for covariate importance \citep{zhu2018significance,chernozhukov2018double,javanmard2018debiasing}. These methods are fundamentally different from the CRT and knockoffs, but we will compare their results with our own in Section~\ref{sec:crt-power-comp}.

\subsection{Notation}

Bold letters are used for matrices or vectors containing i.i.d. observations. For a set $S$, $|S|$ denotes the number of elements in $S$. The cumulative distribution function (CDF) of the Gaussian distribution $\mathcal N(0,1)$ is denoted by $\Phi$---for $\alpha\in(0,1)$, $z_\alpha$ denotes the $\alpha$-quantile of $\mathcal N(0,1)$, i.e., $\Phi(z_{\alpha})=\alpha$. For random variables or vectors $W_1$ and $W_2$, $\mathcal L(W_1)$ means the distribution of $W_1$ and $\mathcal L(W_1\,|\,W_2)$ means the conditional distribution of $W_1$ given $W_2$. 

\subsection{Outline of the article}
In Section~\ref{sec:single}, we analyze the CRT's power for single hypothesis testing (conditional independence testing), including the case where we leverage unlabeled samples. In Section~\ref{sec:multiple}, we analyze the power of the CRT and knockoffs for multiple testing (variable selection). In Section~\ref{sec:retro}, we study the effect of retrospective sampling. Section~\ref{sec:discussion} supports our theoretical results with simulations. Finally, we conclude with a discussion of some questions raised by our work in Section~\ref{sec:discussion-new}.

\section{Power analysis of conditional independence testing}
\label{sec:single}
In this section, we study the power of the CRT for testing a single hypothesis of conditional independence. 

\subsection{The conditional randomization test}
\label{sec:CRT-introduction}
We begin with a review of the CRT introduced in \citet{candes2018panning}. Consider the generic problem of testing $H_{0}^{(j)}:X_j\ci Y\mid \Xno{j}$ in a regression setting where we have $n$ i.i.d. observations. 
To lighten notation, we label $X_j$ as simply $X$ and $\Xno{j}$ as $Z$, and the data matrix is therefore denoted by $[\mathbf{X}, \ndata{Z},\mathbf{Y}]$, where $\ndata{X}\in\rr^n$ is the covariate vector of interest, $\ndata{Z}\in\rr^{n\times(p-1)}$ is the matrix of confounding variables, and $\ndata{Y}\in\rr^n$ is the response vector.
Suppose we have a test statistic function $T$ of $(\ndata{X},\ndata{Z},\bf{Y})$ that intuitively measures the importance of variable $X$ (e.g., $T$ could be the absolute value of the coefficient for $\ndata{X}$ fitted by the lasso). To construct a test, we need to find a cutoff for the test statistic $T(\ndata{X},\ndata{Z},\bf{Y})$ such that we 
can guarantee $T(\ndata{X},\ndata{Z},\bf{Y})$ only falls above that cutoff with probability at most the nominal level $\alpha$ under $H_0^{(j)}$.
This requires some knowledge of its distribution under the null. \citet{candes2018panning} suggested the following: if we know $\mathcal L(X\,|\,Z)$, then let $\tilde{\ndata{X}}\mid\ndata{Z},\ndata{Y}\sim\mathcal L(\ndata{X}\,|\,\ndata{Z})$ and we will have
\begin{equation*}
T(\ndata{X},\ndata{Z},\mathbf{Y})\eqd T(\tilde{\mathbf{X}},\ndata{Z},\mathbf Y)\mid \ndata{Z},\mathbf{Y}\text{ under the null}.
\end{equation*}
Thus, we can obtain a cutoff using the known distribution $\mathcal L(T(\tilde{\mathbf{X}},\ndata{Z},\mathbf Y)\,|\,\ndata{Z},\mathbf{Y})$. Although such a cutoff can only be computed analytically in special cases \citep{liu2020constructing}, an empirical one can be obtained by repeatedly resampling $\tilde{\ndata{X}}$ and recomputing $T(\tilde{\mathbf{X}},\ndata{Z},\mathbf Y)$. The CRT with an empirical cutoff contains a finite-sample correction to make the test exact, but it converges to the test using the exact quantile if the number of resamples $M_n$ goes to infinity. The cases we consider in this article all have an analytical cutoff available and this is the CRT we study, but the same results would still hold as long as $M_n\to\infty$. It is worthwhile to emphasize that \emph{any} test statistic function $T$ can be used in the CRT.

We have assumed that we know $\mathcal L(X\,|\,Z)$ exactly, and will make this assumption in Section~\ref{sec:CRT}; in Section~\ref{sec:CRT-conditional}, we will study the power when this assumption is relaxed by conditioning and leveraging unlabeled data \citep{huang2019relaxing}, and in Section~\ref{sec:retro}, we will discuss how the power changes with retrospective sampling \citep{barber2019construction}.

\subsection{CRT in high-dimensional linear regression}
\label{sec:CRT}



We begin by analyzing the power of the CRT using several different statistics in the high-dimensional linear regression setting formally defined as follows in Setting~\ref{model:moderate-dim-lr}.
\begin{model}[High-dimensional linear model]
Consider the linear regression model
\begin{equation*}
\ndata{Y}=\ndata{X}\beta+\ndata Z\theta+\varepsilon, \qquad \varepsilon\sim\mathcal N(0,\sigma^2I),
\end{equation*}
where $\ndata{X}\in\rr^n$, $\ndata Z\in\rr^{n\times(p-1)}$, and for each row, they satisfy
\begin{equation*}
\ndata{Z}_{i\cdot}\simiid\mathcal N(0,\Sigma_Z),\qquad 
\ndata{X}_i\simiid\mathcal N(\ndata{Z}_{i\cdot}^\top\xi,1),
\qquad (\ndata{X},\ndata{Z})\ci\varepsilon.
\end{equation*}
This setting assumes the above model under the following high-dimensional asymptotics:
\begin{equation*}
\lim_{n\to\infty}p/n\to\kappa\in(0,\infty),\qquad\lim_{n\to\infty}\theta^\top\Sigma_Z\theta\to v_Z^2,\qquad\limsup_{n\to\infty}\xi^\top\Sigma_Z\xi<\infty,
\end{equation*}
and $\sigma^2$ and $\sqrt n\beta=h>0$ stay constant.
\label{model:moderate-dim-lr}
\end{model}
We emphasize again that the assumptions in Setting~\ref{model:moderate-dim-lr} are for the study of power and are not needed for the validity of the CRT.
Here, $\theta^\top\Sigma_Z\theta$ can be interpreted as the part of $Y$'s variance contributed by $Z$, as $\theta^\top\Sigma_Z\theta=\Var(\e[Y\,|\,Z])$; similarly, $\xi^\top\Sigma_Z\xi$ can be interpreted as the part of $X$'s variance contributed by $Z$. For instance, the assumptions on $\theta^\top\Sigma_Z\theta$ and $\xi^\top\Sigma_Z\xi$ hold if $\Sigma_Z=I$ and the components of $\theta$ and $\xi$ are $n^{-1/2}$-normalized i.i.d. draws from a distribution with finite second moment. More concretely, if $\sqrt n\theta_j\simiid B_0$ and $\Sigma_Z=I$, then $\theta^\top\Sigma_Z\theta\to\kappa\e[B_0^2]$ almost surely, which corresponds to the setting we will consider in Section~\ref{sec:multiple}.

The CRT tests $H_0:X\ci Y\mid Z$, which, under Setting~\ref{model:moderate-dim-lr}, is equivalent to $H_0:\beta=0$, and we are interested in the power under local alternatives $H_1:\beta=h/\sqrt n$ for a fixed $h>0$, which is the regime where the power has a non-trivial limit strictly between $0$ and $1$. In this section, the asymptotic power means the limit of the unconditional power of the test under $\beta=h/\sqrt n$. We consider three different test statistics for the CRT 
and for each one we will show there exists a scalar $\mu$ (which we will give an expression for) such that the asymptotic power is equal to that of a $z$-test with standardized effect size $\mu$ as in Definition~\ref{definition:z-test}.
\begin{defi}
The CRT with test statistic $T$ under a given asymptotic regime is said to have asymptotic power equal to that of a $z$-test with standardized effect size $\mu$, if the level-$\alpha$ one-sided CRT (reject for large values of $T$) has asymptotic power $\Phi(\mu-z_{1-\alpha})$ and the level-$\alpha$ two-sided CRT (reject for large values of $|T|$) has asymptotic power $\Phi(\mu-z_{1-\alpha/2})+\Phi(-\mu-z_{1-\alpha/2})$.
\label{definition:z-test}
\end{defi}



\subsubsection{Marginal covariance}
Consider testing using the statistic $T_\text{MC}(\ndata{X},\ndata{Y},\ndata{Z})=n^{-1}\ndata{Y}^\top\ndata{X}$, which is an unbiased and consistent estimator of $\Cov(X,Y)$. 
Although it may seem na\"ive to consider a marginal test statistic that does not involve $Z$, it is actually a popular choice in many high-dimensional applications such as genome-wide association studies \citep{wu2010screen} and microbiome studies \citep{mcmurdie2014waste}. $T_\text{MC}$ is simple and intuitive and we will see it performs well when the confounding vector $Z$ does not contribute too much variance to $Y$.

\begin{thm}
In Setting~\ref{model:moderate-dim-lr}, the 
CRT with $T_\textnormal{MC}$ has asymptotic power
equal to that of a $z$-test with standardized effect size
\[ \frac{h}{\sqrt{\sigma^2+v_Z^2}}. \]
\label{theorem:prop-power-mc}
\end{thm}

We first note that the power increases as $h$ increases, which is intuitive because $h$ is the coefficient (dropping the normalization term $\sqrt n$) and the effective signal strength. Second, the dimensionality (or equivalently, $\kappa$) does not appear explicitly, which can be understood by noticing that $Z$ only plays a role through $Z^\top\theta$, which we can consider as part of the error, thus adding extra variance $\theta^\top\Sigma_Z\theta\to v_Z^2$. Then the asymptotic effective error variance is $\sigma^2+v_Z^2$, and when this number is large, the power is low.

\subsubsection{Ordinary least squares}
\label{sec:CRT-OLS}
There are many reasons why one might want to use the ordinary least squares (OLS) estimate $T_\text{OLS}=\hat\beta^\text{OLS}$ as the test statistic; for instance, it is the maximum likelihood estimator and the best linear unbiased estimator. In this section, we will assume $\kappa<1$ in Setting~\ref{model:moderate-dim-lr} so that $\hat\beta^\text{OLS}$ is well-defined. 
\begin{thm}
In Setting~\ref{model:moderate-dim-lr} with $\kappa<1$, the 
CRT with $T_\text{OLS}$ has asymptotic power
equal to that of a $z$-test with standardized effect size
\[ \frac{h}{\sigma}\sqrt{1-\kappa}. \]
\label{theorem:prop-model-X-CRT-OLS}
\end{thm}

We can see that the power decreases as $\sigma$ and $\kappa$ increase. Compared to using $T_\text{MC}$, the CRT with $T_\text{OLS}$ has higher power if $\kappa<v_Z^2/(\sigma^2+v_Z^2)$, and vice versa. In fact, as $\kappa$ approach $1$, OLS becomes ill-defined and the test becomes powerless. As another comparison, consider the canonical OLS test that takes $\hat\beta^\text{OLS}$ as the test statistic and conducts a test based on the conditional distribution $\hat\beta^\text{OLS}\mid \ndata{X},\ndata{Z}\sim\mathcal N\left(\beta,\sigma^2\left([\ndata{X},\ndata{Z}]^\top[\ndata{X},\ndata{Z}]\right)^{-1}_{11}\right)$. This canonical test turns out to have the same asymptotic power as the one for the CRT given in Theorem~\ref{theorem:prop-model-X-CRT-OLS} (see Appendix~\ref{sec:fixed-X-OLS} for derivation).



\subsubsection{The distilled lasso statistic}
\label{sec:distilled}
For high-dimensional regression, one might naturally look to the lasso to construct a test statistic. In this section, we consider the distilled lasso statistic, a test statistic proposed in \citet{liu2020constructing} derived from the lasso, which leverages the lasso for fitting a high-dimensional coefficient vector and has very similar power to, but is more computationally efficient than, using $\hat\beta^\text{lasso}$ as the test statistic. 
In our notation, the statistic can be defined as follows. We first regress $\ndata{Y}$ on $\ndata{Z}$ using the lasso with penalty parameter $\lambda$ to obtain $\hat\theta_\lambda$, i.e.,
\begin{equation*}
\hat\theta_\lambda=\argmin_\theta\frac{1}{2}\|\ndata{Y}-\ndata{Z}\theta\|_2^2+\sqrt n\lambda\|\theta\|_1.
\end{equation*}
The distilled lasso statistic is then defined as $T_\text{distilled}(\ndata{X},\ndata{Y},\ndata{Z})=n^{-1}(\ndata{Y}-\ndata{Z}\hat\theta_\lambda)^\top(\ndata{X}-\ndata{Z}\xi)$. Intuitively, it measures the covariance of $X$ and $Y$ after their dependence on $Z$ is removed, where $Y$'s dependence on $Z$ is estimated by the lasso.

To analyze this test statistic, we will leverage
the theory of approximate message passing (AMP), which
has been used to characterize the asymptotic distribution of the lasso coefficient vector \citep{bayati2011lasso}. This asymptotic distribution depends on two important parameters $\alpha_\lambda$ and $\tau_\lambda$
which are uniquely defined as the solution to a system of explicit fixed-point equations depending on $\lambda$,
$\sigma^2$, $\kappa$, and the asymptotic histogram of the true coefficients $\sqrt n\theta_j$'s.\footnote{Note that $\alpha_\lambda$ is unrelated to the level $\alpha$ of the statistical test.} 
In Appendix~\ref{sec:proofs}, we provide the fixed-point equations \eqref{equation:amp-fixed-point-equation}. Intuitively speaking, $\sqrt n\hat\theta_j$ is roughly distributed as $\eta(\sqrt n\theta_j+\tau_\lambda Z;\alpha_\lambda\tau_\lambda)$, where
\begin{equation}
\eta(x;y)=\left\{
\begin{aligned}
&x-y, &x>y,\\
&0, &|x|\le y,\\
&x+y, &x<-y,
\end{aligned}\right.
\label{equation:eta}
\end{equation}
$Z\sim\mathcal N(0,1)$, so that $\tau_\lambda$ plays the role of the level of the asymptotic estimation error and $\alpha_\lambda$ acts as a soft-thresholding parameter. AMP theory, and hence our use of it, relies on additional assumptions as stated in Theorem~\ref{theorem:prop-distilled-power}.




\begin{thm}
Under Setting~\ref{model:moderate-dim-lr} with $\Sigma_Z=I$ and $\xi=0$, if the empirical distribution of $(\sqrt n\theta_j)_{j=1}^{p-1}$ converges to a distribution represented by a random variable $B_0$ and $\|\sqrt n\theta\|_2^2/p\to\e[B_0^2]$,
then the 
CRT with the distilled lasso statistic with lasso parameter $\lambda$ has asymptotic power
equal to that of a $z$-test with standardized effect size
\[ \frac h{\tau_\lambda}. \]
\label{theorem:prop-distilled-power}
\end{thm}
We prove Theorem~\ref{theorem:prop-distilled-power} using results from \citet{bayati2011lasso}. The key step is to analyze the asymptotic correlation between the errors and the fitted residuals of the lasso regression, which has not been studied before. Theorem~\ref{theorem:prop-distilled-power} is clean in that it only depends on the model parameters through a scalar $\tau_\lambda$, making it helpful for guiding the choice of $\lambda$. 

\subsubsection{Comparison of CRT statistics}

In Figure~\ref{figure:tau}, we plot the relationship of the asymptotic power of the CRT with the distilled lasso statistic and $\lambda$ in different settings and compare with that of the CRT using marginal covariance and OLS. We can see that the dependence of the power on $\lambda$ is mild, and the distilled lasso statistic with a good $\lambda$ is always better than marginal covariance and OLS in the considered parameter settings. This is not a coincidence: the best distilled lasso statistic has at least the same power as the marginal covariance and the OLS coefficient. To see this, note that if $\lambda=\infty$, $\hat\theta_\lambda=0$ and $T_\text{distilled}=T_\text{MC}$; if $\lambda=0$, $\hat\theta_\lambda=\hat\theta^\text{OLS}=(\ndata{Z}^\top\ndata{Z})^{-1}\ndata{Z}^\top\ndata{Y}$, and $T_\text{distilled}$ is equal to the numerator of the expression of $T_\text{OLS}$ in Equation~\eqref{equation:ols-expression} in the proof of Theorem~\ref{theorem:prop-model-X-CRT-OLS}, the power of which can be even more easily proved to be the same as $T_\text{OLS}$ following the same proof.

\begin{figure}[H]
    \centering
\includegraphics[width = 0.9\textwidth]{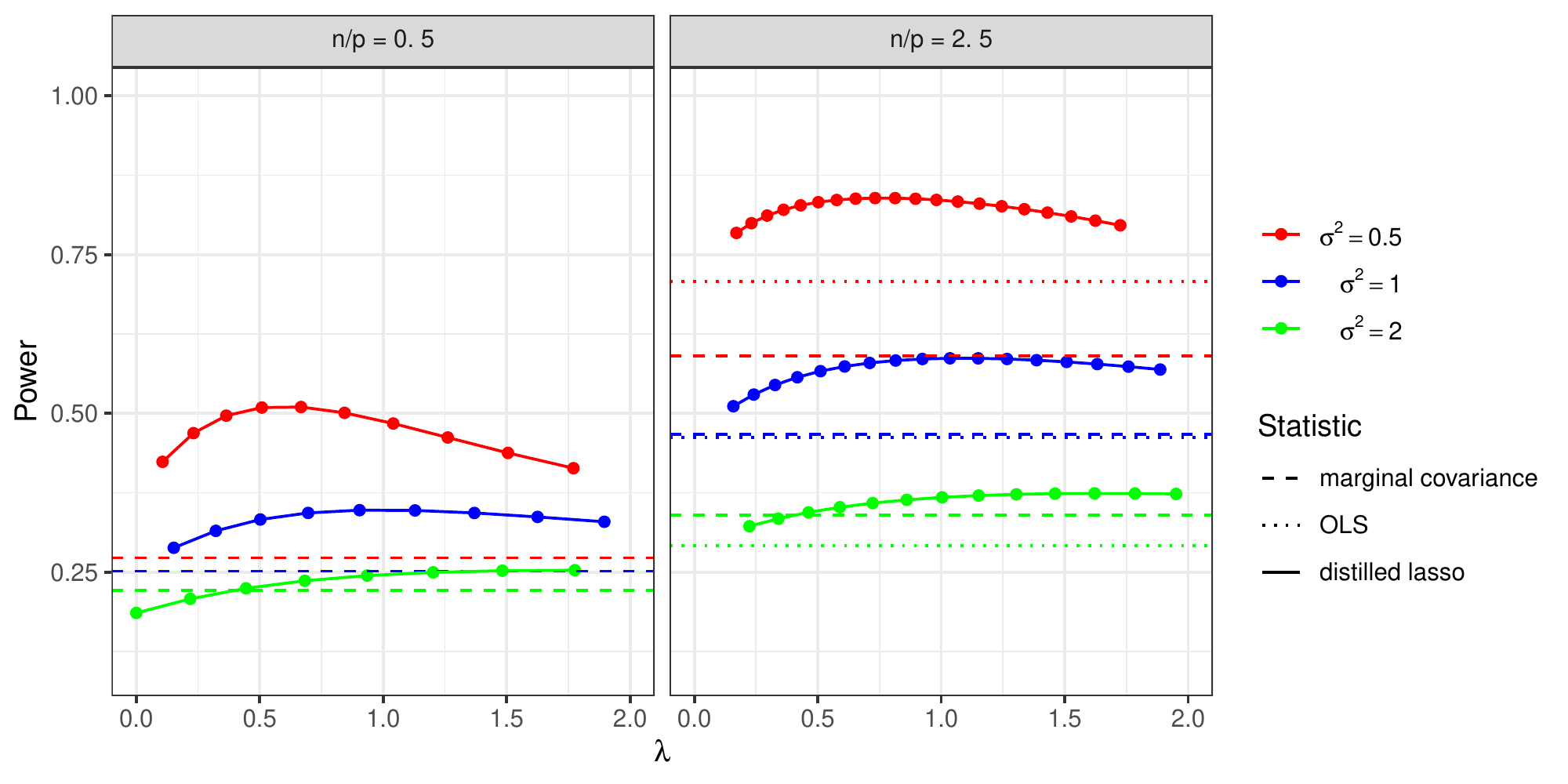}
    \caption{Dependence of the asymptotic power of the CRT with the distilled lasso statistic on $\lambda$, compared with that with the marginal covariance statistic and the OLS statistic; plots are in Setting~\ref{model:moderate-dim-lr} with $B_0\sim0.9\delta_0+0.1\delta_4$, $\Sigma_Z=I$ and $\xi=0$, so $v_Z^2=1.6\kappa$. Signal size is $\beta=2/\sqrt n$. OLS is not applicable for $n=0.5p$.}
    \label{figure:tau}
\end{figure}


\subsubsection{Comparison with other methods}
\label{sec:crt-power-comp}
It is possible to achieve the power $\Phi\left(h/\sigma-z_\alpha\right)$ if $\theta$ and $\xi$ can be estimated at a sufficiently fast rate, which dominates all three of our statistics (for the lasso, note that $\tau_\lambda\ge\sigma$). \citet{javanmard2018debiasing} achieved this rate (their Theorem~3.8) by assuming, among other conditions, $\mathcal L(X,Z)$ is known and $\theta$ has sparsity level $o(n/(\log p)^2)$ (which is not satisfied in our setting). \citet{chernozhukov2018double} also obtained this rate (their Theorem~4.1) by assuming $\hat\xi$ and $\hat\theta$ are consistent and $\|\hat\xi-\xi\|\|\hat\theta-\theta\|=o(n^{-1/2})$, while it is well-known that consistency in high-dimensional settings is usually impossible without strong assumptions such as sparsity (which we do not make).

\citet{zhu2018significance} assumed sub-Gaussianity and moment conditions to derive the same asymptotic power (their Theorem~7) as in our Theorem~\ref{theorem:prop-power-mc} for a different method they proposed and under different assumptions, except that there a two-sided test was considered. There are differences that are worth noting: (a) \citet{zhu2018significance} does not assume $\xi$ is known, but requires $\xi$ to have sparsity $o(\sqrt n/(\log(\max(p,n)))^{5/2})$ in order to estimate it fast enough and gives an asymptotically valid test, and here we assume we know $\xi$ so the test has exact size $\alpha$ for any finite $(n,p)$; (b) we make stronger assumptions on $\mathcal L(Y\,|\,X,Z)$, which is mainly to facilitate analysis for other more complex statistics; our Theorem~\ref{theorem:prop-power-mc} could easily be extended to only assume moment conditions on $\varepsilon$.

\subsection{Leveraging unlabeled data in the CRT}
\label{sec:CRT-conditional}

In Section~\ref{sec:CRT}, we assumed we knew $\mathcal L(X\,|\,Z)$ exactly, which can be understood as a case in which we have sufficient unlabeled data and/or domain knowledge that we effectively know this distribution exactly. In some practical cases, however, we do not know $\mathcal L(X\,|\,Z)$ exactly and would like to leverage finite unlabeled samples to learn more about it. To this end, we explore in this section a useful idea introduced in \citet{huang2019relaxing}, which only assumes a model on $\mathcal L(X\,|\,Z)$ and conditions on a sufficient statistic that uses both labeled and unlabeled data.

As a concrete example, we again consider Setting~\ref{model:moderate-dim-lr}, but with the following changes: $\xi$ is unknown; $\Var(X\,|\,Z)=1$ but is unknown to the CRT. Effectively, we have assumed an unknown Gaussian linear model for $\mathcal L(X\,|\,Z)$. 
In this case, we would not be able to sample from $\mathcal L(\ndata{X}\,|\,\ndata{Z})$ to get $\tilde{\ndata{X}}$ as normally required by the CRT. Exploiting Gaussianity, we can proceed by conditioning on a sufficient statistic as follows. Let $T(\ndata{X},\ndata{Y},\ndata{Z})$ be the test statistic and $S(\ndata{X},\ndata{Z})$ be a sufficient statistic (e.g., $S(\ndata{X},\ndata{Z})=\ndata{Z}^\top\ndata{X}$ is sufficient in Setting~\ref{model:moderate-dim-lr}) for the unknown parameter in $\mathcal L(\ndata{X}\,|\,\ndata{Z})$. By the sufficiency of $S$, $\mathcal L(\ndata{X}\,|\,\ndata{Z},S(\ndata{X},\ndata{Z}))$ does not depend on the unknown $\xi$ or $\Var(X\,|\,Z)$, so we therefore know $\mathcal L(T(\ndata{X},\ndata{Y},\ndata{Z})\,|\,\ndata{Z},S(\ndata{X},\ndata{Z}),\ndata{Y})$ under the null. Thus, we can obtain an analytical or empirical cutoff in the same way as the original unconditional CRT.

Although weakening the assumed knowledge of $\mathcal L(X\,|\,Z)$ by moving from an unconditional test to a conditional one may reduce the power, this reduction can be mitigated by incorporating unlabeled data into the sufficient statistic.
Let the unlabeled samples be denoted by $(X_{n+i},Z_{n+i})_{i=1}^m$, i.e., they are recorded without the response $Y_{n+i}$, where we assume $m/p$ goes to a positive constant. 
Let $n_*=n+m$ and $\ndata{X}_*$ and $\ndata{Z}_*$ be the $n_*\times 1$ and $n_*\times(p-1)$ data matrices containing all $X$ and $Z$ samples, respectively, with the first $n$ rows being labeled and corresponding to $\ndata{Y}$. Similarly to the case without unlabeled data, let $T(\ndata{X}_*,\ndata{Y},\ndata{Z}_*)$ be the test statistic and $S(\ndata{X}_*,\ndata{Z}_*)$ be a sufficient statistic for the unknown parameter in $\mathcal L(\ndata{X}_*\,|\,\ndata{Z}_*)$, so that we know $\mathcal L(T(\ndata{X}_*,\ndata{Y},\ndata{Z}_*)\,|\,\ndata{Z}_*,S(\ndata{X}_*,\ndata{Z}_*),\ndata{Y})$ under the null. We emphasize that this idea also applies to non-Gaussian cases as long as a sufficient statistic exists. We can also see that $\mathcal L(\ndata{X}_*\,|\,\ndata{Z}_*,S(\ndata{X}_*,\ndata{Z}_*))$ does not change even if the labeled samples are collected retrospectively (i.e., based on the response variable $Y$; see, for example, \citet{barber2019construction}), as under the null, $\ndata{X}_*\ci\ndata{Y}\mid\ndata{Z}_*$. On the other hand, the power \emph{will} change, though power analysis with both retrospective sampling and unlabeled data is beyond the scope of this article (while we do provide an analysis in Section~\ref{sec:retro} in the case of known $\mathcal L(X\,|\,Z)$, i.e., infinite unlabeled data); we focus on the case when the labeled samples are collected independent of the responses. It turns out that this procedure admits a quite substantial simplification for the Gaussian distribution. For instance, if $T$ is chosen to be the marginal covariance $T_\text{MC}=n^{-1}\ndata{Y}^\top\ndata{X}$, it simplifies to a statistic that could be seen as a generalization of the OLS statistic, which enables the analysis of its asymptotic power. We defer the details to Appendix~\ref{sec:conditional-crt}, where we also discuss why it might not be beneficial to consider the original OLS statistic in this setting. We present here upper and lower bounds of the asymptotic power together with a conjecture for its exact value.

\begin{thm}
In Setting~\ref{model:moderate-dim-lr} with $\xi$ and $\Var(X\,|\,Z)$ unknown but fixed to be $1$, if there are $m$ additional data points $(X_i,Z_i)_{i=n+1}^{n+m}$, $n_*=n+m$, $n/n_*\to\kappa_*$ and $\kappa\kappa_*<1$, then the conditional CRT with statistic $T_\textnormal{MC}$ has asymptotic power 
lower-bounded (the $\liminf$ is lower-bounded) by that of a $z$-test with standardized effect size
\[ \frac{h\sqrt{1-\kappa\kappa_*}}{\sqrt{\sigma^2+v_Z^2\frac{1}{{1-\kappa\kappa_*}}}} \]
and upper-bounded (the $\limsup$ is upper-bounded) by that of a $z$-test with standardized effect size
\[ \frac{h\sqrt{1-\kappa\kappa_*}}{\sqrt{\sigma^2+v_Z^2\max\left(0,\frac{1-\frac{(1+\sqrt{1/\kappa})^2}{(1-\sqrt{\kappa\kappa_*})^2}\kappa\kappa_*}{1-\kappa\kappa_*}\right)}}. \]
\label{theorem:prop-model-X-unlabeled}
\end{thm}

\begin{conjecture}
In Setting~\ref{model:moderate-dim-lr}, if there are $m$ additional data points $(X_i,Z_i)_{i=n+1}^{n+m}$, $n_*=n+m$, $n/n_*\to\kappa_*$ and $\kappa\kappa_*<1$, then the conditional CRT with statistic $T_\textnormal{MC}$ has asymptotic power
equal to that of a $z$-test with standardized effect size
\[ \frac{h\sqrt{1-\kappa\kappa_*}}{\sqrt{\sigma^2+v_Z^2(1-\kappa_*)}}. \]
\label{conjecture:model-X-unlabeled}
\end{conjecture}
See Figure~\ref{figure:model-X-conjecture} in Section~\ref{sec:discussion} for a numerical validation. We discuss how we arrive at this conjecture in Appendix~\ref{sec:proofs}.
Note that the two bounds in Theorem~\ref{theorem:prop-model-X-unlabeled} match the conjecture when $\kappa_*\to0$.

Trivially when $\mathcal{L}(X\,|\,Z)$ is unknown, unlabeled data helps to run a test if $\kappa>1$, since otherwise no non-trivial test can be run because the sufficient statistic uniquely determines $\ndata{X}$'s exact value, making $\mathcal L(\ndata{X}\,|\,\ndata{Z},S(\ndata{X},\ndata{Z}))$ degenerate so that the only valid tests have power equal to their size under any alternative. 
When $\kappa<1$, assuming Conjecture~\ref{conjecture:model-X-unlabeled} holds, we see that unlabeled data can boost the power compared to using only labeled data ($\kappa_*=1$) if $\kappa>{v_Z^2}/{(\sigma^2+v_Z^2)}$, i.e., if $p$ is close to $n$ or if the nuisance variables $Z$ contribute little variance to $Y$. This condition coincides with the condition under which the unconditional CRT with marginal covariance has higher power than with OLS. Another interesting takeaway is that if we keep $\kappa\kappa_*$ fixed and let $\kappa_*\to0$, the asymptotic power is equal to that of a $z$-test with standardized effect size
\[ \frac{h\sqrt{1-\kappa\kappa_*}}{\sqrt{\sigma^2+v_Z^2}}. \]
This can be interpreted as a setting where $p/n\to\infty$, but the number of unlabeled samples $n_*$ scales with $p$ as $p/n_*$ goes to a non-zero constant $\kappa\kappa_*$.

\section{Power analysis of variable selection}
\label{sec:multiple}
In this section, we consider variable selection and return to our original notation $X_j$ and $\Xno{j}$ instead of $X$ and $Z$ (analogously for their bold counterparts), which were used in Section~\ref{sec:single} in their place while $j$ was fixed. More specifically, suppose we have a data matrix $[\mathbf X, \mathbf Y]$, where each row is an i.i.d. draw from a distribution $F_{X,Y}$, where $X$ is a $p$-dimensional random vector and $Y$ is a random variable. We can define variable selection as simultaneously testing the null hypotheses $H_{0}^{(1)},\dots,H_{0}^{(p)}$, where $H_{0}^{(j)}$ is $X_j\ci Y\mid X_{\text{-}j}$. In this section, the power means the expectation of the ratio between the number of true discoveries and the number of non-null covariates.

To enable theoretical analysis, we study the linear regression setting with independent Gaussian covariates as given in Setting~\ref{model:lr-iid}, where $H_{0}^{(j)}$ reduces to $\beta_j=0$. 

\begin{model}[High-dimensional linear model with independent covariates]
Consider the linear regression model
\begin{equation*}
\ndata Y=\ndata X\beta+\varepsilon, \qquad \varepsilon\sim\mathcal N(0,\sigma^2I),
\end{equation*}
where $\ndata X\in\rr^{n\times p}$ is a random matrix,
\begin{equation*}
X_{ij}\simiid\mathcal N(0,1),\qquad\ndata{X}\ci\varepsilon.
\end{equation*}
This setting assumes the above model under the following high-dimensional asymptotics:
\[ \lim_{n\rightarrow\infty} p/n = \kappa\in(0,\infty), \qquad \sqrt n\beta_j\simiid \gamma\delta_0+(1-\gamma)\pi_1, \]
where $\gamma$ and $\pi_1$ are fixed, $\pi_1$ has bounded support and puts no mass at $0$, and $\beta\ci (\ndata{X},\varepsilon)$.
\label{model:lr-iid}
\end{model}
In the future, we will use $B_0$ to represent a random variable following $\gamma\delta_0+(1-\gamma)\pi_1$. Setting~\ref{model:lr-iid} is a slight modification of Setting~\ref{model:moderate-dim-lr} that makes all $X_j$'s exchangeable. The Gaussian assumption makes theoretical derivation easier, and allows for the use of results on the lasso obtained by AMP theory. While the model is not believed to be appropriate if the covariates are too dependent on each other, in many applications the covariates are only slightly correlated. Hence, although a simple setting, Setting~\ref{model:lr-iid} is expected to be of value and can still guide statistic choice in many applications. 

We analyze two types of procedures that control the FDR or asymptotic FDR in this setting.
\begin{enumerate}
    \item BH and AdaPT applied to $p$-values obtained by the CRT. To the best of our knowledge, these are the first results on the validity or power of BH and AdaPT applied to CRT $p$-values.
    \item The model-X knockoff filter \citep{candes2018panning}.
\end{enumerate}

\subsection{Variable selection with the CRT}
\label{sec:bh-p-val}


A natural way of generalizing the conditional independence tests of Section~\ref{sec:single} to variable selection is to take the $p$-values from the CRT and plug them into a multiple testing procedure. Here, then, we consider the BH procedure \citep{benjamini1995controlling} and the AdaPT procedure \citep{lei2018adapt} for controlling the FDR, defined as
\begin{equation*}
\mathrm{FDR} = \e\left[\mathrm{FDP}\right],\qquad \mathrm{FDP}=\frac{|\hat S\cap S_0|}{\max(|\hat S|,1)},
\end{equation*}
where FDP stands for false discovery proportion, $S_0$ is the set of null variables and $\hat S$ is the set of selected variables. When we refer to AdaPT, we mean the intercept-only AdaPT procedure, (i.e., AdaPT without side information), which rejects all $p$-values below
\begin{equation*}
\max\{t\in[0,1]:\frac{1+\#\{j:p_j\ge1-t\}}{\#\{j:p_j\le t\}}\le q\},
\end{equation*}
with $q$ being the target FDR level. BH is the most used multiple testing procedure for controlling the FDR, and studying AdaPT allows us to directly compare variable selection using the CRT with knockoffs due to an asymptotic equivalence between knockoffs and a certain application of AdaPT, which we will explain in Section~\ref{sec:knockoffs-mc-ols}. 
It is known that BH and AdaPT control the FDR at the nominal level when all $p$-values are independent and the null $p$-values follow the standard uniform distribution on $[0,1]$ \citep{benjamini1995controlling,lei2018adapt}. However, this assumption does not hold for the CRT $p$-values as they are in general not independent. They are also in general only super-uniform under the null, but for all of the test statistics and settings considered in this paper the CRT's $p$-values are indeed exactly uniform under the null.

A key result is that under certain conditions, BH and AdaPT applied to the CRT $p$-values have the same asymptotic FDR and power as if the $p$-values were actually independent. Due to the cumbersome notation required, a formal presentation is deferred to Theorem~\ref{theorem:BH-AdaPT-CRT} in Appendix~\ref{sec:proofs}, where we give conditions on input $p$-values such that BH and AdaPT perform asymptotically as if the input $p$-values were independent.\footnote{Theorem~\ref{theorem:BH-AdaPT-CRT} represents a variation on results of \citet{ferreira2006benjamini} but with a different proof catered to our specific setting.} Here, we only show the following Theorem~\ref{theorem:coro-BH-CRT} that applies Theorem~\ref{theorem:BH-AdaPT-CRT} to the CRT with the three statistics considered in Section~\ref{sec:single}. In order to state Theorem~\ref{theorem:coro-BH-CRT}, we need Definition~\ref{definition:effect-pi}, which allows us to concisely and intuitively characterize the asymptotic power expressions derived from Theorem~\ref{theorem:BH-AdaPT-CRT}. We note that in addition to characterizing the power, these two theorems are the first that we know of to prove asymptotic FDR control of multiple testing with CRT $p$-values.

\begin{defi}
Let $\mathcal P$ be a multiple testing procedure that takes a set of $p$-values as input, e.g., the BH procedure at level $q$. Let $\mathcal P(\pi_\mu)$ be the procedure that applies $\mathcal P$ to $p$-values $\textnormal{pval}_{1}, \dots, \textnormal{pval}_p$ in the following independent normal means model:
for $j=1,2,\dots,p$,
\begin{equation*}
\mu_j\simiid\pi_{\mu},\qquad\varepsilon_j\simiid\mathcal N(0,1),\qquad \textnormal{pval}_j=1-\Phi(\mu_j+\varepsilon_j)\;\;\; \textnormal{(respectively, }\textnormal{pval}_j=2(1-\Phi(|\mu_j+\varepsilon_j|))\textnormal{)}.
\end{equation*}
We say a variable selection procedure has one-sided (respectively, two-sided) effective $\pi_\mu$ with respect to $\mathcal P$ if, as $p\to\infty$,
\begin{enumerate}
    
    \item[(a)] the realized power (i.e., the proportion of rejected non-nulls) of this variable selection procedure 
    converges in probability to the same constant that the realized power of $\mathcal P(\pi_\mu)$ converges in probability to; and
    \item[(b)] when the asymptotic realized power is positive, the FDP of this variable selection procedure 
    converges in probability to the same constant that the FDP of $\mathcal P(\pi_\mu)$ converges in probability to.
\end{enumerate}

\label{definition:effect-pi}
\end{defi}
\begin{thm}
In Setting~\ref{model:lr-iid}, for Lebesgue-almost-every $q\in(0,1)$, BH or AdaPT at level $q$ using CRT $p$-values based on the statistics in Section~\ref{sec:CRT} (respectively, their absolute values) have the following one-sided (respectively, two-sided) effective $\pi_\mu$'s with respect to BH or AdaPT at level $q$:
\begin{enumerate}
    \item For the marginal covariance statistic, the effective $\pi_\mu$ is the distribution of $\frac{1}{\sqrt{\sigma^2+\kappa\e[B_0^2]}}{B_0}$.
    \item For the OLS statistic, assuming $\kappa<1$, the effective $\pi_\mu$ is the distribution of $\frac{\sqrt{1-\kappa}}\sigma B_0$.
    \item For the distilled lasso statistic, the effective $\pi_\mu$ is the distribution of $\frac1{\tau_\lambda}B_0$.
\end{enumerate}
\label{theorem:coro-BH-CRT}
\end{thm}
The effective $\pi_\mu$'s in Theorem~\ref{theorem:coro-BH-CRT} follow from Theorems~\ref{theorem:prop-power-mc}, \ref{theorem:prop-model-X-CRT-OLS} and \ref{theorem:prop-distilled-power}. The key component of the proof of Theorem~\ref{theorem:coro-BH-CRT} is jointly analyzing the test statistics for two different covariates showing that its two coordinates are asymptotically independent. This is particularly non-trivial for the distilled-lasso statistic, where we employ a leave-one-out approach. As one would expect, these procedures have higher power if the respective CRT with the same statistic has higher power. For example, for the marginal covariance statistic, as $\sigma^2+\kappa\e[B_0^2]$ gets smaller, the null and non-null distributions of the $p$-values are more separated and higher power would be obtained. Naturally, this is the same condition under which the CRT with the marginal covariance statistic has higher power, once we realize that $v_Z^2=\kappa\e[B_0^2]$ (see the text immediately after Setting~\ref{model:moderate-dim-lr}). Although we choose BH and AdaPT as representatives, we note that the same proof techniques could be used to establish analogous results for other procedures such as Storey's BH \citep{storey2004strong} that use the empirical distribution of the $p$-values in a certain way. 

\subsection{Model-X knockoffs}
\label{sec:knockoffs}

\subsubsection{Review of knockoffs}
\label{sec:model-x-knockoffs}
We now turn to the analysis of model-X knockoffs \citep{candes2018panning}, 
beginning with a review of the knockoffs procedure.

Consider again the regression setting where our data is composed of $[\ndata{X}, \ndata{Y}]$, whose rows are i.i.d. copies of $(X,Y)\sim F_{X,Y}$.
The (eponymous) first step of the knockoffs procedure is to generate knockoffs. We say the $n\times p$ random matrix $\tilde{\mathbf{X}}$ is a knockoff matrix for $\ndata{X}$ if $\tilde{\ndata{X}}\ci \mathbf Y\mid\ndata{X}$ and the following \emph{pairwise exchangeability} is satisfied for each $j$:
\begin{equation*}
[\mathbf{X},\tilde{\mathbf{X}}]\eqd[\mathbf{X},\tilde{\mathbf{X}}]_{\text{swap}(j)},
\end{equation*}
where the subscript $\text{swap}(j)$ denotes swapping the $j$th and $(j+p)$th columns of a matrix or elements of a vector (in this case, swapping $\ndata{X}_j$ and $\tilde{\ndata{X}}_j$). In Setting~\ref{model:lr-iid}, because the covariates are independent, generating such knockoffs is particularly simple: we can just take $\tilde{\mathbf X}$ to be an i.i.d. copy of $\mathbf X$.

The second step is to define a variable importance statistic
\begin{equation*}
    T:=T([\mathbf{X},\tilde{\mathbf{X}}],\mathbf{Y})=(T_1,\dots,T_p,\tilde T_1,\dots, \tilde T_p),
\end{equation*}
which satisfies
\begin{equation*}
T([\mathbf{X},\tilde{\mathbf{X}}]_{\text{swap}(j)},\mathbf{Y})=(T_1,\dots,T_p,\tilde T_1,\dots, \tilde T_p)_{\text{swap}(j)}.
\end{equation*}
That is, swapping the column corresponding to the $j$th covariate $X_j$ with that of its knockoff $\tilde X_j$ will swap their corresponding variable importance statistics $T_j$ and $\tilde T_j$ and leave the other elements of $T$ unchanged. A typical example of $T$ is the absolute value of the fitted lasso coefficient vector from regressing $\mathbf{Y}$ on $[\mathbf{X},\tilde{\mathbf{X}}]$. $T$ is then plugged into an antisymmetric function $f(\cdot,\cdot)$ (i.e., $f(x,y)=-f(y,x)$) to compute $W\in\rr^p$: $W_j=f(T_j,\tilde T_j)$. For example, we can simply let $f(x,y)=x-y$.

The third step is variable selection.
It was shown in \citet{candes2018panning} that if we select the set of variables
\begin{equation}
\hat S=\left\{j:W_j\ge\hat w\right\},\quad\text{where}\qquad\hat w=\min\left\{w>0:\frac{1+|\{j:W_j\le -w\}|}{|\{j:W_j\ge w\}|}\le q\right\},\footnote{Formally, we make the minimum well-defined by only considering the minimum over $w\in\{|W_j|:W_j\ne0, 1\le j\le p\}$.}
\label{equation:knockoff-adapt}
\end{equation}
then the FDR is controlled at level $q$. 


\subsubsection{Marginal covariance and ordinary least squares variable importance statistics}
\label{sec:knockoffs-mc-ols}
A peculiarity of knockoffs is that its rejections are not determined by a vector of unordered $p$-values, but instead a ordered vector of signs, which could be viewed as ``one-bit'' $p$-values with an order. Thus, it is worthwhile to pause and discuss its relationship with $p$-values. As outlined in Section~\ref{sec:model-x-knockoffs}, knockoffs operates on unordered feature importance statistics $W_1,\dots,W_p$. If all the null $W_j$'s have the same marginal distribution, let $F$ be its CDF and consider oracle $p$-values given by $p_j=1-F(W_j)$ (such $p$-values cannot be computed in practice because $F$ is unknown).
Knockoffs with nominal FDR level $q$ rejects all $p$-values below $\hat t_\text{KF}$, where
\begin{equation*}
\hat t_\text{KF}=\max\left\{{t\in[0,\frac12)}:\frac{1+\#\{j:p_j\ge 1-t\}}{\#\{j:p_j\le t\}}\le q\right\}.
\end{equation*}
This is equivalent to the intercept-only AdaPT procedure applied to the $p_j$'s \citep{lei2018adapt}.
Thus, we can understand the asymptotic behavior of the knockoffs procedure by studying the joint distribution of the $p_j$'s. In fact, Theorem~\ref{theorem:knockoff-power} in Appendix~\ref{sec:proofs} shows that under certain conditions, we can treat the $p_j$'s as independent draws from their respective asymptotic marginal distributions.


In proving the expressions for the asymptotic power of multiple testing with CRT $p$-values, we needed to analyze the asymptotic distributions of pairs of test statistics, and it turns out the same tools are sufficient for both establishing the assumptions of Theorem~\ref{theorem:knockoff-power} and for characterizing the marginal distributions of the $p_j$'s, except that analysis of the asymptotic distributions of sets of \emph{four} test statistics is needed.
In particular, Lemma~\ref{lemma:knockoff-conditions} in Appendix~\ref{sec:proofs} says that we just need to check that for distinct $j$ and $k$, $(T_j, T_k, \tilde T_j, \tilde T_k)$ converges in distribution to a random vector with independent coordinates in order for Theorem~\ref{theorem:knockoff-power} to hold. 

While our asymptotic analysis of $(T_j, T_k, \tilde T_j, \tilde T_k)$ for a given statistic allows us to obtain the asymptotic power of knockoffs for any antisymmetric function $f$, when $f(x,y)=x-y$, if the test statistic is the marginal covariance or the OLS coefficient, there is a direct and easily interpretable connection to the AdaPT procedure applied to a normal means model, and we can state our results using the language of effective $\pi_\mu$ from Definition~\ref{definition:effect-pi}. 
\begin{thm}
In Setting~\ref{model:lr-iid}, for almost every $q\in(0,1)$, knockoffs with $\tilde{\mathbf{X}}$ an i.i.d. copy of $\ndata{X}$ and the antisymmetric function $f(x,y)=x-y$ at level $q$ with marginal covariance or OLS test statistic has the following one-sided effective $\pi_\mu$'s with respect to the AdaPT procedure at level $q$:
\begin{enumerate}
    \item For the marginal covariance statistic, the effective $\pi_\mu$ is the distribution of $\frac{1}{\sqrt{2(\sigma^2+\kappa\e[B_0^2])}}B_0$.
    \item For the OLS statistic, assuming $\kappa<1/2$, the effective $\pi_\mu$ is the distribution of $\frac{\sqrt{1-2\kappa}}{\sqrt{2\sigma^2}}B_0$.
\end{enumerate}
\label{theorem:coro-knockoff}
\end{thm}
Note that our general result Theorem~\ref{theorem:knockoff-power} covers the two-sided case, which is equivalent to taking $f(x,y)=|x|-|y|$, but it cannot be expressed in terms of an effective $\pi_\mu$.
We can observe that the effective $\pi_\mu$'s in Theorem~\ref{theorem:coro-knockoff} agree with Theorems~\ref{theorem:prop-power-mc} and \ref{theorem:prop-model-X-CRT-OLS}. 
We see that knockoffs with the OLS statistic outperforms knockoffs with the marginal covariance statistic if $\sigma^2/(1-2\kappa)<\sigma^2+\kappa\e[B_0^2]$, and vice versa. Comparing Theorems~\ref{theorem:coro-knockoff} and \ref{theorem:coro-BH-CRT}, we can see that multiple testing with CRT $p$-values effectively increases the signal size by a factor of $\sqrt2$ compared to knockoffs for the marginal covariance. For OLS, multiple testing with CRT $p$-values effectively increases the signal size by a factor of $\sqrt2\sqrt{(1-\kappa)/(1-2\kappa)}>\sqrt2$ over knockoffs, with the additional factor $\sqrt{(1-\kappa)/(1-2\kappa)}$ approaching infinity as $\kappa\to1/2$ from below.

\subsubsection{Knockoffs with the lasso coefficient}
\label{sec:knockoff-lasso}
The lasso coefficient is a popular statistic frequently used with knockoffs. Specifically, let $\hat\beta^\lambda$ be the coefficient estimate from using the lasso to regress $\ndata{Y}$ on $[\ndata{X},\tilde{\mathbf X}]$ with penalty parameter $\lambda$. We suppress the superscript $\lambda$ when there is no confusion. For $j=1,2,\dots,p$, let $T_j=\sqrt n\hat\beta_j$, $\tilde T_j=\sqrt n\hat\beta_{j+p}$, and $W_j=f(T_j,\tilde T_j)$ for some antisymmetric function $f$. 

\begin{thm}
In Setting~\ref{model:lr-iid}, knockoffs with $\tilde{\ndata{X}}$ an i.i.d. copy of $\ndata{X}$, antisymmetric function $f(x,y)$ satisfying the mild regularity condition in Theorem~\ref{theorem:lasso-fdr-power}, and variable importance statistic $\hat\beta^\lambda$, at Lebesgue-almost-every level $q\in(0,1)$, has the same asymptotic power as if the $\sqrt n\hat\beta^\lambda_j$'s were independent, where for $j=1,2,\dots,p$, $\sqrt n\hat\beta^\lambda_j\sim\eta(B_0+\tau_\lambda Z;\alpha_\lambda\tau_\lambda)$ and $\sqrt n\hat\beta^\lambda_{j+p}\sim\eta(\tau_\lambda Z;\alpha_\lambda\tau_\lambda)$, $Z\sim\mathcal N(0,1)$ independent of $B_0$.
\label{theorem:lasso-fdr-power-simple}
\end{thm}
See Theorem~\ref{theorem:lasso-fdr-power} in Appendix~\ref{sec:proofs} for a detailed presentation. We prove the asymptotic independence via a symmetry argument, and extra care is taken due to the fact that the $\hat\beta_j$'s have a delta mass at $0$. During the preparation of this manuscript, we discovered an independent and parallel work on the asymptotic power of knockoffs using the lasso coefficient difference statistic \citep{weinstein2020power}, which provides a nearly identical power result to our Theorem~\ref{theorem:lasso-fdr-power-simple}. 

Now we heuristically compare the asymptotic power of knockoffs with the lasso coefficient with that of multiple testing with CRT $p$-values obtained from the distilled lasso statistic. Since the two results involve two different $\tau_\lambda$'s, we differentiate them with $\tau^\textnormal{CRT}_{\lambda_\textnormal{CRT}}$ and $\tau^\textnormal{KF}_{\lambda_\text{KF}}$, respectively (note that the two are generally different even when $\lambda_\textnormal{CRT}=\lambda_\textnormal{KF}$, as they also implicitly depend on other parameters). From Theorem~\ref{theorem:lasso-fdr-power-simple}, we can interpret $\tau_\lambda$ as the standard deviation of the noise added to the signal $B_0$, with a thresholding operation afterwards. On the other hand, we see from Theorem~\ref{theorem:coro-BH-CRT}, by a rescaling of mean and variance, $\tau^\textnormal{CRT}_{\lambda_\textnormal{CRT}}$ as the standard deviation of noise added to the signal $B_0$. It turns out that if we choose the best oracle $\lambda$ for the CRT, $\tau^\textnormal{CRT}_{\lambda_\textnormal{CRT}}\le\tau^\textnormal{KF}_{\lambda_\text{KF}}$, because knockoffs doubles the dimension of covariates and thus introduces more noise (see Appendix~\ref{sec:tau} for a formal proof). Intuitively, we should expect higher power from the CRT. We provide a numerical comparison in Section~\ref{sec:numerical-comparison}.

\subsection{Asymptotic power comparison of multiple testing with CRT $p$-values and knockoffs}
\label{sec:numerical-comparison}

When the marginal covariance or the OLS coefficient is used as the variable importance statistic and the antisymmetric function is $f(x,y)=x-y$, Theorem~\ref{theorem:coro-knockoff} provides a direct comparison between the asymptotic power of knockoffs with that of multiple testing with CRT $p$-values (see the text after Theorem~\ref{theorem:coro-knockoff}). In practice, we usually do not know the signs of the signals, so it is more common to use the absolute value of the marginal covariance, the OLS coefficient, or the lasso coefficient as the variable importance statistic (or, equivalently, take $f(x,y)=|x|-|y|$). These results do not fit into Definition~\ref{definition:effect-pi} with an effective $\pi_\mu$ because asymptotically, although the test statistics are independent, they are not marginally Gaussian. In this section, we numerically compare these results with the power of multiple testing with CRT $p$-values. We see in Figures~\ref{figure:mc_compare}, \ref{figure:ols_compare}, and \ref{figure:lasso_compare} that knockoffs is less powerful than the CRT methods. It is interesting that in lower-dimensional settings such as $n=2.5p$, knockoffs with two-sided test statistics is more powerful than the $\sqrt2$-signal strength reduction relative to the CRT suggested by the analysis for one-sided test statistics in Section~\ref{sec:knockoffs-mc-ols}, and for the lasso statistic, there is almost no power difference in such lower-dimensional settings. 

\begin{figure}[h]
    \centering
\includegraphics[width = 0.9\textwidth]{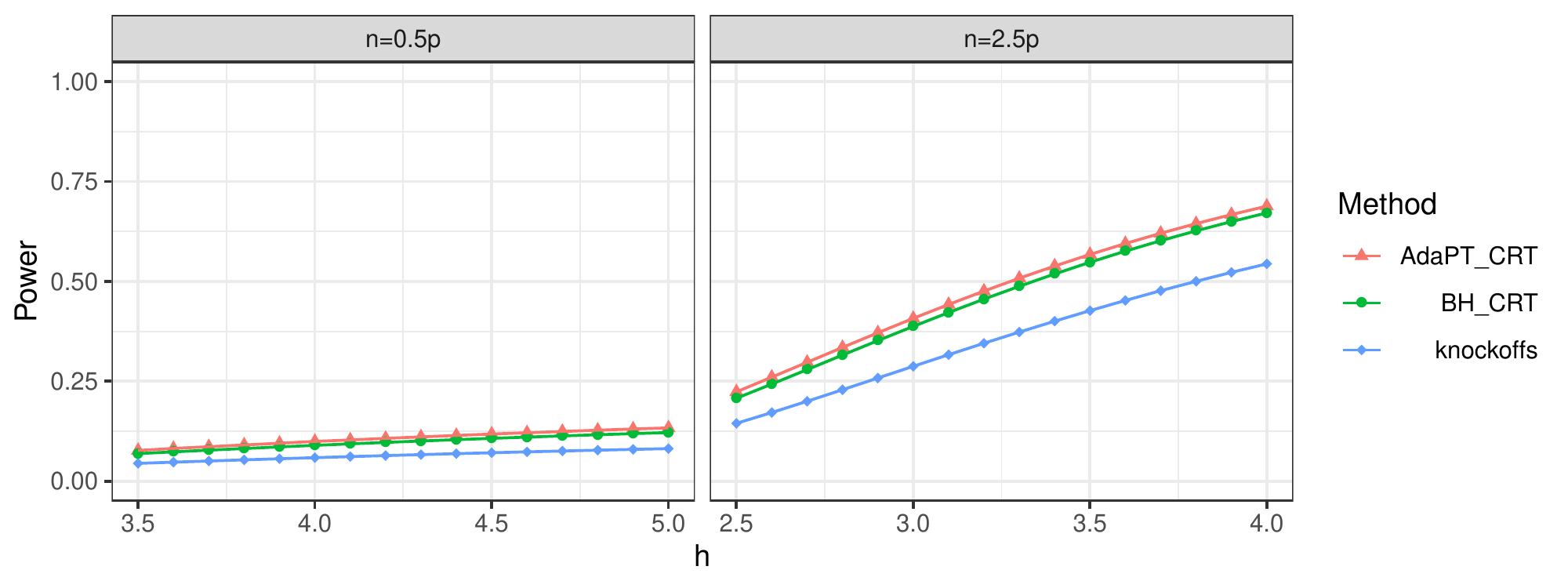}
    \caption{Asymptotic power comparison for BH and AdaPT applied to two-sided CRT $p$-values and knockoffs with the absolute value of the marginal covariance with the original signal size and $\sqrt2$ times the signal size. Plot is in Setting~\ref{model:lr-iid} with $\gamma=0.9$ and $\pi_1=\delta_h$ with varying $h$, and the nominal FDR level is $0.1$.}
    \label{figure:mc_compare}
\end{figure}

\begin{figure}[h]
    \centering
\includegraphics[width = 0.6\textwidth]{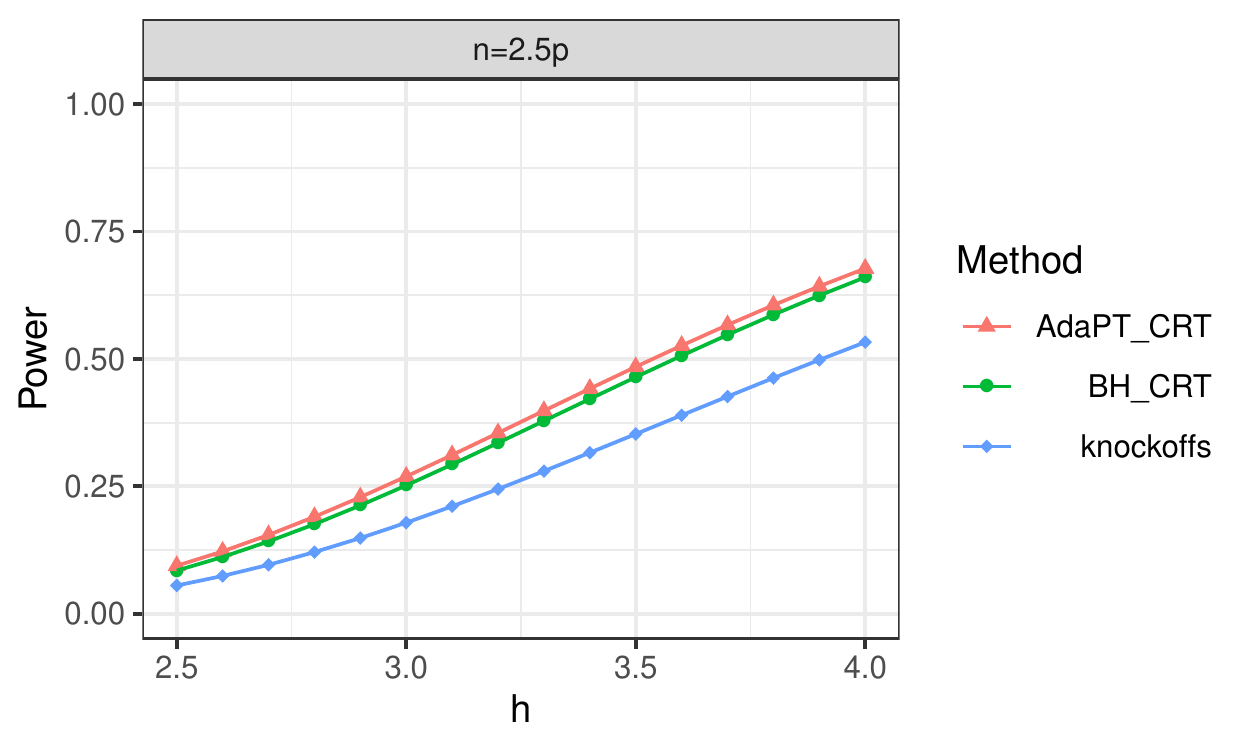}
    \caption{Asymptotic power comparison for BH and AdaPT applied to two-sided CRT $p$-values and knockoffs with the absolute value of the OLS coefficient with the original signal size and $\sqrt2$ times the signal size. The setting is the same as in Figure~\ref{figure:mc_compare}.}
    \label{figure:ols_compare}
\end{figure}

\begin{figure}[h]
    \centering
\includegraphics[width = 0.9\textwidth]{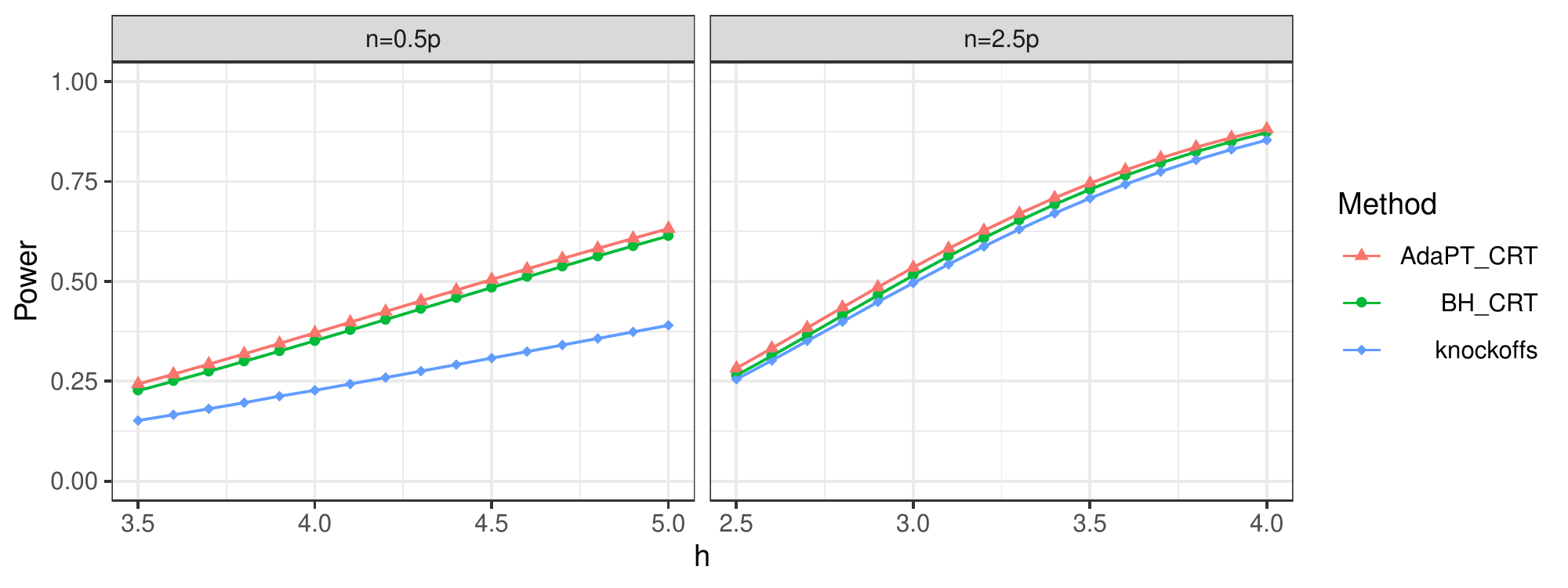}
    \caption{Asymptotic power comparison for BH and AdaPT applied to two-sided CRT $p$-values with the distilled lasso statistic and knockoffs with the absolute value of the lasso coefficient with the original signal size and $\sqrt2$ times the signal size. The setting is the same as in Figure~\ref{figure:mc_compare}. For all methods, $\lambda$ is chosen so that the asymptotic power is maximized (for the CRT, this is equivalent to minimizing $\tau_\lambda$).}
    \label{figure:lasso_compare}
\end{figure}



\section{Retrospective sampling}
\label{sec:retro}
As a generalization of our results for the CRT in Section~\ref{sec:single}, we consider a case in which we know the distribution of $\mathcal L(X\,|\,Z)$, but the data have been collected retrospectively. Specifically, we assume the following model.
\begin{model}[High-dimensional linear model with retrospective sampling] Let $g:\rr\to[0,1]$ be a Borel function that is not almost everywhere $0$. For each $p$, generate i.i.d. data from Setting~\ref{model:moderate-dim-lr} and reject each data point $(X_i,Y_i,Z_i)$ with probability $1-g(Y_i)$, until $n$ data points have been collected, such that $p/n\to\kappa$ still holds.
\label{model:retro}
\end{model}
\citet[Proposition~1]{barber2019construction} established that the CRT remains valid when Setting~\ref{model:moderate-dim-lr} is assumed but the data actually come from Setting~\ref{model:retro}. In addition to single hypothesis testing in Setting~\ref{model:retro}, we will also consider the variable selection with the CRT $p$-values, coming from Setting~\ref{model:retro-crt} as follows.

\begin{model}[High-dimensional linear model with retrospective sampling] Let $g:\rr\to[0,1]$ be a Borel function that is not almost everywhere $0$. For each $p$, generate i.i.d. data from Setting~\ref{model:lr-iid} and reject each data point $(X_i,Y_i)$ with probability $1-g(Y_i)$, until we have $n$ data points, such that $p/n\to\kappa$ still holds.
\label{model:retro-crt}
\end{model}
\citet{barber2019construction} also established that the knockoffs are still valid when Setting~\ref{model:lr-iid} is assumed but the data actually come from Setting~\ref{model:retro-crt}. Thus, in this section, we will consider knockoffs generated independently as in Section~\ref{sec:knockoffs}. 
The following theorem gives the asymptotic power of the CRT and knockoffs using the marginal covariance statistic with retrospective sampling.
\begin{thm}
Consider using the test statistics $T=n^{-1}\ndata{X}^\top\ndata{Y}$ for the CRT, and $T_j=n^{-1}\ndata{X}_j^\top\ndata{Y}$ for multiple testing with CRT $p$-values and knockoffs. Let $M_{\textnormal{retro}}^2$ be the asymptotic second moment of the retrospectively collected $Y_i$, i.e.,
\begin{equation*}
M_{\textnormal{retro}}^2=\frac{\e[Y_\textnormal{raw}^2g(Y_\textnormal{raw})]}{\e[g(Y_\textnormal{raw})]},
\end{equation*}
where $Y_\textnormal{raw}\sim\mathcal N(0,\sigma^2+v_Z^2)$ is drawn from the asymptotic distribution of $Y$ without rejection.\footnote{$M_\textnormal{retro}$ always exists because $g(y)\in[0,1]$ and is not almost everywhere zero.} Note that in Setting~\ref{model:retro-crt}, the corresponding $v_Z^2$ (or $v_{X_{\text{-}j}}^2$) is equal to $\kappa\e[B_0^2]$.
\begin{enumerate}
    \item In Setting~\ref{model:retro}, the asymptotic power of the CRT is equal to that of a $z$-test with standardized effect size
    \[ \frac{h{M_{\textnormal{retro}}}}{v_Z^2+\sigma^2}. \]

\item In Setting~\ref{model:retro-crt}, for almost all $q\in(0,1)$, BH or AdaPT at level $q$ applied to CRT $p$-values using $T_j$ (or $|T_j|$) have one-sided (or two-sided) effective $\pi_\mu$ given by the distribution of $\frac{M_{\textnormal{retro}}}{\sigma^2+\kappa\e[B_0^2]}B_0$ with respect to BH or AdaPT at level $q$.

\item In Setting~\ref{model:retro-crt}, for almost all $q\in(0,1)$, knockoffs with $\tilde{\ndata{X}}$ an i.i.d. copy of $\ndata{X}$, antisymmetric function $f(x,y)=x-y$, test statistic $T_j$, and level $q$ has one-sided effective $\pi_\mu$ given by the distribution of $\frac{M_{\textnormal{retro}}}{\sqrt{2}(\sigma^2+\kappa\e[B_0^2])}B_0$ with respect to AdaPT at level $q$.
\end{enumerate}
\label{theorem:retro-mc}
\end{thm}
To sum up, Theorem~\ref{theorem:retro-mc} establishes that for retrospective sampling, the same results on the asymptotic power hold with the signal size multiplied by $\frac{M_\text{retro}}{\sqrt{\sigma^2+v_Z^2}}$.
Thus, the power gets higher as $M_{\textnormal{retro}}$ gets larger. This is intuitive, since it should be easier for us to detect the signal in regions where $Y$ has extreme values. As a special case, if $g\equiv1$, then $M_{\textnormal{retro}}=\sqrt{\sigma^2+v_Z^2}$ and we return to the non-retrospective sampling case.

While the asymptotic power expressions for retrospective sampling can be higher than that of non-retrospective sampling, it comes at a price of requiring more raw samples, and it is worthwhile to discuss the implications. Let $n_\text{raw}$ be the number of raw samples needed to get $n$ retrospective samples, then $n/n_\text{raw}\to\int\phi_{\sigma^2+v_Z^2}(y)g(y)\di y$. If we do not discard any samples and use all $n_\text{raw}$, we return to the non-retrospective sampling settings with $h$ increased to $h/\sqrt{\int\phi_{\sigma^2+v_Z^2}(y)g(y)\di y}$ (or $B_0$ to $B_0/\sqrt{\int\phi_{\sigma^2+v_Z^2}(y)g(y)\di y}$). One can then directly compare the asymptotic powers and note that, as intuition would suggest, the power is maximized when no sample is rejected. In practice, however, collecting covariates might be expensive. Therefore, it can be beneficial to decide whether or not to collect the covariates based on a screening step using the value of $Y$. A natural question is then how to achieve the highest power while fixing the sampling cost. This is equivalent to maximizing $M_{\textnormal{retro}}$ while fixing $\int\phi_{\sigma^2+v_Z^2}(y)g(y)\di y$ and it is not hard to see that the maximum is attained when $g(y)=\mathbf{1}_{|y|>C}$ for an appropriate $C$.

\section{Simulations}
\label{sec:discussion}
In this section, we examine the finite-sample accuracy of our asymptotic power expressions.

\subsection{CRT in Setting~\ref{model:moderate-dim-lr}}
\label{sec:simu-crt}
In Figure~\ref{figure:CRT}, we compare the power of the CRT with each statistic mentioned in Section~\ref{sec:bh-p-val}. 
We plot as a horizontal line the power of the CRT with an oracle statistic that is the upper bound for the achievable power with the CRT (see Appendix~\ref{sec:crt-opt-bayes}). We can see that the distilled lasso statistic has comparable power with the optimal statistic.

\begin{figure}[H]
    \centering
\includegraphics[width = 0.9\textwidth]{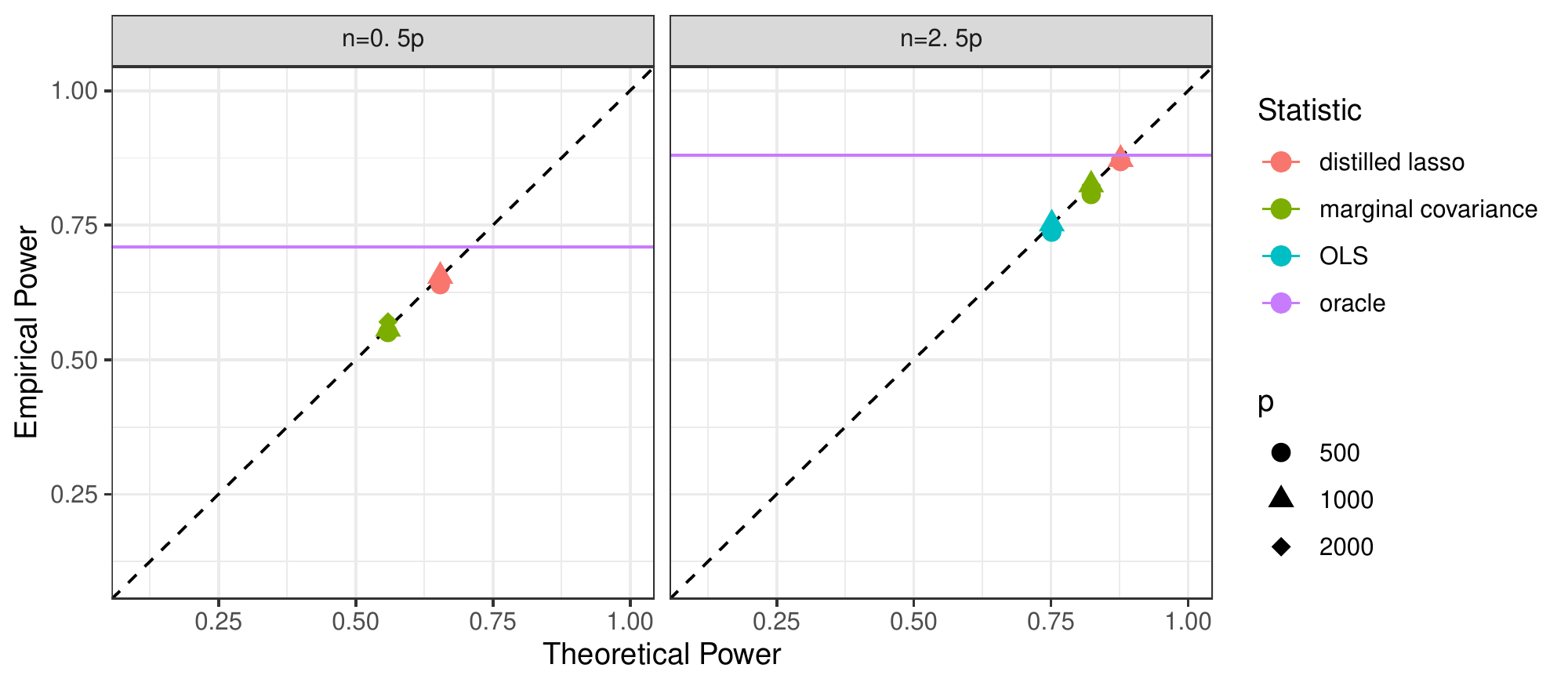}
    \caption{Comparison of the powers of the CRT using different statistics. The setting is Setting~\ref{model:moderate-dim-lr} with $\sigma^2=1$, $\xi=0$, $\beta=3$, $\sqrt n\theta_j\simiid0.9\delta_0+0.1\delta_3$. The results for Bayes are empirical based on 960 independent simulations. For the distilled lasso statistic, $\lambda$ is chosen so that $\tau_\lambda$ is minimized for highest asymptotic power. All standard errors are below $0.01$.}
    \label{figure:CRT}
\end{figure}

\subsection{Conjecture~\ref{conjecture:model-X-unlabeled}}
In this section, we show simulation results regarding Conjecture~\ref{conjecture:model-X-unlabeled} in Section~\ref{sec:CRT-conditional}. In Figure~\ref{figure:model-X-conjecture}, we plot the conjectured asymptotic power and empirical finite-sample power as a function of $n_*/p$ with $n/p=1.5$ fixed for two different values of $v_Z^2$. Note that the conjectured power must agree with the empirical power in the limit as $p$ and $n_*/p$ go to infinity (Theorem~\ref{theorem:prop-model-X-unlabeled}). 
\begin{figure}[H]
    \centering
\includegraphics[width = 0.9\textwidth]{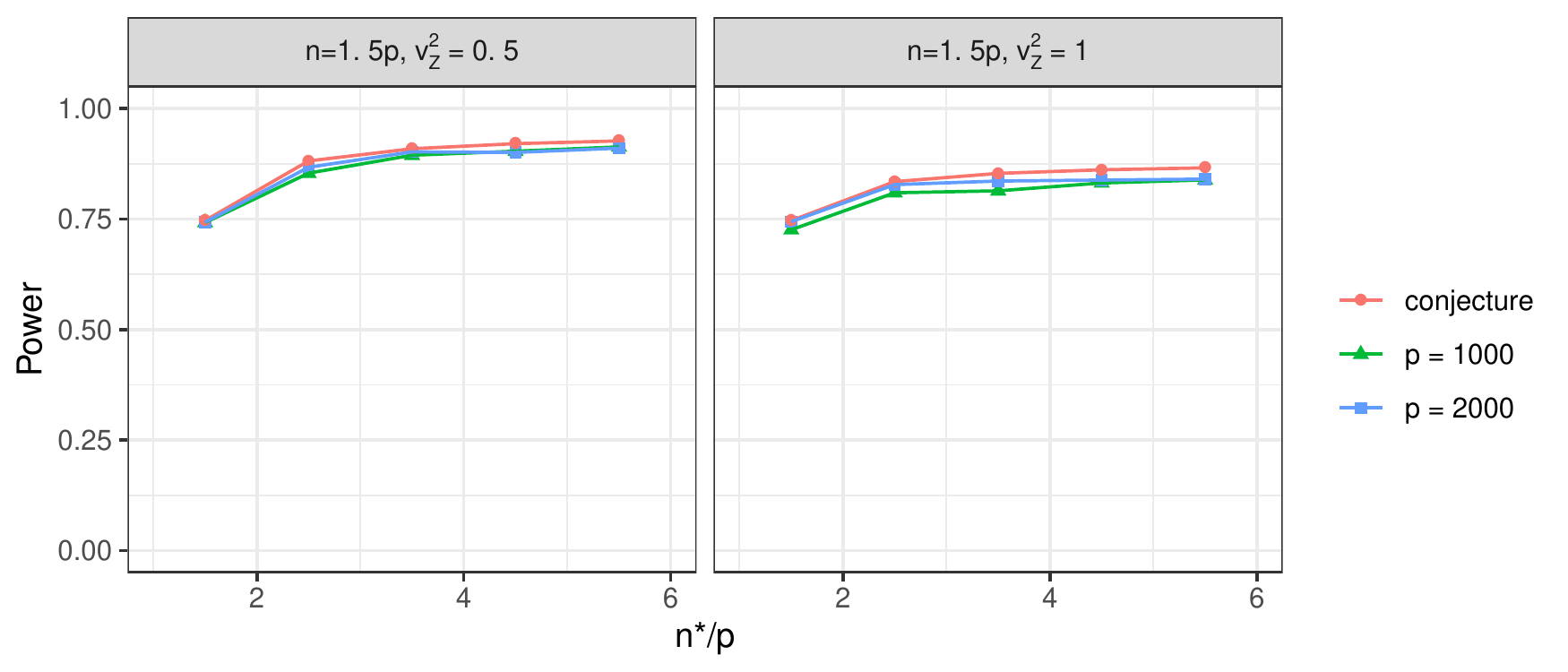}
    \caption{Simulations are for $p=1,000$ and $p=2,000$ with $h=4$ in the setting of Conjecture~\ref{conjecture:model-X-unlabeled}. All standard errors are below $0.01$.} 
    \label{figure:model-X-conjecture}
\end{figure}

\subsection{Multiple testing with CRT $p$-values and knockoffs}
In this section, we show some simulation results of BH applied to CRT $p$-values (BH-CRT) and knockoffs with the statistics discussed in this paper. We defer the results of AdaPT applied to CRT $p$-values to Appendix~\ref{sec:bh-adapt} so we do not crowd the plots; in summary, AdaPT has slightly higher power and FDR and converges more slowly than BH, especially in low-power settings. In our simulations, $\gamma=0.9$ and $\pi_1$ is a point mass at $h=4$. We use absolute values of the statistics, since in practice we do not know the sign of $h$. 
Points with the same color represent methods with the same statistic and different $p$'s, including $p=\infty$, which is calculated based on our theory. It can be seen that points of the same color form separated clusters, which means that our theory could guide statistic choices even in finite samples. We note that the finite-sample agreement is not quite as good for knockoffs in lower-power settings as that for the BH-CRT, because of the discreteness in the numerator of the FDP estimate in the knockoffs procedure (the fraction in Equation~\eqref{equation:knockoff-adapt}). We also include the results of an oracle using the Bayesian method that controls the Bayesian FDR (see Appendix~\ref{sec:bayes-fdr}), and BH-CRT with the distilled lasso statistic can be close to this oracle method when $n=2.5p$, while when $n<p$, there is still a substantial gap as would be expected given the relative value of the prior to the smaller sample size. 

\begin{figure}[H]
    \centering
\includegraphics[width = 0.9\textwidth]{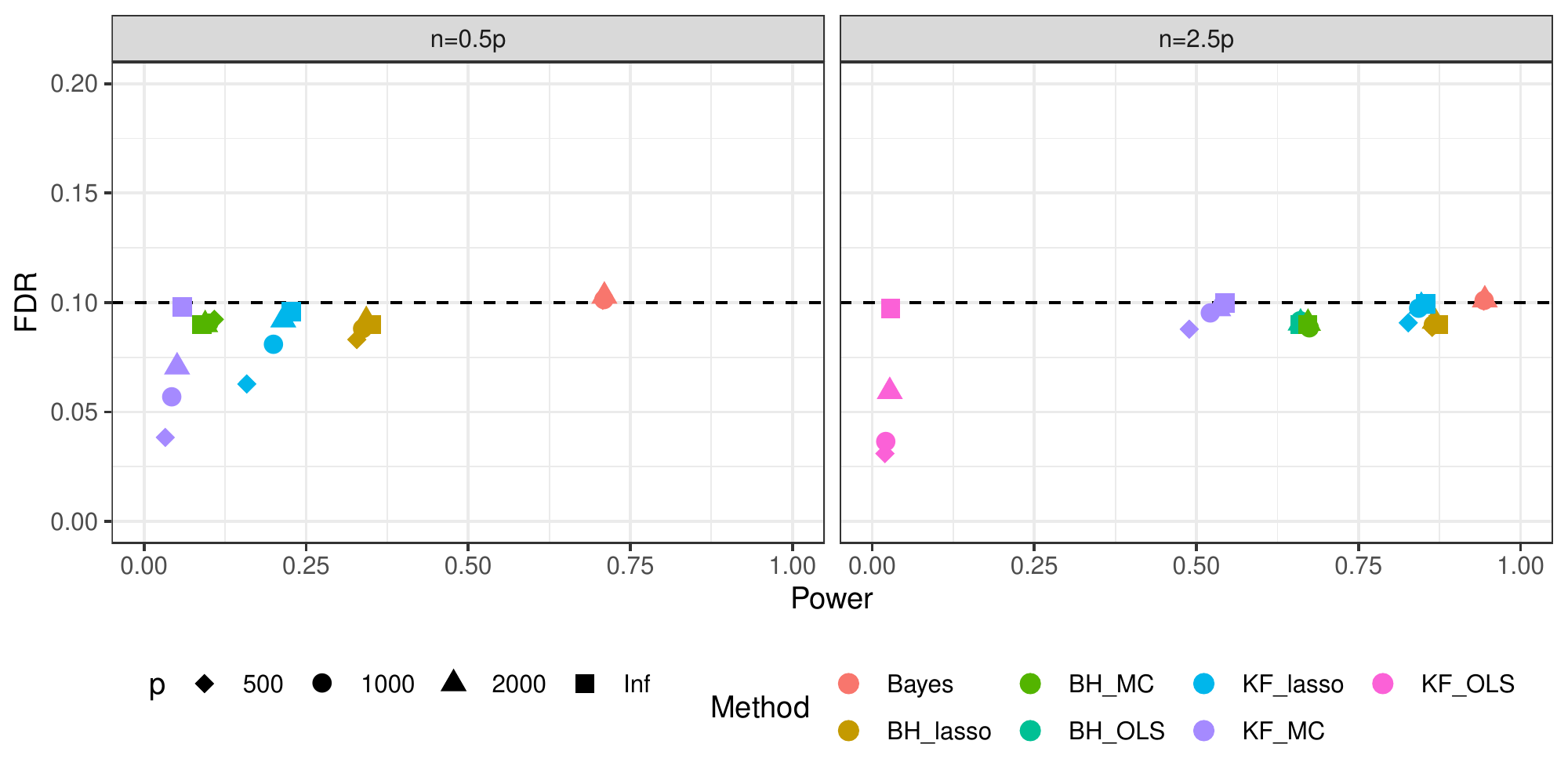}
    \caption{Power comparison of different methods at FDR level $0.1$ with the same setting as that of Figure~\ref{figure:lasso_compare}. For the lasso statistics, $\lambda$ is chosen so that the asymptotic power is maximized (for the CRT, this is equivalent to minimizing $\tau_\lambda$), while we note that the specific choice of $\lambda$ only affects the power mildly within a reasonable range (see Figure~\ref{figure:tau}). All standard errors are below $0.01$.} 
    \label{figure:comp_all}
\end{figure}

\subsection{Retrospective sampling}
In this section, we compare the empirical and theoretical powers of the CRT in Setting~\ref{model:retro}, where $g$ is taken to be of the form $g(x)=\ind_{|x|>\text{threshold}}$ for different $\text{threshold}$s.
\begin{figure}[H]
    \centering
\includegraphics[width = 0.9\textwidth]{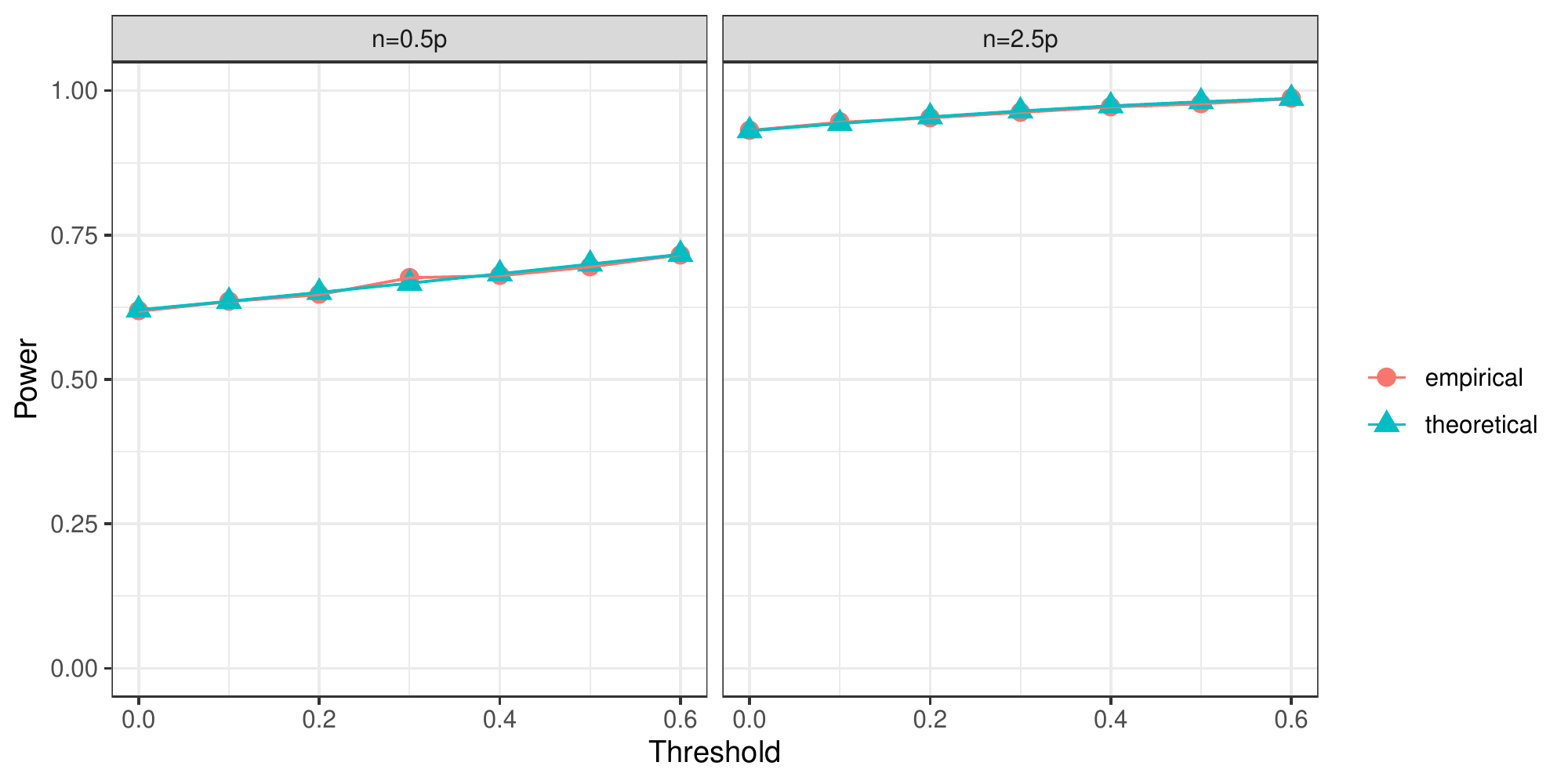}
    \caption{Comparison of the empirical ($p=500$) and theoretical (asymptotic) powers of the CRT with the marginal covariance statistic in the retrospective sampling setting (Setting~\ref{model:retro} with $\sigma^2=1$, $\xi=0$, $\beta=4$ and $\sqrt n\theta_j\simiid0.9\delta_0+0.1\delta_4$). All standard errors are below $0.01$.} 
    \label{figure:retro}
\end{figure}

\section{Discussion}
\label{sec:discussion-new}
This paper studied the asymptotic powers of the CRT and knockoffs in the high-dimensional regime, i.e., as $n,p\to\infty$, $n/p$ goes to a positive constant and a fixed non-zero proportion of variables are non-null. 
A very natural future direction is to study the behavior of the CRT and knockoffs with different statistics and/or in other settings. For example, \citet{celentano2020lasso} could provide starting points on extending our lasso power analysis to settings with correlated covariates, while \citet{sur2019modern,liang2020precise} could enable the study of binary regression settings and their corresponding test statistics. 
Alternatively, the power analysis of oracle test statistics (e.g., the one in Appendix~\ref{sec:crt-opt-bayes}) could provide theoretical bounds on the power of these methods with any statistics.

\section*{Acknowledgements}
The authors would like to thank Hong Hu, Tracy Ke, Natesh Pillai, Subhabrata Sen, and Pragya Sur for valuable discussions and suggestions. L. J. was partially supported by the William F. Milton Fund.

\bibliography{refs}{}

\newpage
\appendix

\section{Notation}
Bold letters are used for matrices or vectors containing i.i.d. observations. Unless specified otherwise, a vector is always a column vector instead of row vector. For a vector $a$, $a_S$ denotes the sub-vector that consists of elements indexed by $S$; for a matrix $A$, $A_{S,S}$ denotes the sub-matrix that consists of rows and columns indexed by $S$. For integers $i\le j$, the notation $i:j$ means the set $\{i,i+1,\dots,j\}$, and we use $[p]$ to denote $1:p$. For a set $S\subseteq[p]$, $|S|$ denotes the number of elements in $S$, $\text{-}S$ denotes the set $[p]\setminus S$. Let $I_d$ be the $d\times d$ identity matrix and for $d_1\le d_2$, let $I_{d_1\times d_2}$ be the matrix obtained by adding $(d_2-d_1)$ rows of zeros to $I_{d_1}$. Let $\mathbb S^{d-1}=\{x\in\rr^d,\|x\|_2=1\}$. The indicator function of $B$ is denoted as $\mathbf{1}_B$, i.e., it takes value $1$ on $B$ and zero otherwise. The cumulative distribution function (CDF) of the Gaussian distribution $\mathcal N(0,1)$ is denoted by $\Phi$---for $\alpha\in(0,1)$, $z_\alpha$ denotes the $\alpha$-quantile of $\mathcal N(0,1)$, i.e., $\Phi(z_{\alpha})=\alpha$. We use $\chi_k^2$ and Inv-$\chi_k^2$ to denote the chi-squared distribution and inverse chi-squared distribution with $k$ degrees of freedom. For random variables or vectors $W_1$ and $W_2$, $\mathcal L(W_1)$ means the distribution of $W_1$ and $\mathcal L(W_1\,|\,W_2)$ means the conditional distribution of $W_1$ given $W_2$. To ease notation when analyzing the power and false discovery rate, we use the convention that $0/0$ is defined to be $0$. Unless another measure is explicitly specified, ``almost everywhere'' or ``almost every'' is with respect to the Lebesgue measure.
\section{CRT under low-dimensional asymptotics}
\label{sec:scalar-CRT}
As a side note, we consider a case in which we test a scalar parameter with no nuisance parameters under the asymptotics of local alternatives. One can think of this case as testing if a coefficient is zero in a linear regression setting, where the other coefficients are known. A similar problem was studied in \citet{katsevich2020theoretical}, the difference of which we will discuss below.


We consider the setting with i.i.d. data $(X_i,Y_i,Z_i)_{i=1}^n=(\ndata{X},\ndata{Y},\ndata{Z})$. Recall that $(X,Y,Z)$ is actually simplified notation for $(X_{j},Y,\Xno{j})$. The null distribution is $(\ndata{X},\ndata{Y},\ndata{Z})\sim P_{\theta_0}^n$, where $X\ci Y\mid Z$ under $P_{\theta_0}$. The alternative distribution is $(\ndata{X},\ndata{Y},\ndata{Z})\sim P^n_{\theta_0+hn^{-1/2}}$, where $h$ is a fixed scalar. We assume $P_\theta$ is q.m.d. and thus the two sequences are contiguous (see Appendix~\ref{sec:contiguity-qmd}). In other words, we are testing $H_0:\theta=\theta_0$ against $H_1:\theta=\theta_0+h/\sqrt{n}$ with $n$ independent draws from $P_\theta$. For presentational simplicity, suppose we know the sign of $h$ is positive, while the case where we do not know the sign of $h$ can be similarly studied. We remark that although contiguity gives us an interesting setting to analyze non-trivial power, it is not a necessary condition (see Appendix~\ref{sec:local}).

Asymptotically linear statistics are an important class of statistics, which are of the form
\begin{equation*}
T_n(\ndata{X},\ndata{Y},\ndata{Z})=n^{-1/2}\sum_{i=1}^n\psi(X_i,Y_i,Z_i)+o_{\p_{H_0}}(1),
\end{equation*}
where 
$o_{\p_{H_0}}(1)$ denotes a term that goes to zero in probability under $H_0$. Many statistics can be written in this form, e.g., the log-likelihood ratio statistic and the score statistic. We will see that this class of statistics also can offer a most powerful test.

As suggested in Section~\ref{sec:CRT-introduction}, the CRT is run by finding a cutoff $c_\alpha(\ndata{Y},\ndata{Z})$ such that we get an exactly size-$\alpha$ test conditional on $(\ndata{Y},\ndata{Z})$ by rejecting when $T_n< c_\alpha$, accepting when $T_n> c_\alpha$, and randomizing when $T_n= c_\alpha$.
\begin{thm}
\label{theorem:asym-power-local-alt}
The asymptotic unconditional power of the above test under the local alternatives is
\begin{equation*}
1-\Phi\left(z_{1-\alpha}-h\sqrt{\Var_{H_0}(s(X_i,Y_i,Z_i,\theta_0))}\Corr_{H_0}\left({\psi(X_i,Y_i,Z_i)-e_0(Y_i,Z_i)},s(X_i,Y_i,Z_i,\theta_0)\right)\right),
\end{equation*}
where,
\begin{equation*}
e_0(Y,Z)=\e_{H_0}[\psi(X,Y,Z)\,|\,Y,Z]
\end{equation*}
and $s$ is the score function, which, under very general regularity conditions,\footnote{See Theorem~12.2.1 in \citet{lehmann2006testing} for an example of such conditions. There, the notation $\tilde\eta$ is used instead of $s$.} admits the common form
\begin{equation*}
s(X,Y,Z,\theta_0)=\frac{\frac\partial{\partial\theta}\big|_{\theta=\theta_0}p_{\theta}(X,Y,Z)}{p_{\theta_0}(X,Y,Z)},
\end{equation*}
where $p_\theta$ is the density of $P_\theta$.
\end{thm}
Let $\varphi_\psi(X_i,Y_i,Z_i)=\psi(X_i,Y_i,Z_i)-e_0(Y_i,Z_i)$. We see that to achieve high power, we need to find a $\psi$ such that $\varphi_\psi$ is highly correlated with $s(X_i,Y_i,Z_i,\theta_0)$.

\begin{remark}
\label{remark:optimal}
If $\e_{H_0}[s(X,Y,Z,\theta_0)\,|\,Y,Z]=0$, which is satisfied when the distribution of $(Y,Z)$ does not depend on $\theta$ (but this is not necessary), then we can use $\psi=s$ itself and achieve the optimal asymptotic power (this is also the Neyman--Pearson statistic and achieves the unconditional optimal asymptotic power; see Example~12.3.12 in \citet{lehmann2006testing}). This means the family of asymptotically linear statistics includes an asymptotically most powerful test if $\e_{H_0}[s(X,Y,Z,\theta_0)\,|\,Y,Z]=0$. This partially answers the question about model-X optimality in Remark~1 of \citet{katsevich2020theoretical} (i.e., the CRT with the score statistic is optimal among all valid tests in a certain asymptotic regime), which can also be seen as a generalization of the discussion ``A precise parallel with OLS'' in their Section~5.3 to non-linear regression settings.
\end{remark}

\begin{remark}
\label{remark:fixed-X}
Notation-wise, $X$ and $Y$ are symmetric, and thus the same result holds if we swap $X$ and $Y$, which actually corresponds to the traditional fixed-X test, i.e., a test that is valid conditional on the covariates $(X,Z)$. Since $\e_{H_0}[s(X,Y,Z,\theta_0)\,|\,X,Z]=0$ always holds when $(X,Z)$ is the covariate and $\mathcal L(X,Z)$ does not depend on $\theta$, we can always use $\psi=s$ to achieve the optimal asymptotic power in the fixed-X framework.
\end{remark}

\begin{remark}
Consider using the maximum likelihood estimate (MLE)
\begin{equation*}
\hat\theta_n=\argmax_\theta p_\theta(\ndata{X},\ndata{Y},\ndata{Z})
\end{equation*}
as the test statistic, which is equivalent to using $\sqrt{n}(\hat\theta_n-\theta_0)$ that satisfies
\begin{equation*}
\sqrt{n}(\hat\theta_n-\theta_0)=\frac{-\frac1{\sqrt n}\sum_{i=1}^n\ell'(\theta_0\,|\,X_i,Y_i,Z_i)}{\frac1n\sum_{i=1}^n\ell''(\theta'\,|\,X_i,Y_i,Z_i)}
\end{equation*}
for some $\theta'$ between $\theta_0$ and $\hat\theta_n$, where $\ell$ is the log-likelihood function. Since $-n^{-1}\sum_{i=1}^n\ell''(\theta'\,|\,X_i,Y_i,Z_i)$ converges in probability to the Fisher information $I(\theta_0)$ under the null (thus also under the alternative by contiguity; see, e.g., Theorem~12.3.2 in \citet{lehmann2006testing}), we see that
the standardized MLE is asymptotically equivalent to the score statistic up to a multiplicative constant. This signifies that it also enjoys the optimal asymptotic power under the same condition $\e_{H_0}[s(X,Y,Z,\theta_0)\,|\,Y,Z]=0$.
\end{remark}

\begin{remark}
This result is closely related to Theorem~1 in \citet{katsevich2020theoretical}, and we would like to highlight the key differences. (a) Our result applies to a general distribution $P_\theta$ and a general asymptotic linear statistic $\psi$, while \citet{katsevich2020theoretical} assume $\mathcal L(Y\,|\,X,Z)$ is Gaussian and considers a family of score-like statistics. (b) We assume there is no nuisance parameter, which corresponds to knowing the function $g$ in \citet{katsevich2020theoretical}; there, a deterministic estimate $\hat g$ is used instead, and the accuracy of $\hat g$ explicitly affects the power.
\end{remark}

We wish to emphasize that it is not true that the fixed-X framework always provides an optimal test, as seemingly suggested by Remarks~\ref{remark:optimal} and \ref{remark:fixed-X}. Specifically, Appendix~\ref{sec:fixed-no-power} exhibits a case where no fixed-X test can have nontrivial power, while a model-X test can, and Appendix~\ref{sec:dominate} shows that when testing a scalar parameter without nuisance parameters in non-asymptotic regimes, the optimal test can be a model-X one instead of a fixed-X one.

\section{Simple examples}
\label{sec:simple}

\subsection{Example where fixed-X has no power}
\label{sec:fixed-no-power}
Despite the fact that the fixed-X framework has been more heavily studied, it is not always ``better'' than the model-X framework. In fact, we provide a simple toy example where model-X methods have to be used for non-trivial inference. Consider the regression model
\begin{equation*}
\ndata{Y}\mid \ndata{X}\sim\mathcal N(\ndata{X}^\top\beta,I_n),X_{ij}\simiid\mathcal N(0,1), i=1,2,\dots,n, j=1,2,\dots,p, n<p-1.
\end{equation*}
Here, we use $\ndata{X}$ to denote the $n\times p$ data matrix and $\ndata{Y}$ to denote the $n\times1$ response vector. Now suppose we would like to construct a fixed-X statistical test for $H_0:\beta_1=0$. We claim that such a test must have trivial power. Formally, let $T_{\ndata{X}}(\ndata{Y})$ be a valid level-$\alpha$ test, i.e.,
\begin{equation}
\p(T_{\ndata{X}}(\ndata{Y})=1\mid \ndata{X},\beta)\le\alpha, \forall\beta\in\rr^p\text{ s.t. }\beta_1=0.
\label{eq:type-I-error-rate}
\end{equation}
To analyze its power, consider any $\gamma$ where $\gamma_1\ne0$. There exists $\tilde\gamma$ with $X^\top\tilde\gamma=0$ and $\tilde\gamma_1\ne0$. Then for any $a\in\rr$,
\begin{equation*}
(T_{\ndata{X}}(\ndata{Y})\mid \ndata{X},\gamma)\eqd (T_{\ndata{X}}(\ndata{Y})\mid \ndata{X},\gamma+a\tilde\gamma).
\end{equation*}
By picking $a^*=-\gamma_1/\tilde\gamma_1$, we conclude that the power
\begin{equation*}
\p(T_{\ndata{X}}(\ndata{Y})=1\mid \ndata{X},\gamma)=\p(T_{\ndata{X}}(\ndata{Y})=1\mid \ndata{X},\gamma+a^*\tilde\gamma)\le\alpha
\end{equation*}
by equation~\eqref{eq:type-I-error-rate}, since $(\gamma+a^*\tilde\gamma)_1=0$.

On the other hand, we could construct a non-trivial model-X test in the following way. Consider the test statistic
\begin{equation*}
T(\ndata{X},\ndata{Y})=\frac{\ndata{Y}^\top \ndata{X}_1}{\|\ndata{Y}\|}\sim\mathcal N(0,1)\text{ if }\beta_1=0.
\end{equation*}
where $\ndata{X}_1$ is the first column of $\ndata{X}$. If $\beta_1=0$, the statistic follows $\mathcal N(0,1)$. We will prove the power of the test which rejects when $|T(\ndata{X},\ndata{Y})|>z_{\alpha/2}=\chi_{1,\alpha}$ goes to a constant greater than $\alpha$ for a fixed $\beta_{\text{-}1}$ as $\beta_1\to\infty$. Let $\varepsilon=\ndata{Y}-\ndata{X}\beta$ and note that
\begin{equation*}
\frac{\ndata{X}_1^\top \ndata{Y}}{\|\ndata{Y}\|}=\frac{\ndata{X}_1^\top \ndata{Y}/\beta_1}{\|\ndata{Y}\|/\beta_1}=\frac{\|\ndata{X}_1\|^2+\sum_{j=2}^p\beta_j\ndata{X}_1^\top \ndata{X}_j/\beta_1+\ndata{X}_1^\top\varepsilon/\beta_1}{\|\ndata{X}_1+\sum_{j=2}^p\beta_j\ndata{X}_j/\beta_1+\varepsilon/\beta_1\|}\cip\|\ndata{X}_1\|\sim\chi_n.
\end{equation*}
Since $n>1$, the limit power is greater than $\alpha$.

\subsection{Example where model-X strictly dominates fixed-X}
\label{sec:dominate}



We present a simple example in this section, which reveals that in the finite-sample case, the most powerful test can be model-X instead of fixed-X. We will see in the following sections that this is not the case in asymptotic regimes. Let $X\sim f$ and $Y\mid X\sim\Bern(g_\theta(X))$. Let $(X_i,Y_i)_{i=1}^n$ be i.i.d. copies of $(X,Y)$. Assume $g_0(x)\equiv1/2$, $g_\theta(x)+g_\theta(-x)\equiv1$ and $g_\theta(x)$ is an increasing function of $x$ for $\theta>0$. We also assume $f$ has symmetric tails; that is, there is a positive constant $M$ such that $X\mid|X|>M\eqd-X\mid|X|>M$.
Consider testing $H_0:\theta=0$ versus $H_1:\theta=\theta_1>0$. Follow the Neyman--Pearson Lemma, the most powerful test is with rejection region of the form
\begin{equation*}
\left\{(X_i, Y_i)_{i=1}^n:\prod_{i=1}^n(\ind_{\{Y_i=1\}}g_{\theta_1}(X_i)+\ind_{\{Y_i=0\}}g_{\theta_1}(-X_i))\ge c_\alpha\right\}.
\end{equation*}
For simplicity, let $Z_i=2Y_i-1$ be the symmetric version of $Y_i$, then the rejection region is
\begin{equation*}
\left\{(X_i, Z_i)_{i=1}^n:\prod_{i=1}^ng_{\theta_1}(Z_iX_i)\ge c_\alpha\right\}.
\end{equation*}
Since $c_\alpha$ goes to $1$ as $\alpha\to0$, there is sufficiently small $\alpha$ such that $c_\alpha>g_{\theta_1}(M)$. For this $\alpha$, it is clear that
\begin{equation*}
\{x:\prod_{i=1}^ng_{\theta_1}(z_ix_i)\ge c_\alpha\}\subseteq\{x:|x_i|>M\},
\end{equation*}
for any fixed binary $\pm1$ sequence $z_1,z_2,\dots,z_n$. In this region, $f$ is symmetric, so this most powerful test has the correct size $\alpha$ conditional on $Z$. Put another way, the unique most powerful test is indeed a valid model-X test.

What if we restrict ourselves to fixed-X tests? Due to the discrete nature of this problem, the optimal fixed-X test will involve a randomization step for every level $\alpha\in(0,1)$ except for a finite number of values. Thus, for almost every $\alpha$, the most powerful fixed-X test is not the most powerful test.
\section{Testability of alternative sequences}
\label{sec:local}

\subsection{Contiguity and q.m.d.}
\label{sec:contiguity-qmd}
We wish to first note that it is not true that if the alternative sequence is not contiguous to the null then there must exist a test with power converges to one. If the dimension can be fixed, a simple counterexample is $\Unif[0,1]$ versus $\Unif[1/2,3/2]$. If we require them to be the measure on $n$ i.i.d. samples, then let $P_0=\Unif[0,1]^n$ and $P_n=\Unif[0,1+1/n]^n$. Obviously, the event $A_n=\{\max_{1\le i\le n}|X_i|>1\}$ has probability $0$ under $P_0$, but probability $1-(1+1/n)^{-n}\to1-e^{-1}$ under $P_n$. So $P_n$ is not contiguous to $P_0$. The most powerful level-$\alpha$ test is to reject when $\max_{1\le i\le n}|X_i|>1$ and reject with probability $\alpha$ if $\max|X_i|\le 1$. The power under $P_n$ is
\begin{equation*}
1-\frac{1}{(1+1/n)^n}+\alpha\times\frac{1}{(1+1/n)^n}\to1-\frac{1-\alpha}e<1.
\end{equation*}

Now we present some background on contiguity and q.m.d.
\begin{defi}[Contiguity, \citet{lehmann2006testing}]
\label{definition:contiguity}
Let $P_n$ and $Q_n$ be probability distributions on $(\mathcal X_n,\mathcal F_n)$. The sequence $\{Q_n\}$ is contiguous to the sequence $\{P_n\}$ if $P_n(E_n)\to0$ implies $Q_n(E_n)\to0$ for every sequence $\{E_n\}$ with $E_n\in\mathcal F_n$. If $\{Q_n\}$ is contiguous to $\{P_n\}$ and vice versa, we say $\{P_n\}$ and $\{Q_n\}$ are contiguous.
\end{defi}
\begin{lemma}[\citet{lehmann2006testing}]
Let $\{P_\theta,\theta\in\Omega\}$ with $\Omega$ being an open subset of $\rr^k$ be \emph{quadratic mean differentiable} (q.m.d.) with densities $p_\theta(\cdot)$. Then for a fixed $h$, $P^n_{\theta_0+hn^{-1/2}}$ and $P_{\theta_0}^n$ are contiguous.
\label{lemma:qmd-contiguity}
\end{lemma}

\subsection{Total variation distance}
\label{sec:TV}
Let $H_1:P\in\mathcal P_{1,n}$ be alternatives against $H_0: P\in\mathcal P_{0,n}$, with possibly growing dimensions. The problem is untestable (i.e., every level-$\alpha$ test has power bounded by $\alpha$) if \citep{romano2004non}
\begin{equation*}
\inf_{P_0\in\mathcal P_{0,n},P_1\in\mathcal P_{1,n}}\TV(P_0,P_1)=0.
\end{equation*}
Thus, if
\begin{equation*}
\lim_{n\to\infty}\inf_{P_0\in\mathcal P_{0,n},P_1\in\mathcal P_{1,n}}\TV(P_0,P_1)=0,
\end{equation*}
the sequence of alternatives is indistinguishable from the null.

To examine the converse, if the total variation distance is lower bounded away from zero, there could still be no test that has non-trivial power against all alternatives. For example, if $\mathcal P_{0,n}=\{\Unif[0,1]\}$ and $\mathcal P_{1,n}=\{\Unif[0,1/2],\Unif[1/2,1]\}$. For any test $\psi$ which rejects with probability $\psi(x)$ if $x$ is observed,
\begin{equation*}
\int_{[0,1]}\psi(x)\di x\le\alpha.
\end{equation*}
This test cannot have non-trivial power for both alternatives, because at least one inequality holds in
\begin{equation*}
\int_{[0,1/2]}\psi(x)\di x\le\alpha/2,\quad \int_{[1/2,1]}\psi(x)\di x\le\alpha/2.
\end{equation*}

Another more non-trivial example is testing $n=0$ against $n\ge1$ in
\begin{equation*}
p_n(x)=1+\sin(2n\pi x),x\in[0,1], n\in\mathbb N.
\end{equation*}
It is easy to calculate that
\begin{equation*}
\TV(p_0,p_n)=2/\pi,n>1.
\end{equation*}
But any test level-$\alpha$ test $\psi$ will satisfy
\begin{equation*}
\int_{[0,1]}\psi(x)(1+\sin(2n\pi x))\di x\le\alpha+\int_{[0,1]}\psi(x)\sin(2n\pi x)\di x\to\alpha
\end{equation*}
as $n\to\infty$ by Riemann--Lebesgue Lemma.

\section{Proofs}
\label{sec:proofs}

\begin{lemma}
Assume $X\ci Y\mid Z$. Let
\begin{equation*}
\hat R_n(t)=\p(T_n(\ndata{X},\ndata{Y},\ndata{Z})\le t\mid \ndata{Y},\ndata{Z}).
\end{equation*}
Let ${\tilde {\ndata{X}}}$ be a conditionally independent copy of $\ndata{X}$ given $\ndata{Y}$ and $\ndata{Z}$.
If
\begin{equation}
(T_n(\ndata{X},\ndata{Y},\ndata{Z}),T_n({\tilde {\ndata{X}}},\ndata{Y},\ndata{Z}))\cid(T,\tilde T),
\label{eq:convergence}
\end{equation}
where $T$ and $\tilde T$ are independent with CDF $R(\cdot)$. Then for every $t$ which is a continuity point of $R(\cdot)$, we have
\begin{equation}
\hat R_n(t)\cip R(t).
\end{equation}
\label{lemma:asy-cond}
\end{lemma}

\begin{proof}[Proof of Lemma~\ref{lemma:asy-cond}]
Let $t$ be a continuity point of $R(\cdot)$. By equation~\eqref{eq:convergence}
\begin{equation*}
    \e[\hat R_n(t)]=\p(T_n(\ndata{X},\ndata{Y},\ndata{Z})\le t)\to R(t)
\end{equation*}
Now it suffices to show that
\begin{equation*}
    \Var[\hat R_n(t)]\to0.
\end{equation*}
This is equivalent to
\begin{equation*}
    \e[\hat R_n(t)^2]\to R(t)^2.
\end{equation*}
Note that
\begin{equation*}
\begin{aligned}
\hat R_n(t)^2&=\p(T_n(\ndata{X},\ndata{Y},\ndata{Z})\le t\mid \ndata{Y},\ndata{Z})^2\\
&=\p(T_n(\ndata{X},\ndata{Y},\ndata{Z})\le t,T_n({\tilde {\ndata{X}}},{{\ndata{Y}}},\ndata{Z})\le t\mid \ndata{Y},\ndata{Z}),
\end{aligned}
\end{equation*}
hence, also by equation~\eqref{eq:convergence},
\begin{equation*}
\begin{aligned}
\e[\hat R_n(t)^2]&=\e[\p(T_n(\ndata{X},\ndata{Y},\ndata{Z})\le t,T_n({\tilde {\ndata{X}}},{{\ndata{Y}}},\ndata{Z})\le t\mid \ndata{Y},\ndata{Z})]\\
&=\p(T_n(\ndata{X},\ndata{Y},\ndata{Z})\le t,T_n({\tilde {\ndata{X}}},{{\ndata{Y}}},\ndata{Z})\le t)\to\p(T\le t,\tilde T\le t)=R(t)^2.
\end{aligned}
\end{equation*}
\end{proof}

\begin{lemma}
Let
\begin{equation*}
\hat R_n(t)=\p(T_n(\ndata{X},\ndata{Y},\ndata{Z})\le t\mid \ndata{Y},\ndata{Z}).
\end{equation*}
Suppose for every $t$ which is a continuity point of a CDF $R(\cdot)$, we have
\begin{equation}
\hat R_n(t)\cip R(t).
\end{equation}
Let $r(1-\alpha)=\inf\{t:R(t)\ge1-\alpha\}$; suppose $R(\cdot)$ is continuous and strictly increasing at $r(1-\alpha)$, then
\begin{equation*}
\hat r_n(1-\alpha)\cip r(1-\alpha).
\end{equation*}
\label{lemma:asy-eq}
\end{lemma}
\begin{proof}[Proof of Lemma~\ref{lemma:asy-eq}]
This is a direct consequence of Lemma 11.2.1 (ii) in \citet{lehmann2006testing}.
\end{proof}

\noindent\textbf{Theorem~\ref{theorem:asym-power-local-alt}.} \textit{
The asymptotic unconditional power of the test in Appendix~\ref{sec:scalar-CRT} under the local alternatives is
\begin{equation*}
1-\Phi\left(z_{1-\alpha}-h\sqrt{\Var_{H_0}(s(X_i,Y_i,Z_i,\theta_0))}\Corr_{H_0}\left({\psi(X_i,Y_i,Z_i)-e_0(Y_i,Z_i)},s(X_i,Y_i,Z_i,\theta_0)\right)\right),
\end{equation*}
where,
\begin{equation*}
e_0(Y,Z)=\e_{H_0}[\psi(X,Y,Z)\,|\,Y,Z]
\end{equation*}
and $s$ is the score function, which, under very general regularity conditions,\footnote{See Theorem~12.2.1 in \citet{lehmann2006testing} for an example of such conditions. There, the notation $\tilde\eta$ is used instead of $s$.} admits the common form
\begin{equation*}
s(X,Y,Z,\theta_0)=\frac{\frac\partial{\partial\theta}\big|_{\theta=\theta_0}p_{\theta}(X,Y,Z)}{p_{\theta_0}(X,Y,Z)},
\end{equation*}
where $p_\theta$ is the density of $P_\theta$.
}
\begin{proof}[Proof of Theorem~\ref{theorem:asym-power-local-alt}]
Consider an asymptotically linear statistic
\begin{equation*}
T_n(\ndata{X},\ndata{Y},\ndata{Z})=n^{-1/2}\sum_{i=1}^n\psi(X_i,Y_i,Z_i)+o_{\p_{H_0}}(1),
\end{equation*}
Suppose we know the direction of the alternative and thus would like a test that rejects when $T_n$ is above a threshold. Since the test is to be valid conditional on $(\ndata{Y},\ndata{Z})$, it would be equivalent to consider the statistic
\begin{equation*}
S_n(\ndata{X},\ndata{Y},\ndata{Z})=T_n-n^{-1/2}\sum_{i=1}^ne_0(Y_i,Z_i)=n^{-1/2}\sum_{i=1}^n(\psi(X_i,Y_i,Z_i)-e_0(Y_i,Z_i))+o_{\p_{H_0}}(1),
\end{equation*}
where
\begin{equation*}
e_0(y,z)=\e_{H_0}[\psi(X_i,Y_i,Z_i)\mid Y_i=y,Z_i=z].
\end{equation*}
Under the null,
\begin{equation*}
S_n\cid\mathcal N\left(0,\e_{H_0}[v_0(Y,Z)]\right),
\end{equation*}
where
\begin{equation*}
v_0(y,z)=\Var_{H_0}[\psi(X_i,Y_i,Z_i)\mid Y_i=y,Z_i=z].
\end{equation*}
In addition, note that if ${\tilde {\ndata{X}}}$ is a copy of $\ndata{X}$ conditionally independent of $\ndata Y$ given $\ndata Z$ (as in Lemma~\ref{lemma:asy-cond}), then
\begin{equation*}
\begin{aligned}
&\Cov_{H_0}(\psi(X_i,Y_i,Z_i)-e_0(Y_i,Z_i),\psi(\tilde X_i, Y_i,Z_i)-e_0(Y_i,Z_i)) \\
&\hspace{1.5cm}=\e_{H_0}[\Cov_{H_0}(\psi(X_i,Y_i,Z_i)-e_0(Y_i,Z_i),\psi(\tilde X_i,Y_i,Z_i)-e_0(Y_i,Z_i)\mid Y_i,Z_i)]\\
&\hspace{2.0cm}+\Cov_{H_0}(\e_{H_0}[\psi(X_i,Y_i,Z_i)-e_0(Y_i,Z_i)\mid Y_i,Z_i],\,\e_{H_0}[\psi(\tilde X_i,Y_i,Z_i)-e_0(Y_i,Z_i)\mid Y_i,Z_i])\\
&\hspace{1.5cm}=0+0=0.
\end{aligned}
\end{equation*}
By the bivariate central limit theorem, under $H_0$,
\begin{equation*}
\left(
  \begin{array}{c}
S_n(\ndata{X},\ndata{Y},\ndata{Z})\\
S_n(\tilde{\ndata{X}},\ndata{Y},\ndata{Z})\\
  \end{array}
\right)\cid\mathcal N\left(\left(
  \begin{array}{c}
0\\
0\\
  \end{array}
\right),\left[
  \begin{array}{cc}
\e_{H_0}[v_0(Y,Z)] & 0\\
0 & \e_{H_0}[v_0(Y,Z)]\\
  \end{array}
\right]\right).
\end{equation*}
The test $\phi_n$ rejects when $S_n>\hat r_n(1-\alpha)$, accepts when $S_n<\hat r_n(1-\alpha)$, and possibly randomizes when $S_n=\hat r_n(1-\alpha)$. By Lemma~\ref{lemma:asy-eq}, $\hat r_n(1-\alpha)\cip z_{1-\alpha}\sqrt{\e_{H_0}[v_0(Y,Z)]}$ under $H_0$.

Since the null distribution $\p_{H_0}=P^n_{\theta_0}$ and the alternative is a sequence $P^n_{\theta_0+hn^{-1/2}}$, where the family is q.m.d., by contiguity (Lemma~\ref{lemma:qmd-contiguity}), $\hat r_n(1-\alpha)\cip z_{1-\alpha}\sqrt{\e_{H_0}[v_0(Y,Z)]}$ under the alternative sequence as well. To study the asymptotic power under local alternatives, we introduce Le Cam's Third Lemma.
\begin{lemma}[Le Cam's Third Lemma, Corollary~12.3.2 in \citet{lehmann2006testing}]
\label{lemma:lecamthird}
If
\begin{equation*}
\left(
  \begin{array}{c}
X_n\\
\log\frac{\di{Q_n}}{\di{P_n}}\\
  \end{array}
\right)\cid\mathcal N\left(\left(
  \begin{array}{c}
\mu_1\\
\mu_2\\
  \end{array}
\right),\left[
  \begin{array}{cc}
\sigma_1^2 & \sigma_{1,2}\\
\sigma_{1,2} & \sigma_2^2\\
  \end{array}
\right]\right)\text{ under }P_n,
\end{equation*}
where $\frac{\di{Q_n}}{\di{P_n}}$ is the likelihood ratio and $\mu_2=-\sigma_2^2/2$ so that $Q_n$ is contiguous to $P_n$, then
\begin{equation*}
X_n\cid\mathcal N(\mu_1+\sigma_{1,2},\sigma_1^2)\text{ under }Q_n.
\end{equation*}
\end{lemma}
By taking $X_n$ to be $S_n$, $P_n$ to be $P_{\theta_0}^n$ and $Q_n$ to be $P^n_{\theta_0+hn^{-1/2}}$ in Lemma~\ref{lemma:lecamthird}, under $P^n_{\theta_0+hn^{-1/2}}$, $S_n\cid\mathcal N(\sigma_{1,2},\e_{H_0}[v_0(Y,Z)])$ ($\log\frac{\di{Q_n}}{\di{P_n}}$ is asymptotically the score; see Example 12.3.8 in \citet{lehmann2006testing}), where
\begin{equation*}
\sigma_{1,2}=h\Cov_{H_0}(\psi(X_i,Y_i,Z_i)-e_0(Y_i,Z_i),s(X_i,Y_i,Z_i,\theta_0)).
\end{equation*}
The asymptotic power is thus
\begin{equation*}
1-\Phi\left(z_{1-\alpha}-\frac{\sigma_{1,2}}{\sqrt{\e_{H_0}[v_0(Y,Z)]}}\right).
\end{equation*}
\end{proof}

\noindent\textbf{Theorem~\ref{theorem:prop-power-mc}.} \textit{In Setting~\ref{model:moderate-dim-lr}, the 
CRT with $T_\textnormal{MC}$ has asymptotic power
equal to that of a $z$-test with standardized effect size
\[ \frac{h}{\sqrt{\sigma^2+v_Z^2}}. \]}
\begin{proof}[Proof of Theorem~\ref{theorem:prop-power-mc}]
We only prove the one-sided case, while the two-sided case can be dealt with almost identically.

Under the null, $T_\text{MC}(\tilde{\mathbf X},\ndata{Y},\ndata{Z})\mid\ndata{Y},\ndata{Z}\sim\mathcal N(n^{-1}\ndata{Y}^\top\ndata{Z}\xi,\|\ndata{Y}\|^2/n^2)$, so the power is
\begin{equation}
\p_{\beta=h/\sqrt n}\left(\frac1n\ndata{Y}^\top\ndata{X}\ge\frac1n \ndata{Y}^\top\ndata{Z}\xi+z_{1-\alpha}\frac{\|\ndata{Y}\|}{n}\right)=\p_{\beta=h/\sqrt n}\left(\frac{1}{\sqrt n}\ndata{Y}^\top(\ndata{X}-\ndata{Z}\xi)\ge z_{1-\alpha}\frac{\|\ndata{Y}\|}{\sqrt n}\right).
\label{equation:mc-power}
\end{equation}

The elements of $\ndata{X}-\ndata{Z}\xi$ are conditionally independent given $\ndata{Y}$ and the distribution $\mathcal L(X-Z^\top\xi\,|\,Y)$ is (by applying the conditional distribution formula to the bivariate Gaussian distribution of $(X-Z^\top\xi,Y)$)
\begin{equation*}
\mathcal N\left(\frac{\beta Y}{(\theta+\beta\xi)^\top\Sigma(\theta+\beta\xi)+\beta^2+\sigma^2},1-\frac{\beta^2}{(\theta+\beta\xi)^\top\Sigma(\theta+\beta\xi)+\beta^2+\sigma^2}\right).
\end{equation*}
Thus,
\begin{multline*}
\frac{1}{\sqrt n}\ndata{Y}^\top(\ndata{X}-\ndata{Z}\xi)\mid\ndata{Y}\\
\sim\mathcal N\left(\frac{n^{-1/2}\beta\|\ndata{Y}\|^2}{(\theta+\beta\xi)^\top\Sigma(\theta+\beta\xi)+\beta^2+\sigma^2},\frac{\|\ndata{Y}\|^2}n\left(1-\frac{\beta^2}{(\theta+\beta\xi)^\top\Sigma(\theta+\beta\xi)+\beta^2+\sigma^2}\right)\right)
\end{multline*}
and
\begin{equation*}
\begin{aligned}
&\p_{\beta=h/\sqrt n}\left(\frac{1}{\sqrt n}\ndata{Y}^\top(\ndata{X}-\ndata{Z}\xi)\ge z_{1-\alpha}\frac{\|\ndata{Y}\|}{\sqrt n}\right)\\
&=\e\left[\p_{\beta=h/\sqrt n}\left(\frac{1}{\sqrt n}\ndata{Y}^\top(\ndata{X}-\ndata{Z}\xi)\ge z_{1-\alpha}\frac{\|\ndata{Y}\|}{\sqrt n}\,|\,\ndata{Y}\right)\right]\\
&=\e\left[\Phi\left(\frac{\frac{n^{-1}h\|\ndata{Y}\|^2}{(\theta+h\xi/\sqrt n)^\top\Sigma(\theta+h\xi/\sqrt n)+h^2/n+\sigma^2}-z_{1-\alpha}\frac{\|\ndata{Y}\|}{\sqrt n}}{\frac{\|\ndata{Y}\|}{\sqrt n}\sqrt{1-\frac{h^2/n}{(\theta+h\xi/\sqrt n)^\top\Sigma(\theta+h\xi/\sqrt n)+h^2/n+\sigma^2}}}\right)\right]\\
&\to\Phi\left(\frac{h}{\sqrt{v_Z^2+\sigma^2}}-z_{1-\alpha}\right),
\end{aligned}
\end{equation*}
where we used $\|\ndata{Y}\|^2/n\cip v_Z^2+\sigma^2$. 
To see why this is the case, note that
\begin{equation*}
Y_i\simiid\mathcal N\left(0,\frac{h^2}{n}+(\theta+h\xi/\sqrt n)^\top\Sigma_Z(\theta+h\xi/\sqrt n)+\sigma^2\right),
\end{equation*}
so we only need to show
\begin{equation}
    (\theta+h\xi/\sqrt n)^\top\Sigma_Z(\theta+h\xi/\sqrt n)\to v_Z^2.
    \label{equation:mc-to-show}
\end{equation}
Equation \eqref{equation:mc-to-show} holds because by assumption, $\theta^\top\Sigma_Z\theta\to v_Z^2$, $n^{-1}\xi^\top\Sigma_Z\xi\to0$, and by the Cauchy--Schwarz inequality, the cross term satisfies
\begin{equation*}
\theta^\top\Sigma_Z\xi/\sqrt n\le\sqrt{n^{-1} \theta^\top\Sigma_Z\theta\cdot\xi^\top\Sigma_Z\xi}\to0.
\end{equation*}

\end{proof}

\noindent\textbf{Theorem~\ref{theorem:prop-model-X-CRT-OLS}.} \textit{In Setting~\ref{model:moderate-dim-lr} with $\kappa<1$, the 
CRT with $T_\text{OLS}$ has asymptotic power
equal to that of a $z$-test with standardized effect size
\[ \frac{h}{\sigma}\sqrt{1-\kappa}. \]}

\begin{proof}[Proof of Theorem~\ref{theorem:prop-model-X-CRT-OLS}]
We only prove the one-sided case, while the two-sided case can be dealt with almost identically.

We look at the expression of the normalized OLS statistic $T_\text{OLS}(\ndata{X},\ndata{Y},\ndata{Z})=\sqrt n\hat\beta$:
\begin{equation}
T_\text{OLS}(\ndata{X},\ndata{Y},\ndata{Z})=\sqrt n\hat\beta=\frac{\ndata{X}^\top(I-\ndata{Z}(\ndata{Z}^\top \ndata{Z})^{-1}\ndata{Z}^\top)\ndata{Y}/\sqrt n}{\ndata{X}^\top(I-\ndata{Z}(\ndata{Z}^\top \ndata{Z})^{-1}\ndata{Z}^\top)\ndata{X}/n},
\label{equation:ols-expression}
\end{equation}
and the rejection region is $\{T_\text{OLS}\ge\hat c_\alpha\}$, where $\hat c_\alpha$ is the upper $\alpha$-quantile of the distribution of
\begin{equation*}
\tilde T_\text{OLS}(\tilde{\mathbf X},\ndata{Y},\ndata{Z})=\frac{\tilde{\mathbf X}^\top(I-\ndata{Z}(\ndata{Z}^\top \ndata{Z})^{-1}\ndata{Z}^\top)\ndata{Y}/\sqrt n}{\tilde{\mathbf X}^\top(I-\ndata{Z}(\ndata{Z}^\top \ndata{Z})^{-1}\ndata{Z}^\top)\tilde{\mathbf X}/n},
\end{equation*}
conditional on $(\ndata{Y},\ndata{Z})$. Looking at the numerator and denominator individually, we see that
\begin{align*}
\mathcal L\left(\tilde{\mathbf X}^\top(I-\ndata{Z}(\ndata{Z}^\top \ndata{Z})^{-1}\ndata{Z}^\top)\ndata{Y}/\sqrt n\mid\ndata{Y},\ndata{Z}\right)&\sim\mathcal N(0,\ndata Y^\top(I-\ndata{Z}(\ndata{Z}^\top \ndata{Z})^{-1}\ndata{Z}^\top)\ndata{Y}/n),\\
\mathcal L\left(\tilde{\mathbf X}^\top(I-\ndata{Z}(\ndata{Z}^\top \ndata{Z})^{-1}\ndata{Z}^\top)\tilde{\mathbf X}/n\mid\ndata{Y},\ndata{Z}\right)&\sim n^{-1}\chi_{n-p}^2\text{ \citep[Cochran's Theorem]{cochran1934distribution}}.
\end{align*}
Now we assume we are under the local alternative $\beta=h/\sqrt n$. Again by Cochran's Theorem,
\begin{equation*}
\ndata Y^\top(I-\ndata{Z}(\ndata{Z}^\top \ndata{Z})^{-1}\ndata{Z}^\top)\ndata{Y}/n\sim n^{-1}(\sigma^2+\beta^2)\chi_{n-p}^2.
\end{equation*}
Thus, for any $t\in\rr$,
\begin{equation*}
\p\left(
\tilde{\mathbf X}^\top(I-\ndata{Z}(\ndata{Z}^\top \ndata{Z})^{-1}\ndata{Z}^\top)\ndata{Y}/\sqrt n\le t\mid\ndata{Y},\ndata{Z}\right)\cip\Phi(t/\sqrt{\sigma^2(1-\kappa)}).
\end{equation*}
On the other hand, $\tilde{\mathbf X}^\top(I-\ndata{Z}(\ndata{Z}^\top \ndata{Z})^{-1}\ndata{Z}^\top)\tilde{\mathbf X}/n\ci(\ndata{Y},\ndata{Z})$ and for any $t\ne1-\kappa$,
\begin{equation*}
\p\left(\tilde{\mathbf X}^\top(I-\ndata{Z}(\ndata{Z}^\top \ndata{Z})^{-1}\ndata{Z}^\top)\tilde{\mathbf X}/n\le t\mid\ndata{Y},\ndata{Z}\right)\to\mathbf{1}_{\{t>1-\kappa\}}.
\end{equation*}
By Lemma~\ref{lemma:conditional-CDF}, for any $t\in\rr$,
\begin{equation*}
\p\left(\tilde T_\text{OLS}\le t\mid\ndata{Y},\ndata{Z}\right)\cip\Phi(\sqrt{1-\kappa}t/\sigma),
\end{equation*}
and by Lemma 11.2.1 (ii) in \citet{lehmann2006testing},
\begin{equation*}
\hat c_\alpha\cip z_{1-\alpha}\frac{\sigma}{\sqrt{1-\kappa}}.
\end{equation*}
On the other hand, we have the test statistic itself satisfies
\begin{equation*}
T_\text{OLS}\mid\ndata{X},\ndata{Z}\sim\mathcal N(h,\sigma^2n\hat\Omega_{11}),
\end{equation*}
where $\hat\Omega$ is the inverse of the matrix $(\ndata{X},\ndata{Z})^\top(\ndata{X},\ndata{Z})$ that follows an inverse-Wishart distribution, and then $n\hat\Omega_{11}\cip1/(1-\kappa)$ by moment calculations. Therefore, $T_\text{OLS}\cid\mathcal N(h,\sigma^2/(1-\kappa))$. It follows then
\begin{equation*}
\p_{\beta=h/\sqrt n}\left(T_\text{OLS}\ge\hat c_\alpha\right)\to1-\Phi_{\sigma^2/(1-\kappa)}(\frac{\sigma}{\sqrt{1-\kappa}}z_{1-\alpha}-h)=\Phi\left(\frac{h}{\sigma}\sqrt{1-\kappa}-z_{1-\alpha}\right),
\end{equation*}
where $\Phi_{\sigma^2/(1-\kappa)}$ is the CDF of $\mathcal N(0,\sigma^2/(1-\kappa))$.

\end{proof}

\begin{lemma}
Let $\mathcal L(X_n\,|\, Z_n)$ have random CDF $F_n$ and $\mathcal L(Y_n\,|\, Z_n)$ have deterministic CDF $G_n$ (in other words, $Y_n\ci Z_n$). Let $\mathcal L(Y_n\,|\, Z_n)$ converge in distribution to a point mass at $c$, $c>0$, and for a continuous and deterministic CDF $F$ on $\rr$, let $F_n(t)\cip F(t)$ for any $t\in\rr$. Let $H_n$ be the CDF of $\mathcal L(X_nY_n\,|\,Z_n)$. Then for any $t\in\rr$, $H_n(t)\cip F(t/c)$.
\label{lemma:conditional-CDF}
\end{lemma}
\begin{proof}[Proof of Lemma~\ref{lemma:conditional-CDF}]
Without loss of generality, assume $c=1$. Fix $t\in\rr$ and $\varepsilon>0$. Pick $\delta>0$ such that $|F(t)-F(t/(1+\delta))|\le\varepsilon/2$.
\begin{equation*}
\begin{aligned}
H_n(t)&=\p\left(X_nY_n\le t\mid Z_n\right)\\
&\ge\p\left(X_n\le\frac{t}{1+\delta},Y_n\le1+\delta\mid Z_n\right)\\
&=\p\left(\left\{X_n\le\frac{t}{1+\delta}\right\}\setminus\left\{Y_n>1+\delta\right\}\mid Z_n\right)\\
&\ge\p\left(X_n\le\frac{t}{1+\delta}\mid Z_n\right)-\p\left(Y_n>1+\delta\mid Z_n\right)\\
&=F_n(t/(1+\delta))-(1-G_n(1+\delta))\cip F(t/(1+\delta)).
\end{aligned}
\end{equation*}
It follows that $\p(H_n(t)\ge F(t/(1+\delta))-\varepsilon/2)\to1$. By the choice of $\delta$, $\p(H_n(t)\ge F(t)-\varepsilon)\to1$. Similarly, we can get $\p(H_n(t)\le F(t)+\varepsilon)\to1$, thus proving the claim.
\end{proof}

\noindent\textbf{Theorem~\ref{theorem:prop-distilled-power}.} \textit{Under Setting~\ref{model:moderate-dim-lr} with $\Sigma_Z=I$ and $\xi=0$, if the empirical distribution of $(\sqrt n\theta_j)_{j=1}^{p-1}$ converges to a distribution represented by a random variable $B_0$ and $\|\sqrt n\theta\|_2^2/p\to\e[B_0^2]$,
then the 
CRT with the distilled lasso statistic with lasso parameter $\lambda$ has asymptotic power
equal to that of a $z$-test with standardized effect size
\[ \frac h{\tau_\lambda}. \]}

\begin{proof}[Proof of Theorem~\ref{theorem:prop-distilled-power}]

To use the results in \citet{bayati2011lasso}, we apply the following re-normalization: assume $(\ndata{X},\ndata{Z})$ is divided by $\sqrt n$, $(\beta,\theta)$ is multiplied by $\sqrt n$, and the statistic is $T_\text{distilled}(\ndata{X},\ndata{Y},\ndata{Z})=(\ndata{Y}-\ndata{Z}\hat\theta_\lambda)^\top\ndata{X}$. The proof is a direct consequence of Lemma~\ref{lemma:theorem-training-loss}.
\end{proof}

\begin{lemma}
Assume Setting~\ref{model:moderate-dim-lr} with $\Sigma_Z=I$, $\xi=0$, $\sqrt n\beta$ universally bounded (but not necessarily a constant), and $\varepsilon_i$'s and $X_i$'s do not change with $n,p$ as long as $n\ge i$. If the empirical distribution of $(\sqrt n\theta_j)_{j=1}^{p-1}$ converges to a distribution represented by a random variable $B_0$ and $\|\sqrt n\theta\|_2^2/p\to\e[B_0^2]$, then we have
\begin{equation}
    \frac{\|\ndata{Y}-\ndata{Z}\hat\theta_\lambda\|_2^2}{n}\cas\frac{\lambda^2}{\alpha_\lambda^2},
    \label{equation:AMP-variance}
\end{equation}

\begin{equation}
\frac1n(\ndata{Y}-\ndata{Z}\hat\theta_\lambda)^\top(\ndata{Y}-\ndata{Z}\theta)\cas\frac{\lambda}{\alpha_\lambda\tau_\lambda}\sigma^2,
\label{equation:AMP-signal}
\end{equation}
and
\begin{equation*}
T_\textnormal{distilled}(\ndata{X},\ndata{Y},\ndata{Z})-\frac{\sqrt n\beta\lambda}{\alpha_\lambda\tau_\lambda}\cid\mathcal N\left(0,\frac{\lambda^2}{\alpha_\lambda^2}\right).
\end{equation*}
Here, $\alpha_\lambda$ and $\tau_\lambda$ satisfy
\begin{equation}
\begin{aligned}
\lambda&=\alpha_\lambda\tau_\lambda\left(1-\kappa\e[\eta'(B_0+\tau_\lambda W;\alpha_\lambda\tau_\lambda)]\right),\\
\tau_{\lambda}^2&=\sigma^2+\kappa\e[(\eta(B_0+\tau_\lambda W;\alpha_\lambda\tau_\lambda)-B_0)^2],
\end{aligned}
\label{equation:amp-fixed-point-equation}
\end{equation}
where $W\sim\mathcal N(0,1)$ is independent of $B_0$ and $\eta'$ is the derivative of $\eta$.
\label{lemma:theorem-training-loss}
\end{lemma}

\begin{proof}[Proof of Lemma~\ref{lemma:theorem-training-loss}]
We assume $\varepsilon_i$'s and $X_i$'s do not change with $n,p$ to satisfy Definition~1 (b) in \citet{bayati2011lasso} by
\begin{equation*}
\|\varepsilon\|_2^2/n\cas\sigma^2\text{ and }\|\beta\ndata{X}\|_2^2/n\cas0.
\end{equation*}
This additional assumption on $\varepsilon_i$'s and $X_i$'s does not change the asymptotic power; in fact, it does not change the power for any fixed pair of $(n,p)$, because the power is a marginal quantity for each pair of $(n,p)$ and does not depend on the relationship of the random variables between different pairs of $(n,p)$'s.

To use the results in \citet{bayati2011lasso}, we again apply the following re-normalization: assume $(\ndata{X},\ndata{Z})$ is divided by $\sqrt n$, $(\beta,\theta)$ is multiplied by $\sqrt n$, and the statistic is $T_\text{distilled}(\ndata{X},\ndata{Y},\ndata{Z})=(\ndata{Y}-\ndata{Z}\hat\theta_\lambda)^\top\ndata{X}$.

Note that in the $\ndata Y$ against $\ndata Z$ regression, we can absorb $\ndata X$ into the error and under the assumption that $\beta$ stays universally bounded, the effective error
\begin{equation*}
    \varepsilon'=\ndata{Y}-\ndata{Z}\theta=\varepsilon+\beta\ndata{X}
\end{equation*}
still has the property that its empirical distribution converges to $\mathcal N(0,\sigma^2)$ and its second moment converges to $\sigma^2$. We first prove \eqref{equation:AMP-variance}. The AMP iteration is
\begin{equation*}
\begin{aligned}
\theta^{t+1}&=\eta(\ndata{Z}^\top z^t+\theta^t;\alpha_\lambda\tau_t),\\
z^t&=\ndata{Y}-\ndata{Z}\theta^t+\kappa z^{t-1}\langle\eta'(\ndata{Z}^\top z^{t-1}+\theta^{t-1};\alpha_\lambda\tau_{t-1})\rangle\\
\tau_{t+1}^2&=\sigma^2+\kappa\e[(\eta(B_0+\tau_t W;\alpha_\lambda\tau_t)-B_0)^2],
\end{aligned}
\end{equation*}
where $\eta'$ is the derivative of $\eta$ and $\langle\cdot\rangle$ means taking the average of the coordinates of a vector. We denote
\begin{equation*}
w_t=\kappa\langle\eta'(\ndata{Z}^\top z^{t-1}+\theta^{t-1};\alpha_\lambda\tau_{t-1})\rangle.
\end{equation*}

We first see that by the reverse triangle inequality,
\begin{equation*}
\left|\frac{\|\ndata{Y}-\ndata{Z}\hat\theta\|_2}{\sqrt n}-\frac{\|\ndata{Y}-\ndata{Z}\theta^t\|_2}{\sqrt n}\right|\le\frac{\|\ndata{Z}(\theta^t-\hat\theta)\|_2}{\sqrt n}.
\end{equation*}
Note that
\begin{equation*}
\frac{\|\ndata{Z}(\theta^t-\hat\theta)\|_2^2}{n}\le\frac{\sigma_{\max{}}^2(\ndata{Z})\|\theta^t-\hat\theta\|_2^2}{n},
\end{equation*}
where $\sigma_{\max{}}(\ndata{Z})$ is almost surely bounded (see, e.g., Theorem~F.2 in \citet{bayati2011lasso}) and Theorem~1.8 in \citet{bayati2011lasso} states that
\begin{equation*}
\lim_{t\to\infty}\lim_{n\to\infty}\frac{\|\theta^t-\hat\theta\|_2^2}{n}=0\text{ almost surely}.
\end{equation*}
Thus,
\begin{equation*}
\lim_{t\to\infty}\lim_{n\to\infty}\frac{\|\ndata{Z}(\theta^t-\hat\theta)\|_2^2}{n}=0\text{ almost surely}.
\end{equation*}
Now we just have to show
\begin{equation}
\lim_{t\to\infty}\lim_{n\to\infty}\frac{\|\ndata{Y}-\ndata{Z}\theta^t\|_2^2}{n}=\frac{\lambda^2}{\alpha_\lambda^2}\text{ almost surely},
\label{equation:amp-residual}
\end{equation}
which will prove \eqref{equation:AMP-variance}. By definition, $\ndata{Y}-\ndata{Z}\theta^t=z^t-z^{t-1}w_t$. By the reverse triangle inequality,
\begin{equation*}
\left|\frac{\|z^t-z^{t-1}w_t\|_2}{\sqrt n}-\frac{\|z^{t-1}(1-w_t)\|_2}{\sqrt n}\right|\le\frac{\|z^t-z^{t-1}\|_2}{\sqrt n},
\end{equation*}
and the right hand side goes to $0$ as stated by Lemma~4.3 in \citet{bayati2011lasso}. Thus, to prove \eqref{equation:amp-residual}, we can just analyze the limit of $\|z^{t-1}(1-w_t)\|_2^2/n$. Directly by Lemma~4.1 in \citet{bayati2011lasso},
\begin{equation*}
\lim_{t\to\infty}\lim_{n\to\infty}\frac1n\|z^{t-1}\|_2^2=\lim_{t\to\infty}\tau_t^2=\tau_\lambda^2\text{ almost surely}.
\end{equation*}
Almost surely,
\begin{equation}
\begin{aligned}
\lim_{t\to\infty}\lim_{n\to\infty}w_t&=\lim_{t\to\infty}\kappa\e[\eta'(B_0+\tau_{t-1}W;\alpha_\lambda\tau_{t-1})]&\text{(Equation~(4.11) in \citet{bayati2011lasso})}\\
&=\kappa\e[\eta'(B_0+\tau_\lambda W;\alpha_\lambda\tau_\lambda)]&\text{(bounded convergence theorem)}\\
&=1-\frac\lambda{\alpha_\lambda\tau_\lambda}.&\text{(definition of $\alpha_\lambda$ and $\tau_\lambda$, Equation~\eqref{equation:amp-fixed-point-equation})}
\end{aligned}
\label{equation:wt-limit}
\end{equation}
Combining the above results, we get
\begin{equation*}
\lim_{t\to\infty}\lim_{n\to\infty}\frac{\|z^{t-1}(1-w_t)\|_2^2}{n}=\lim_{t\to\infty}\lim_{n\to\infty}(1-w_t)^2\lim_{t\to\infty}\lim_{n\to\infty}\frac1n\|z^{t-1}\|_2^2=\frac{\lambda^2}{\alpha_\lambda^2}\text{ almost surely}.
\end{equation*}

Now we prove \eqref{equation:AMP-signal}. By (F.12) in Lemma F.3(d) in \citet{bayati2011lasso} (take $\varphi(u,v)=v$, take their $r$ and $s$ to both be $t$, their $w$ is our $\varepsilon'$ and their $b^t$ is our $\varepsilon'-z^t$),
\begin{equation*}
\lim_{n\to\infty}\langle \varepsilon'-z^t,\varepsilon'\rangle=0\Rightarrow\sigma^2=\langle\varepsilon',\varepsilon'\rangle=\lim_{n\to\infty}\langle \varepsilon',z^t\rangle\text{ almost surely}.
\end{equation*}
Thus,
\begin{equation*}
\langle\varepsilon',\ndata{Y}-\ndata{Z}\theta^t\rangle=\langle \varepsilon',z^t- z^{t-1}w^t\rangle\cas\sigma^2-\sigma^2w^t\text{ as }n\to\infty.
\end{equation*}
Combining the above equation with \eqref{equation:wt-limit}, we see that
\begin{equation*}
\lim_{t\to\infty}\lim_{n\to\infty}\langle\varepsilon',\ndata{Y}-\ndata{Z}\theta^t\rangle=\sigma^2\frac{\lambda}{\alpha_\lambda\tau_\lambda}\text{ almost surely}.
\end{equation*}
What we are interested in is the limit of $\langle\varepsilon',\ndata{Y}-\ndata{Z}\hat\theta\rangle$ as $n\to\infty$. Note that
\begin{equation*}
\langle\varepsilon',\ndata{Y}-\ndata{Z}\theta^t\rangle-\langle\varepsilon',\ndata{Y}-\ndata{Z}\hat\theta\rangle=\langle\varepsilon',\ndata{Z}(\hat\theta-\theta^t)\rangle.
\end{equation*}
By the Cauchy--Schwartz inequality,
\begin{equation*}
|\langle\varepsilon',\ndata{Z}(\hat\theta-\theta^t)\rangle|\le\sqrt{\frac{\|\varepsilon'\|_2^2}{n}\frac{\|\ndata{Z}(\theta^t-\hat\theta)\|_2^2}{n}}.
\end{equation*}
Since $\|\varepsilon'\|_2^2/n\cas\sigma^2$ and we have showed $\|\ndata{Z}(\theta^t-\hat\theta)\|_2^2/n\cas0$ (as $n\to\infty$ then $t
\to\infty$), this means
\begin{equation*}
\lim_{n\to\infty}\langle\varepsilon',\ndata{Y}-\ndata{Z}\hat\theta\rangle=\lim_{t\to\infty}\lim_{n\to\infty}\langle\varepsilon',\ndata{Y}-\ndata{Z}\theta^t\rangle=\sigma^2\frac{\lambda}{\alpha_\lambda\tau_\lambda}\text{ almost surely}.
\end{equation*}

Note that
\begin{equation*}
\ndata{X}\mid \ndata{Y},\ndata{Z}\sim\mathcal N\left(\frac{\beta}{n\sigma^2+\beta^2}\varepsilon',\frac{\sigma^2}{n\sigma^2+\beta^2}\right),
\end{equation*}
where we remind the reader that $\varepsilon'=\ndata{Y}-\ndata{Z}^\top\theta=\varepsilon+\beta\ndata{X}$. Hence,
\begin{equation*}
T_\textnormal{distilled}(\ndata{X},\ndata{Y},\ndata{Z})\mid \ndata{Y},\ndata{Z}\sim\mathcal N\left(\frac{\beta}{n\sigma^2+\beta^2}(\ndata{Y}-\ndata{Z}\hat\theta_\lambda)^\top\varepsilon',\frac{\sigma^2}{n\sigma^2+\beta^2}\|\ndata{Y}-\ndata{Z}\hat\theta_\lambda\|_2^2\right)
\end{equation*}
Now it is clear that
\begin{equation*}
T_\textnormal{distilled}(\ndata{X},\ndata{Y},\ndata{Z})-\frac{\beta\lambda}{\alpha_\lambda\tau_\lambda}\cid\mathcal N\left(0,\frac{\lambda^2}{\alpha_\lambda^2}\right).
\end{equation*}
\end{proof}

\noindent\textbf{Theorem~\ref{theorem:prop-model-X-unlabeled}.} \textit{In Setting~\ref{model:moderate-dim-lr} with $\xi$ and $\Var(X\,|\,Z)$ unknown but fixed to be $1$, if there are $m$ additional data points $(X_i,Z_i)_{i=n+1}^{n+m}$, $n_*=n+m$, $n/n_*\to\kappa_*$ and $\kappa\kappa_*<1$, then the conditional CRT with statistic $T_\textnormal{MC}$ has asymptotic power 
lower-bounded (the $\liminf$ is lower-bounded) by that of a $z$-test with standardized effect size
\[ \frac{h\sqrt{1-\kappa\kappa_*}}{\sqrt{\sigma^2+v_Z^2\frac{1}{{1-\kappa\kappa_*}}}} \]
and upper-bounded (the $\limsup$ is upper-bounded) by that of a $z$-test with standardized effect size
\[ \frac{h\sqrt{1-\kappa\kappa_*}}{\sqrt{\sigma^2+v_Z^2\max\left(0,\frac{1-\frac{(1+\sqrt{1/\kappa})^2}{(1-\sqrt{\kappa\kappa_*})^2}\kappa\kappa_*}{1-\kappa\kappa_*}\right)}}. \]}

\begin{proof}[Proof of Theorem~\ref{theorem:prop-model-X-unlabeled}]
This proof uses some results and notation in the detailed introduction of the conditional CRT in Appendix~\ref{sec:conditional-crt} and should be read after that.

We first show that in our asymptotic regime, we can assume $\Var(X\,|\,Z)$ is known. Then we analyze the asymptotic power assuming $\Var(X\,|\,Z)$ is known.
\paragraph{Knowledge of $\Var(X\,|\,Z)$.}
We show that we could assume $\Var(X\,|\,Z)$ is known by making the following claim: suppose we obtain a cutoff without knowing $\Var(X\,|\,Z)$ and another oracle cutoff with the knowledge of $\Var(X\,|\,Z)$ and then we proceed to use the two cutoffs to perform the CRT with the same test statistic. The two decisions differ if and only if the test statistic falls between the two cutoffs, and we claim the probability of this happening goes to $0$.

By the conditional nature, the following modified statistic is equivalent when used for the CRT.
\begin{equation*}
T_\text{modified}(\ndata{Y},\ndata{X}_*,\ndata{Z}_*)=\frac{T_\text{MC}^\text{ess}(\ndata{Y},\ndata{X}_*,\ndata{Z}_*)}{\sqrt{\ndata{Y}^\top I_{n\times n_*} A_{\ndata{Z}_*}A_{\ndata{Z}_*}^\top I_{n_*\times n}\ndata{Y}}}=\frac{\ndata{Y}^\top I_{n\times n_*}A_{\ndata{Z}_*}{A_{\ndata{Z}_*}^\top\varepsilon^X_*}{}}{\sqrt{\ndata{Y}^\top I_{n\times n_*} A_{\ndata{Z}_*}A_{\ndata{Z}_*}^\top I_{n_*\times n}\ndata{Y}}}.
\end{equation*}
If we know $\Var(X\,|\,Z)=1$, the test is simply $T_\text{modified}\ge z_{1-\alpha}$. When we do not know $\Var(X\,|\,Z)$, the test can be done by replacing $z_{1-\alpha}$ with the $\alpha$-upper quantile of
\begin{equation*}
\mathcal L\left(\frac{\ndata{Y}^\top I_{n\times n_*}A_{\ndata{Z}_*}\|A_{\ndata{Z}_*^\top}\varepsilon^X_*\|\frac{W}{\|W\|}}{\sqrt{\ndata{Y}^\top I_{n\times n_*} A_{\ndata{Z}_*}A_{\ndata{Z}_*}^\top I_{n_*\times n}\ndata{Y}}}\mid\ndata{Y},\ndata{Z}_*,\|A_{\ndata{Z}_*}^\top\varepsilon_*^X\|\right), W\text{ is independent }\mathcal N\left(0,I_{n_*-p}\right),
\label{equation:ols-model-x-no-tau}
\end{equation*}
which we denote by $\hat c_\alpha^n$. Evidently, we are interested in the limiting behavior of
\begin{equation*}
\p_{\beta=h/\sqrt n}\left(T_\text{modified}(\ndata{Y},\ndata{X}_*,\ndata{Z}_*)\in\left(\min(z_{1-\alpha},\hat c_\alpha^n),\max(z_{1-\alpha},\hat c_\alpha^n)\right)\right),
\end{equation*}
which we will show goes to $0$.

Since
\begin{equation*}
\mathcal L\left(\frac{\ndata{Y}^\top I_{n\times n_*}A_{\ndata{Z}_*}\frac{\|A_{\ndata{Z}_*^\top}\varepsilon^X_*\|}{\sqrt{n_*-p}}{W}{}}{\sqrt{\ndata{Y}^\top I_{n\times n_*} A_{\ndata{Z}_*}A_{\ndata{Z}_*}^\top I_{n_*\times n}\ndata{Y}}}\mid\ndata{Y},\ndata{Z}_*,\|A_{\ndata{Z}_*}^\top\varepsilon_*^X\|\right)=\mathcal N\left(0,\underbrace{\frac{\|A_{\ndata{Z}_*}^\top\varepsilon_*^X\|^2}{n_*-p}}_{\cip1}\right)
\end{equation*}
and
\begin{equation*}
\mathcal L\left(\frac{\sqrt{n_*-p}}{\|W\|}\mid\ndata{Y},\ndata{Z}_*,\|A_{\ndata{Z}_*}^\top\varepsilon_*^X\|\right)=\underbrace{\sqrt{n_*-p}\cdot\text{Inv-}\chi^2_{n_*-p}}_{\cip1},
\end{equation*}
we can use Lemma~\ref{lemma:conditional-CDF} and Lemma 11.2.1 (ii) in \citet{lehmann2006testing} to establish that $\hat c_\alpha^n$ converges to $z_{1-\alpha}$ in probability (note that by this analysis, the statement is true under both the null and the alternative sequence). Now for any $\delta>0$,
\begin{multline*}
\left\{T_\text{modified}(\ndata{Y},\ndata{X}_*,\ndata{Z}_*)\in\left(\min(z_{1-\alpha},\hat c_\alpha^n),\max(z_{1-\alpha},\hat c_\alpha^n)\right)\right\}\\
\subseteq\left\{|z_{1-\alpha}-\hat c^n_\alpha|>\delta\right\}\cup\left\{T_\text{modified}(\ndata{Y},\ndata{X}_*,\ndata{Z}_*)\in\left(z_{1-\alpha}-\delta, z_{1-\alpha}+\delta\right)\right\}.
\end{multline*}
Under $\beta=h/\sqrt n$, by calculating $\mathcal L(\varepsilon^X_*\,|\,\ndata Y,\ndata Z_*)$, we get
\begin{equation*}
\begin{aligned}
T_{\text{modified}}(\ndata{Y},\ndata{X}_*,\ndata{Z}_*)\mid \ndata{Y},\ndata{Z}_*&\sim\mathcal N(\mu_n(\ndata{Y},\ndata{Z}_*),\sigma^2_n(\ndata{Y},\ndata{Z}_*)),\text{ where}\\
\mu_n(\ndata{Y},\ndata{Z}_*)&=\frac{h}{\sigma^2+h^2/n}\frac{\ndata{Y}^\top(I-\ndata{Z}(\ndata{Z}_*^\top \ndata{Z}_*)^{-1}\ndata{Z}^\top)(\ndata{Y}-\ndata{Z}(\theta+\beta\eta))}{\sqrt{n}\sqrt{\ndata{Y}^\top(I-\ndata{Z}(\ndata{Z}_*^\top \ndata{Z}_*)^{-1}\ndata{Z}^\top)\ndata{Y}}},\\
\sigma^2_n(\ndata{Y},\ndata{Z}_*)&=\frac{\sigma^2}{\sigma^2+h^2/n}+\frac{h^2/n}{h^2/n+\sigma^2}\frac{\ndata{Y}^\top \ndata{Z}(\ndata{Z}_*^\top \ndata{Z}_*)^{-1}\ndata{Z}^\top \ndata{Y}}{\ndata{Y}^\top(I-\ndata{Z}(\ndata{Z}_*^\top \ndata{Z}_*)^{-1}\ndata{Z}^\top)\ndata{Y}}\\
&\quad\quad\quad-\frac{h^2/n}{\sigma^2+h^2/n}\frac{\ndata{Y}^\top \ndata{Z}(\ndata{Z}_*^\top \ndata{Z}_*)^{-1}\ndata{Z}^\top \ndata{Z}(\ndata{Z}_*^\top \ndata{Z}_*)^{-1}\ndata{Z}^\top \ndata{Y}}{\ndata{Y}^\top(I-\ndata{Z}(\ndata{Z}_*^\top \ndata{Z}_*)^{-1}\ndata{Z}^\top)\ndata{Y}}.
\end{aligned}
\end{equation*}
Note that
\begin{equation*}
\begin{aligned}
&\quad\ \ndata{Y}^\top \ndata{Z}(\ndata{Z}_*^\top \ndata{Z}_*)^{-1}\ndata{Z}^\top \ndata{Y}-\ndata{Y}^\top \ndata{Z}(\ndata{Z}_*^\top \ndata{Z}_*)^{-1}\ndata{Z}^\top \ndata{Z}(\ndata{Z}_*^\top \ndata{Z}_*)^{-1}\ndata{Z}^\top \ndata{Y}\\
&=\ndata{Y}^\top \ndata{Z}(\ndata{Z}_*^\top \ndata{Z}_*)^{-1}\ndata{Z}_*^\top \ndata{Z}_*(\ndata{Z}_*^\top \ndata{Z}_*)^{-1}\ndata{Z}^\top \ndata{Y}-\ndata{Y}^\top \ndata{Z}(\ndata{Z}_*^\top \ndata{Z}_*)^{-1}\ndata{Z}^\top \ndata{Z}(\ndata{Z}_*^\top \ndata{Z}_*)^{-1}\ndata{Z}^\top \ndata{Y}\\
&=\ndata{Y}^\top \ndata{Z}(\ndata{Z}_*^\top \ndata{Z}_*)^{-1}(\ndata Z_*^\top\ndata Z_*-\ndata{Z}^\top \ndata{Z})(\ndata{Z}_*^\top \ndata{Z}_*)^{-1}\ndata{Z}^\top \ndata{Y}\\
&=\ndata{Y}^\top \ndata{Z}(\ndata{Z}_*^\top \ndata{Z}_*)^{-1}\left(\sum_{i=n+1}^{n+m}\ndata Z_i\ndata{Z}_i^\top\right)(\ndata{Z}_*^\top \ndata{Z})^{-1}\ndata{Z}^\top \ndata{Y}\ge0,
\end{aligned}
\end{equation*}
where $\ndata Z_i$ is the $i$th row of $\ndata{Z}_*$ as a column vector. Therefore, we see that $\sigma^2_n(\ndata{Y},\ndata{Z}_*)\ge\sigma^2/(\sigma^2+h^2/n)$, so the conditional density of $\mathcal L(T_\text{modified}\mid\ndata{Y},\ndata{Z}_*)$ is upper bounded by $\sqrt{\sigma^2+h^2/n}/\sqrt{2\pi\sigma^2}$. Thus,
\begin{equation*}
\begin{aligned}
&\p_{\beta=h/\sqrt n}\left(T_\text{modified}(\ndata{Y},\ndata{X}_*,\ndata{Z}_*)\in\left(z_{1-\alpha}-\delta,z_{1-\alpha}+\delta\right)\right)\\
&=\e_{\beta=h/\sqrt n}\left[\p_{\beta=h/\sqrt n}\left(T_\text{modified}(\ndata{Y},\ndata{X}_*,\ndata{Z}_*)\in\left(z_{1-\alpha}-\delta, z_{1-\alpha}+\delta\right)\mid\ndata{Y},\ndata{Z}_*\right)\right]\\
&\le\e_{\beta=h/\sqrt n}\left[\frac{2\delta\sqrt{\sigma^2+h^2/n}}{\sqrt{2\pi\sigma^2}}\right]\le\frac{2\delta\sqrt{\sigma^2+h^2/n}}{\sqrt{2\pi\sigma^2}}.
\end{aligned}
\end{equation*}
We then obtain
\begin{multline*}
\p_{\beta=h/\sqrt n}\left(T_\text{modified}(\ndata{Y},\ndata{X}_*,\ndata{Z}_*)\in\left(\min(z_{1-\alpha},\hat c_\alpha^n),\max(z_{1-\alpha},\hat c_\alpha^n)\right)\right)\\
\le\underbrace{\p_{\beta=h/\sqrt n}\left(|z_{1-\alpha}-\hat c^n_\alpha|>\delta\right)}_{\to0}+\underbrace{\p_{\beta=h/\sqrt n}\left(T_\text{modified}(\ndata{Y},\ndata{X}_*,\ndata{Z}_*)\in\left(z_{1-\alpha}-\delta, z_{1-\alpha}+\delta\right)\right)}_{\le\frac{2\delta\sqrt{\sigma^2+h^2/n}}{\sqrt{2\pi\sigma^2}}\to\frac{2\delta}{\sqrt{2\pi}}}
\end{multline*}
and hence
\begin{equation*}
\limsup\p_{\beta=h/\sqrt n}\left(T_\text{modified}(\ndata{Y},\ndata{X}_*,\ndata{Z}_*)\in\left(\min(z_{1-\alpha},\hat c_\alpha^n),\max(z_{1-\alpha},\hat c_\alpha^n)\right)\right)\le\frac{2\delta}{\sqrt{2\pi}}.
\end{equation*}
Let $\delta\to0$ and the claim is proved.

\paragraph{Analysis of power assuming $\Var(X\,|\,Z)=1$ is known.}

Since we condition on $\ndata Y$ in the model-X framework, it would be equivalent to consider
\begin{equation*}
T_{\text{model-X}}(\ndata{Y},\ndata{X}_*,\ndata{Z}_*)=\ndata{Y}^\top(\ndata{X}-\ndata{Z}(\ndata{Z}_*^\top \ndata{Z}_*)^{-1}\ndata{Z}_*^\top \ndata{X}_*)/\|\ndata{Y}\|.
\end{equation*}
By straightforward calculation,
\begin{equation*}
\begin{aligned}
T_{\text{model-X}}(\ndata{Y},\ndata{X}_*,\ndata{Z}_*)\mid \ndata{Y},\ndata{Z}_*&\sim\mathcal N(\mu_\beta(\ndata{Y},\ndata{Z}_*),\sigma^2_\beta(\ndata{Y},\ndata{Z}_*)),\text{ where}\\
\mu_\beta(\ndata{Y},\ndata{Z}_*)&=\frac{\beta}{\sigma^2+\beta^2}\frac{\ndata{Y}^\top(I-\ndata{Z}(\ndata{Z}_*^\top \ndata{Z}_*)^{-1}\ndata{Z}^\top)(\ndata{Y}-\ndata{Z}(\theta+\beta\eta))}{\|\ndata{Y}\|},\\
\sigma^2_\beta(\ndata{Y},\ndata{Z}_*)&=\frac{\sigma^2}{\sigma^2+\beta^2}+\frac{\beta^2-\sigma^2}{\beta^2+\sigma^2}\frac{\ndata{Y}^\top \ndata{Z}(\ndata{Z}_*^\top \ndata{Z}_*)^{-1}\ndata{Z}^\top \ndata{Y}}{\|\ndata{Y}\|^2}\\
&\quad\quad\quad-\frac{\beta^2}{\sigma^2+\beta^2}\frac{\ndata{Y}^\top \ndata{Z}(\ndata{Z}_*^\top \ndata{Z}_*)^{-1}\ndata{Z}^\top \ndata{Z}(\ndata{Z}_*^\top \ndata{Z}_*)^{-1}\ndata{Z}^\top \ndata{Y}}{\|\ndata{Y}\|^2}.
\end{aligned}
\end{equation*}
Similarly, we have
\begin{equation*}
T_{\text{model-X}}(\ndata{Y},\tilde{\ndata{X}}_*,\ndata{Z}_*)\mid \ndata{Y},\ndata{Z}_*\sim\mathcal N\left(0,\sigma^2_0(\ndata{Y},\ndata{Z}_*)=\left(1-\frac{\ndata{Y}^\top \ndata{Z}(\ndata{Z}_*^\top \ndata{Z}_*)^{-1}\ndata{Z}^\top \ndata{Y}}{\|\ndata{Y}\|^2}\right)\right),
\end{equation*}
and thus we reject if $T_{\text{model-X}}(\ndata{Y},\ndata{X}_*,\ndata{Z}_*)\ge z_{1-\alpha}\sqrt{\sigma^2_0(\ndata Y,\ndata Z_*)}$.
Under $\beta=h/\sqrt n$, the power is
\begin{equation}
\begin{aligned}
&\p_{\beta=h/\sqrt n}\left(T_{\text{model-X}}(\ndata{Y},\ndata{X}_*,\ndata{Z}_*)\ge z_{1-\alpha}\sqrt{\sigma_0^2(\ndata{Y},\ndata{Z}_*)}\right)\\
&\quad=\e\left[\p_{\beta=h/\sqrt n}\left(T_{\text{model-X}}(\ndata{Y},\ndata{X}_*,\ndata{Z}_*)\ge z_{1-\alpha}\sqrt{\sigma_0^2(\ndata{Y},\ndata{Z}_*)}\mid\ndata{Y},\ndata{Z}_*\right)\right]\\
&\quad=\e\left[1-\Phi\left(\frac{z_{1-\alpha}\sqrt{\sigma_0^2(\ndata{Y},\ndata{Z}_*)}-\mu_{h/\sqrt n}(\ndata{Y},\ndata{Z}_*)}{\sqrt{\sigma^2_{h/\sqrt n}(\ndata{Y},\ndata{Z}_*)}}\right)\right].
\end{aligned}
\label{equation:ols-unlabeled-power}
\end{equation}


\paragraph{Mean term.}
We first look at the term
\begin{equation*}
\mu_{\beta}(\ndata{Y},\ndata{Z}_*)=\frac{\beta}{\sigma^2+\beta^2}\frac{\ndata{Y}^\top(I-\ndata{Z}(\ndata{Z}_*^\top \ndata{Z}_*)^{-1}\ndata{Z}^\top)(\ndata{Y}-\ndata{Z}(\theta+\beta\eta))}{\|\ndata{Y}\|}.
\end{equation*}

Let $\epsilon=\ndata{Y}-\ndata{Z}(\theta+\beta\eta)$ be the residue vector that is independent of $\ndata Z$.
\begin{equation*}
\begin{aligned}
&\ndata{Y}^\top(I-\ndata{Z}(\ndata{Z}_*^\top \ndata{Z}_*)^{-1}\ndata{Z}^\top)(\ndata{Y}-\ndata{Z}(\theta+\beta\eta))\\
&=\epsilon^\top(I-\ndata{Z}(\ndata{Z}_*^\top \ndata{Z}_*)^{-1}\ndata{Z}^\top)(\ndata{Z}(\theta+\beta\eta)+\epsilon)\\
&=\epsilon^\top(I-\ndata{Z}(\ndata{Z}_*^\top \ndata{Z}_*)^{-1}\ndata{Z}^\top)\epsilon+\epsilon^\top(I-\ndata{Z}(\ndata{Z}_*^\top \ndata{Z}_*)^{-1}\ndata{Z}^\top)\ndata{Z}(\theta+\beta\eta)\\
&=\epsilon^\top\epsilon-\epsilon^\top(\ndata{Z}(\ndata{Z}_*^\top \ndata{Z}_*)^{-1}\ndata{Z}^\top)\epsilon+\epsilon^\top(I-\ndata{Z}(\ndata{Z}_*^\top \ndata{Z}_*)^{-1}\ndata{Z}^\top)\ndata{Z}(\theta+\beta\eta).
\end{aligned}
\end{equation*}
We normalize the above expression by $n$ and study each term. Note that we are under $\beta=h/\sqrt n$.
\begin{enumerate}
\item Since $\epsilon^\top\epsilon\sim(\sigma^2+\beta^2)\chi_n^2$, $n^{-1}\epsilon^\top\epsilon\cip\sigma^2$.

\item Let $A_i=(\ndata{Z}_*^\top \ndata{Z}_*)^{-1}\ndata{Z}_i\ndata{Z}_i^\top$, where $\ndata{Z}_i$ is the $i$th row of $\ndata{Z}_*$ as a column vector. Since all $A_i$'s are exchangeable and $\sum_{i=1}^{n_*}A_i=I_p$, $\e[A_i]=I_p/n_*$.
\begin{equation*}
\begin{aligned}
\e[n^{-1}\epsilon^\top(\ndata{Z}(\ndata{Z}_*^\top \ndata{Z}_*)^{-1}\ndata{Z}^\top)\epsilon]&=\frac{\sigma^2+\beta^2}{n}\e[\trace(\ndata{Z}(\ndata{Z}_*^\top \ndata{Z}_*)^{-1}\ndata{Z}^\top)]\\
&=\frac{\sigma^2+\frac1nh^2}{n}\e[\trace((\ndata{Z}_*^\top \ndata{Z}_*)^{-1}(\ndata{Z}^\top \ndata{Z})]\\
&=\frac{\sigma^2+\frac1nh^2}{n}\trace(\e[(\ndata{Z}_*^\top \ndata{Z}_*)^{-1}(\ndata{Z}^\top \ndata{Z})])\\
&=\frac{\sigma^2+\frac1nh^2}{n}\trace\left(\e\left[\sum_{j=1}^nA_j\right]\right)\\
&=\left(\sigma^2+\frac1nh^2\right)\trace(I_p/n_*)\\
&=\left(\sigma^2+\frac1nh^2\right)\frac{p}{n_*}\to\sigma^2\kappa\kappa_*.
\end{aligned}
\end{equation*}

Note that
\begin{equation*}
\begin{aligned}
0&=\Var\left(\trace(\ndata{Z}_*(\ndata{Z}_*^\top \ndata{Z}_*)^{-1}\ndata{Z}_*^\top)\right)\\
&=\Var\left(\trace\left(\sum_{i=1}^{n_*}A_i\right)\right)\\
&=n_*\Var(\trace(A_1))+n_*(n_*-1)\Cov(\trace(A_1),\trace(A_2))\\
&\quad\Rightarrow\Cov(\trace(A_1),\trace(A_2))=-\frac1{n_*-1}\Var(\trace(A_1)).
\end{aligned}
\end{equation*}

Thus, we have
\begin{equation*}
\begin{aligned}
\frac{\Var\left(\e[n^{-1}\epsilon^\top(\ndata{Z}(\ndata{Z}_*^\top \ndata{Z}_*)^{-1}\ndata{Z}^\top)\epsilon\mid\ndata{Z}_*]\right)}{\sigma^2+h^2/n}&=n^{-2}\Var\left(\trace(\ndata{Z}(\ndata{Z}_*^\top \ndata{Z}_*)^{-1}\ndata{Z}^\top)\right)\\
&=n^{-2}\Var\left(\trace\left(\sum_{i=1}^nA_i\right)\right)\\
&=n^{-1}\Var(\trace(A_1))+n^{-1}(n-1)\Cov(\trace(A_1),\trace(A_2))\\
&=n^{-1}\Var(\trace(A_1))+n^{-1}(n-1)\left(-\frac1{n_*-1}\Var(\trace(A_1))\right)\\
&=\frac{n_*-n}{n(n_*-1)}\Var(\trace(A_1))\le\frac{n_*-n}{n(n_*-1)}\to0,
\end{aligned}
\end{equation*}
where we use $\trace(A_1)=\lambda_\text{max}(\ndata{Z}_1^\top(\ndata{Z}_*^\top \ndata{Z}_*)^{-1}\ndata{Z}_1)\le1$. Too see this, note that $\ndata{Z}_1^\top(\ndata{Z}_*^\top \ndata{Z}_*)^{-1}\ndata{Z}_1$ is the $(1,1)$-entry of $\ndata{Z}_*(\ndata{Z}_*^\top \ndata{Z}_*)^{-1}\ndata{Z}_*^\top$, so $\lambda_\text{max}(\ndata{Z}_1^\top(\ndata{Z}_*^\top \ndata{Z}_*)^{-1}\ndata{Z}_1)\le\lambda_\text{max}(\ndata{Z}_*(\ndata{Z}_*^\top \ndata{Z}_*)^{-1}\ndata{Z}_*^\top)=1$ (a projection matrix).

Recall the variance formula for Gaussian quadratic forms \citep{rencher2008linear}, i.e., if $W\sim\mathcal N(\mu, \Sigma)$, then
\begin{equation*}
\Var(W^\top\Lambda W)=2\trace(\Lambda\Sigma\Lambda\Sigma)+4\mu^\top\Lambda\Sigma\Lambda\mu.
\end{equation*}
Thus, we have
\begin{equation*}
\begin{aligned}
\frac{\e\left[\Var\left(n^{-1}\epsilon^\top(\ndata{Z}(\ndata{Z}_*^\top \ndata{Z}_*)^{-1}\ndata{Z}^\top)\epsilon\mid \ndata{Z}_*\right)\right]}{(\sigma^2+h^2/n)^2}&=2n^{-2}\e\left[\trace\left(\left(\ndata{Z}(\ndata{Z}_*^\top \ndata{Z}_*)^{-1}\ndata{Z}^\top\right)^2\right)\right]\\
&\le2n^{-2}\e\left[\lambda_\text{max}(\ndata{Z}(\ndata{Z}_*^\top \ndata{Z}_*)^{-1}\ndata{Z}^\top)^2\trace(I_p)\right]\\
&\le2n^{-2}p\to0.
\end{aligned}
\end{equation*}
Here,
\begin{equation}
\begin{aligned}
&\lambda_\text{max}(\ndata{Z}(\ndata{Z}_*^\top \ndata{Z}_*)^{-1}\ndata{Z}^\top)\le\lambda_\text{max}\left(\ndata{Z}_*(\ndata{Z}_*^\top \ndata{Z}_*)^{-1}\ndata{Z}_*^\top\right)=1,
\label{equation:lambda-max}
\end{aligned}
\end{equation}
because $\ndata{Z}(\ndata{Z}_*^\top \ndata{Z}_*)^{-1}\ndata{Z}^\top$ is the matrix of the first $n$ rows and $n$ columns of $\ndata{Z}_*(\ndata{Z}_*^\top \ndata{Z}_*)^{-1}\ndata{Z}_*^\top$.
To sum up,
\begin{equation*}
\frac1n\epsilon^\top(\ndata{Z}(\ndata{Z}_*^\top \ndata{Z}_*)^{-1}\ndata{Z}^\top)\epsilon\cip\sigma^2\kappa\kappa_*.
\end{equation*}

\item Trivially,
\begin{equation*}
\e\left[\epsilon^\top(I-\ndata{Z}(\ndata{Z}_*^\top \ndata{Z}_*)^{-1}\ndata{Z}^\top)\ndata{Z}(\theta+\beta\eta)\mid\ndata{Z}_*\right]=0.
\end{equation*}

As for the variance,
\begin{equation*}
\begin{aligned}
\frac{\Var\left[\epsilon^\top(I-\ndata{Z}(\ndata{Z}_*^\top \ndata{Z}_*)^{-1}\ndata{Z}^\top)\ndata{Z}(\theta+\beta\eta)\mid\ndata{Z}_*\right]}{\sigma^2+h^2/n}&=(\theta+\beta\eta)^\top\ndata{Z}^\top(I-\ndata{Z}(\ndata{Z}_*^\top \ndata{Z}_*)^{-1}\ndata{Z}^\top)^2\ndata{Z}(\theta+\beta\eta)\\
&\le(\theta+\beta\eta)^\top \ndata{Z}^\top \ndata{Z}(\theta+\beta\eta).
\end{aligned}
\end{equation*}

Then we have
\begin{equation*}
\begin{aligned}
&\e\left(\Var\left[\epsilon^\top(I-\ndata{Z}(\ndata{Z}_*^\top \ndata{Z}_*)^{-1}\ndata{Z}^\top)\ndata{Z}(\theta+\beta\eta)\mid\ndata{Z}_*\right]\right)\\
&\le n^{-2}(\sigma^2+h^2/n)(\theta+\beta\eta)^\top\e\left[\ndata{Z}^\top \ndata{Z}\right](\theta+\beta\eta)\\
&=\frac{\sigma^2+h^2/n}{n}(\theta+\beta\eta)^\top\Sigma_Z(\theta+\beta\eta)\to0,
\end{aligned}
\end{equation*}
using $\theta^\top\Sigma_Z\theta\to v_Z^2<\infty$, $\beta=h/\sqrt n$, $\eta^\top\Sigma_Z\eta$ bounded and $\theta^\top\Sigma_Z\eta\le\sqrt{\theta^\top\Sigma_Z\theta\cdot\eta^\top\Sigma_Z\eta}$.
\end{enumerate}
We have established
\begin{equation*}
\frac1n\ndata{Y}^\top(I-\ndata{Z}(\ndata{Z}_*^\top \ndata{Z}_*)^{-1}\ndata{Z}^\top)(\ndata{Y}-\ndata{Z}(\theta+\beta\eta))\cip\sigma^2(1-\kappa\kappa_*).
\end{equation*}
On the other hand, since
\begin{equation*}
Y_i\simiid\mathcal N\left(0,\frac{h^2}{n}+(\theta+h\eta/\sqrt n)^\top\Sigma_Z(\theta+h\eta/\sqrt n)+\sigma^2\right),
\end{equation*}
we can use \eqref{equation:mc-to-show} to get $\|\ndata{Y}\|^2/n\cip v_Z^2+\sigma^2$.
Now we can see that
\begin{equation*}
\mu_{h/\sqrt n}(\ndata{Y},\ndata{Z}_*)=\frac{h}{\sigma^2+\frac1nh^2}\frac{\frac1n\ndata{Y}^\top(I-\ndata{Z}(\ndata{Z}_*^\top \ndata{Z}_*)^{-1}\ndata{Z}^\top)(\ndata{Y}-\ndata{Z}(\theta+\beta\eta))}{\frac1{\sqrt n}\|\ndata{Y}\|}\cip\frac{h(1-\kappa\kappa_*)}{\sqrt{v_Z^2+\sigma^2}}.
\end{equation*}

\paragraph{Variance term.} Next, we look at
\begin{equation*}
\begin{aligned}
\sigma^2_\beta(\ndata{Y},\ndata{Z}_*)&=\frac{\sigma^2}{\sigma^2+\beta^2}+\frac{\beta^2-\sigma^2}{\beta^2+\sigma^2}\frac{\ndata{Y}^\top \ndata{Z}(\ndata{Z}_*^\top \ndata{Z}_*)^{-1}\ndata{Z}^\top \ndata{Y}}{\|\ndata{Y}\|^2}\\
&\quad\quad\quad-\frac{\beta^2}{\sigma^2+\beta^2}\frac{\ndata{Y}^\top \ndata{Z}(\ndata{Z}_*^\top \ndata{Z}_*)^{-1}\ndata{Z}^\top \ndata{Z}(\ndata{Z}_*^\top \ndata{Z}_*)^{-1}\ndata{Z}^\top \ndata{Y}}{\|\ndata{Y}\|^2}.
\end{aligned}
\end{equation*}
We first note that
\begin{equation*}
\frac{\ndata{Y}^\top \ndata{Z}(\ndata{Z}_*^\top \ndata{Z}_*)^{-1}\ndata{Z}^\top \ndata{Z}(\ndata{Z}_*^\top \ndata{Z}_*)^{-1}\ndata{Z}^\top \ndata{Y}}{\|\ndata{Y}\|^2}\le\frac{\ndata{Y}^\top \ndata{Z}(\ndata{Z}_*^\top \ndata{Z}_*)^{-1}\ndata{Z}^\top \ndata{Y}}{\|\ndata{Y}\|^2}\le\lambda_\text{max}\left(\ndata{Z}(\ndata{Z}_*^\top \ndata{Z}_*)^{-1}\ndata{Z}^\top \right)\le1
\end{equation*}
by \eqref{equation:lambda-max}, so the last term in the expression of $\sigma^2_\beta(\ndata Y,\ndata Z_*)$ (recall $\beta=h/\sqrt n$) satisfies
\begin{equation*}
\frac{\beta^2}{\sigma^2+\beta^2}\frac{\ndata{Y}^\top \ndata{Z}(\ndata{Z}_*^\top \ndata{Z}_*)^{-1}\ndata{Z}^\top \ndata{Z}(\ndata{Z}_*^\top \ndata{Z}_*)^{-1}\ndata{Z}^\top \ndata{Y}}{\|\ndata{Y}\|^2}\to0.
\end{equation*}
Similarly,
\begin{equation}
\sigma^2_{\beta}(\ndata{Y},\ndata{Z}_*)-\sigma^2_0(\ndata{Y},\ndata{Z}_*)=\frac{\beta^2}{\sigma^2+\beta^2}\bigg(2\frac{\ndata{Y}^\top \ndata{Z}(\ndata{Z}_*^\top \ndata{Z}_*)^{-1}\ndata{Z}^\top \ndata{Y}}{\|\ndata{Y}\|^2}-\frac{\ndata{Y}^\top \ndata{Z}(\ndata{Z}_*^\top \ndata{Z}_*)^{-1}\ndata{Z}^\top \ndata{Z}(\ndata{Z}_*^\top \ndata{Z}_*)^{-1}\ndata{Z}^\top \ndata{Y}}{\|\ndata{Y}\|^2}\bigg)\to0.
\label{equation:diff-sigma2}
\end{equation}

Next, we analyze $\ndata{Y}^\top \ndata{Z}(\ndata{Z}_*^\top \ndata{Z}_*)^{-1}\ndata{Z}^\top \ndata{Y}$.
\begin{equation*}
\begin{aligned}
&\ndata{Y}^\top \ndata{Z}(\ndata{Z}_*^\top \ndata{Z}_*)^{-1}\ndata{Z}^\top \ndata{Y}\\
&=\epsilon^\top(\ndata{Z}(\ndata{Z}_*^\top \ndata{Z}_*)^{-1}\ndata{Z}^\top)\epsilon+2\epsilon^\top \ndata{Z}(\ndata{Z}_*^\top \ndata{Z}_*)^{-1}\ndata{Z}^\top\ndata{Z}(\theta+\beta\eta)+(\theta+\beta\eta)^\top \ndata{Z}^\top\ndata{Z}(\ndata{Z}_*^\top \ndata{Z}_*)^{-1}\ndata{Z}^\top\ndata{Z}(\theta+\beta\eta).
\end{aligned}
\end{equation*}
We again look at these terms one by one (after normalization by $n$).
\begin{enumerate}
\item We have already shown
\begin{equation*}
\frac1n\epsilon^\top(\ndata{Z}(\ndata{Z}_*^\top \ndata{Z}_*)^{-1}\ndata{Z}^\top)\epsilon\cip\sigma^2\kappa\kappa_*.
\end{equation*}
\item
\begin{equation*}
\e\left[\epsilon^\top \ndata{Z}(\ndata{Z}_*^\top \ndata{Z}_*)^{-1}\ndata{Z}^\top\ndata{Z}(\theta+\beta\eta)\mid\ndata{Z}_*\right]=0.
\end{equation*}
\begin{equation*}
\begin{aligned}
&\frac{\e\left(\Var\left[n^{-1}\epsilon^\top \ndata{Z}(\ndata{Z}_*^\top \ndata{Z}_*)^{-1}\ndata{Z}^\top\ndata{Z}(\theta+\beta\eta)\mid\ndata{Z}_*\right]\right)}{\sigma^2+h^2/n}\\
&=\frac{1}{n^2}\e\left((\theta+\beta\eta)^\top\ndata{Z}^\top\ndata{Z}(\ndata{Z}_*^\top \ndata{Z}_*)^{-1}\ndata{Z}^\top\ndata{Z}(\ndata{Z}_*^\top \ndata{Z}_*)^{-1}\ndata{Z}^\top\ndata{Z}(\theta+\beta\eta)\right)\\
&\le\frac{1}{n^2}\e\left((\theta+\beta\eta)^\top\ndata{Z}^\top\ndata{Z}(\ndata{Z}_*^\top \ndata{Z}_*)^{-1}\ndata{Z}_*^\top\ndata{Z}_*(\ndata{Z}_*^\top \ndata{Z}_*)^{-1}\ndata{Z}^\top\ndata{Z}(\theta+\beta\eta)\right)\\
&=\frac{1}{n^2}\e\left((\theta+\beta\eta)^\top\ndata{Z}^\top\ndata{Z}(\ndata{Z}_*^\top \ndata{Z}_*)^{-1}\ndata{Z}^\top\ndata{Z}(\theta+\beta\eta)\right)\\
&\le\frac{1}{n^2}\e\left((\theta+\beta\eta)^\top\ndata{Z}^\top\ndata{Z}(\theta+\beta\eta)\right)\\
&=\frac{1}{n}(\theta+\beta\eta)^\top\Sigma_Z(\theta+\beta\eta)\to0,
\end{aligned}
\end{equation*}
where the second to last step is because $\lambda_\text{max}(\ndata{Z}(\ndata{Z}_*^\top \ndata{Z}_*)^{-1}\ndata{Z}^\top)\le1$ and the last step is because $\theta^\top\Sigma_Z\theta\to v_Z^2<\infty$, $\beta=h/\sqrt n$, $\eta^\top\Sigma_Z\eta$ bounded and $\theta^\top\Sigma_Z\eta\le\sqrt{\theta^\top\Sigma_Z\theta\cdot\eta^\top\Sigma_Z\eta}$.

\item We now assume without loss of generality that $\Sigma_Z=I_p$, which we can achieve by absorbing $\Sigma_Z^{1/2}$ into $(\theta+\beta\eta)$.
We have the loose bounds
\begin{equation*}
\frac1n(\theta+\beta\eta)^\top \ndata{Z}^\top\ndata{Z}(\ndata{Z}_*^\top \ndata{Z}_*)^{-1}\ndata{Z}^\top\ndata{Z}(\theta+\beta\eta)\ge0,
\end{equation*}
\begin{equation*}
\begin{aligned}
&\frac1n(\theta+\beta\eta)^\top \ndata{Z}^\top\ndata{Z}(\ndata{Z}_*^\top \ndata{Z}_*)^{-1}\ndata{Z}^\top\ndata{Z}(\theta+\beta\eta)\\
&=\frac{1}{nn_*}(\theta+\beta\eta)^\top{\ndata{Z}^\top\ndata{Z}}{}\left(\frac{\ndata{Z}_*^\top \ndata{Z}_*}{n_*}\right)^{-1}{\ndata{Z}^\top\ndata{Z}}{}(\theta+\beta\eta)\\
&\le\frac{1}{\lambda_\text{min}\left(\frac{\ndata{Z}_*^\top \ndata{Z}_*}{n_*}\right)}\frac{1}{nn_*}(\theta+\beta\eta)^\top{\ndata{Z}^\top\ndata{Z}}{}{\ndata{Z}^\top\ndata{Z}}{}(\theta+\beta\eta)\\
&\le\frac{\lambda_\text{max}\left(\frac{\ndata{Z}\ndata{Z}^\top}{p}\right)}{\lambda_\text{min}\left(\frac{\ndata{Z}_*^\top \ndata{Z}_*}{n_*}\right)}\frac{p}{nn_*}(\theta+\beta\eta)^\top{\ndata{Z}^\top}{}{\ndata{Z}}{}(\theta+\beta\eta)\\
&=\frac{\lambda_\text{max}\left(\frac{\ndata{Z}\ndata{Z}^\top}{p}\right)}{\lambda_\text{min}\left(\frac{\ndata{Z}_*^\top \ndata{Z}_*}{n_*}\right)}\frac{p}{n_*}(\theta+\beta\eta)^\top\frac{{\ndata{Z}^\top}{}{\ndata{Z}}{}}{n}(\theta+\beta\eta)\\
&\cip\frac{(1+\sqrt{1/\kappa})^2}{(1-\sqrt{\kappa\kappa_*})^2}\kappa\kappa_*v_Z^2,
\end{aligned}
\end{equation*}
and
\begin{equation*}
\begin{aligned}
&\frac1n(\theta+\beta\eta)^\top \ndata{Z}^\top\ndata{Z}(\ndata{Z}_*^\top \ndata{Z}_*)^{-1}\ndata{Z}^\top\ndata{Z}(\theta+\beta\eta)\\
&\le(\theta+\beta\eta)^\top\frac{{\ndata{Z}^\top}{}{\ndata{Z}}{}}{n}(\theta+\beta\eta)\\
&\cip v_Z^2,
\end{aligned}
\end{equation*}
Thus, we already have
\begin{equation*}
\p_{\beta=h/\sqrt n}\left(1-\frac{\sigma^2\kappa\kappa_*+v_Z^2\min\left(1,\kappa\kappa_*\frac{(1+\sqrt{1/\kappa})^2}{(1-\sqrt{\kappa\kappa_*})^2}\right)}{\sigma^2+v_Z^2}\le\sigma^2_\beta(\ndata{Y},\ndata{Z}_*)\le1-\frac{\sigma^2\kappa\kappa_*}{\sigma^2+v_Z^2}\right)\to1.
\end{equation*}
Since both the lower and upper bounds in the above equation are positive, together with \eqref{equation:diff-sigma2} we get
\begin{equation*}
\frac{\sigma^2_0(\ndata{Y},\ndata{Z}_*)}{\sigma^2_\beta(\ndata{Y},\ndata{Z}_*)}\cip1.
\end{equation*}
Then we get
\begin{equation*}
\begin{aligned}
\limsup\eqref{equation:ols-unlabeled-power}&\le\Phi\left(\frac{h\tau\sqrt{1-\kappa\kappa_*}}{\sqrt{\sigma^2+v_Z^2\max\left(0,\frac{1-\frac{(1+\sqrt{1/\kappa})^2}{(1-\sqrt{\kappa\kappa_*})^2}\kappa\kappa_*}{1-\kappa\kappa_*}\right)}}-z_{1-\alpha}\right),\\
\liminf\eqref{equation:ols-unlabeled-power}&\ge\Phi\left(\frac{h\tau\sqrt{1-\kappa\kappa_*}}{\sqrt{\sigma^2+v_Z^2\frac{1}{{1-\kappa\kappa_*}}}}-z_{1-\alpha}\right).
\end{aligned}
\end{equation*}
\end{enumerate}
\end{proof}

\noindent\textbf{Conjecture~\ref{conjecture:model-X-unlabeled}.} \textit{In Setting~\ref{model:moderate-dim-lr}, if there are $m$ additional data points $(X_i,Z_i)_{i=n+1}^{n+m}$, $n_*=n+m$, $n/n_*\to\kappa_*$ and $\kappa\kappa_*<1$, then the conditional CRT with statistic $T_\textnormal{MC}$ has asymptotic power
equal to that of a $z$-test with standardized effect size
\[ \frac{h\sqrt{1-\kappa\kappa_*}}{\sqrt{\sigma^2+v_Z^2(1-\kappa_*)}}. \]}

\begin{proof}[Analysis of Conjecture~\ref{conjecture:model-X-unlabeled}]
Finding the asymptotic power is finding the exact limit of \eqref{equation:ols-unlabeled-power}, which requires a more careful analysis. Picking up from the end of the proof of Theorem~\ref{theorem:prop-model-X-unlabeled}, we now find the limit of the expectation of the third term in the variance decomposition normalized by $n$, i.e.,
\begin{equation*}
\lim\e\left[\frac1n(\theta+\beta\eta)^\top \ndata{Z}^\top\ndata{Z}(\ndata{Z}_*^\top \ndata{Z}_*)^{-1}\ndata{Z}^\top\ndata{Z}(\theta+\beta\eta)\right].
\end{equation*}
Recall that we are assuming $\Sigma_Z=I$ without loss of generality. We will show $\e[\ndata{Z}^\top\ndata{Z}(\ndata{Z}_*^\top\ndata{Z}_*)^{-1}\ndata{Z}^\top\ndata{Z}]=c(p,n,n_*)I_p$ is a mutliple of $I_p$. Once this is done, we will find the limit of $c(p,n,n_*)/n$.

To see this, we show that the expectation (a $p\times p$ matrix) is invariant under orthogonal transformation. Let $Q$ be any $p\times p$ orthogonal matrix.
\begin{equation*}
\begin{aligned}
\e[Q \ndata{Z}^\top\ndata{Z}(\ndata{Z}_*^\top\ndata{Z}_*)^{-1}\ndata{Z}^\top\ndata{Z}Q^\top]&=\e[Q \ndata{Z}^\top\ndata{Z}Q^\top(Q\ndata{Z}_*^\top\ndata{Z}_*Q^\top)^{-1}Q\ndata{Z}^\top\ndata{Z}Q^\top]\\
&=\e[\ndata{W}^\top \ndata{W}(\ndata{W}_*^\top \ndata{W}_*)^{-1}\ndata{W}^\top \ndata{W}]\text{ ($\ndata{W}_*=\ndata{Z}_*Q^\top\eqd\ndata{Z}_*$)}\\
&=\e[\ndata{Z}^\top\ndata{Z}(\ndata{Z}_*^\top\ndata{Z}_*)^{-1}\ndata{Z}^\top\ndata{Z}].
\end{aligned}
\end{equation*}
This shows the expectation must be a multiple of $I_p$. Now we consider
\begin{equation*}
\begin{aligned}
n_*I_p&=\e[\ndata{Z}_*^\top\ndata{Z}_*]\\
&=\e[\ndata{Z}_*^\top\ndata{Z}_*(\ndata{Z}_*^\top\ndata{Z}_*)^{-1}\ndata{Z}_*^\top\ndata{Z}_*]\\
&=\e\left[\left(\sum_{i=1}^{n_*}\ndata{Z}_i\ndata{Z}_i^\top\right)(\ndata{Z}_*^\top\ndata{Z}_*)^{-1}\left(\sum_{i=1}^{n_*}\ndata{Z}_i\ndata{Z}_i^\top\right)\right]\\
&=\sum_{i=1}^{n_*}\e\left[\ndata{Z}_i\ndata{Z}_i^\top(\ndata{Z}_*^\top\ndata{Z}_*)^{-1}\ndata{Z}_i\ndata{Z}_i^\top\right]+\sum_{i\ne j}\e\left[\ndata{Z}_i\ndata{Z}_i^\top(\ndata{Z}_*^\top\ndata{Z}_*)^{-1}\ndata{Z}_j\ndata{Z}_j^\top\right]\\
&=n_*\e[\ndata{Z}_1\ndata{Z}_1^\top(\ndata{Z}_*^\top\ndata{Z}_*)^{-1}\ndata{Z}_1\ndata{Z}_1^\top]+n_*(n_*-1)\e[\ndata{Z}_1\ndata{Z}_1^\top(\ndata{Z}_*^\top\ndata{Z}_*)^{-1}\ndata{Z}_2\ndata{Z}_2^\top]\\
&=n_*a_{p,n_*}I_p+n_*(n_*-1)b_{p,n_*}I_p,
\end{aligned}
\end{equation*}
where $\ndata{Z}_i$ is the $i$th row of $\ndata{Z}_*$ as a column vector and
\begin{equation*}
\begin{aligned}
\e[\ndata{Z}_1\ndata{Z}_1^\top(\ndata{Z}_*^\top\ndata{Z}_*)^{-1}\ndata{Z}_1\ndata{Z}_1^\top]&=a_{p,n_*}I_p,\\
\e[\ndata{Z}_1\ndata{Z}_1^\top(\ndata{Z}_*^\top\ndata{Z}_*)^{-1}\ndata{Z}_2\ndata{Z}_2^\top]&=b_{p,n_*}I_p
\end{aligned}
\end{equation*}
are multiples of $I_p$ by the same argument. From this we get $a_{p,n_*}+(n_*-1)b_{p,n_*}=1$. Similarly,
\begin{equation*}
\begin{aligned}
&\e[\ndata{Z}^\top\ndata{Z}(\ndata{Z}_*^\top\ndata{Z}_*)^{-1}\ndata{Z}^\top\ndata{Z}]\\
&=\e\left[\left(\sum_{i=1}^{n}\ndata{Z}_i\ndata{Z}_i^\top\right)(\ndata{Z}_*^\top\ndata{Z}_*)^{-1}\left(\sum_{i=1}^{n}\ndata{Z}_i\ndata{Z}_i^\top\right)\right]\\
&=n\e[\ndata Z_1\ndata Z_1^\top(\ndata{Z}_*^\top\ndata{Z}_*)^{-1}\ndata Z_1\ndata Z_1^\top]+n(n-1)\e[\ndata Z_1\ndata Z_1^\top(\ndata{Z}_*^\top\ndata{Z}_*)^{-1}\ndata Z_2\ndata Z_2^\top]\\
&=na_{p,n_*}I_p+n(n-1)b_{p,n_*}I_p\\
&=na_{p,n_*}I_p+n(n-1)\frac{1-a_{p,n_*}}{n_*-1}I_p.
\end{aligned}
\end{equation*}
Therefore, by representing $b_{p,n_*}$ with $a_{p,n_*}$, we have
\begin{equation*}
\begin{aligned}
n_*^{-1}c(p,n,n_*)I_p&=\e[\ndata{Z}^\top\ndata{Z}(\ndata{Z}_*^\top\ndata{Z}_*)^{-1}\ndata{Z}^\top\ndata{Z}/n_*]\\
&=\frac{n}{n_*}a_{p,n_*}I_p+n(n-1)\frac{1-a_{p,n_*}}{n_*(n_*-1)}I_p\\
&=\frac{n}{n_*}\left(\left(1-\frac{n-1}{n_*-1}\right)a_{p,n_*}+\frac{n-1}{n_*-1}\right)I_p.
\end{aligned}
\end{equation*}
We have shown that
\begin{equation}
n_*^{-1}c(p,n,n_*)=\frac{n}{n_*}\left(\left(1-\frac{n-1}{n_*-1}\right)a_{p,n_*}+\frac{n-1}{n_*-1}\right).
\label{equation:central}
\end{equation}
Recall that we are interested in the limit of this term, so we can focus on the limit of $a_{p,n_*}$.
\begin{equation*}
\begin{aligned}
a_{p,n_*}&=\e\left[\trace\left(\frac{\ndata{Z}_1\ndata{Z}_1^\top(\ndata{Z}_*^\top\ndata{Z}_*)^{-1}\ndata{Z}_1\ndata{Z}_1^\top}{p}\right)\right]\\
&=\e\left[\trace\left(\frac{\ndata{Z}_1^\top(\ndata{Z}_*^\top\ndata{Z}_*)^{-1}\ndata{Z}_1\ndata{Z}_1^\top \ndata{Z}_1}{p}\right)\right]\\
&=\e\left[\ndata{Z}_1^\top(\ndata{Z}_*^\top\ndata{Z}_*)^{-1}\ndata{Z}_1\cdot\frac{\|\ndata{Z}_1\|^2}{p}\right].
\end{aligned}
\end{equation*}
To study this expectation, note that
\begin{equation*}
\begin{aligned}
\e\left[\ndata{Z}_1^\top(\ndata{Z}_*^\top\ndata{Z}_*)^{-1}\ndata{Z}_1\right]&=\e\left[\trace\left(\ndata{Z}_1^\top(\ndata{Z}_*^\top\ndata{Z}_*)^{-1}\ndata{Z}_1\right)\right]\\
&=\e\left[\trace\left((\ndata{Z}_*^\top\ndata{Z}_*)^{-1}\ndata{Z}_1\ndata{Z}_1^\top\right)\right].
\end{aligned}
\end{equation*}
Note that by symmetry,
\begin{equation*}
\begin{aligned}
n_*\e\left[\trace\left((\ndata{Z}_*^\top\ndata{Z}_*)^{-1}\ndata{Z}_1\ndata{Z}_1^\top\right)\right]&=\sum_{i=1}^{n_*}\e\left[\trace\left((\ndata{Z}_*^\top\ndata{Z}_*)^{-1}\ndata{Z}_i\ndata{Z}_i^\top\right)\right]\\
&=\e\left[\trace\left(\sum_{i=1}^{n_*}(\ndata{Z}_*^\top\ndata{Z}_*)^{-1}\ndata{Z}_i\ndata{Z}_i^\top\right)\right]\\
&=\e\left[\trace\left((\ndata{Z}_*^\top\ndata{Z}_*)^{-1}\left(\sum_{i=1}^{n_*}\ndata{Z}_i\ndata{Z}_i^\top\right)\right)\right]\\
&=\e\left[\trace\left((\ndata{Z}_*^\top\ndata{Z}_*)^{-1}\left(\ndata{Z}_*^\top\ndata{Z}_*\right)\right)\right]\\
&=\e\left[\trace\left(I_p\right)\right]=p\Rightarrow\e\left[\ndata{Z}_1^\top(\ndata{Z}_*^\top\ndata{Z}_*)^{-1}\ndata{Z}_1\right]=p/n_*.
\end{aligned}
\end{equation*}
We can also note that $\ndata{Z}_1^\top(\ndata{Z}_*^\top\ndata{Z}_*)^{-1}\ndata{Z}_1$ is the first diagonal element of the projection matrix $\ndata{Z}_*(\ndata{Z}_*^\top\ndata{Z}_*)\ndata{Z}_*^\top$, whose eigenvalues are all $0$ or $1$. Thus, $0\le \ndata{Z}_1^\top(\ndata{Z}_*^\top\ndata{Z}_*)^{-1}\ndata{Z}_1\le1$.

Note that all the equations above are exact and no limit has been taken yet. Now we show the number sequence $a_{p,n_*}\to\kappa\kappa_*$, which we recall is the limit of $p/n_*$.

For any $\delta>0$, note that $\|\ndata{Z}_1\|^2\sim\chi_p^2$.
\begin{equation*}
\begin{aligned}
&\e\left[\ndata{Z}_1^\top(\ndata{Z}_*^\top\ndata{Z}_*)^{-1}\ndata{Z}_1\cdot\frac{\|\ndata{Z}_1\|^2}{p}\right]\\
&=\e\left[\ndata{Z}_1^\top(\ndata{Z}_*^\top\ndata{Z}_*)^{-1}\ndata{Z}_1\cdot\frac{\|\ndata{Z}_1\|^2}{p}\cdot\mathbb I(\frac{\|\ndata{Z}_1\|^2}{p}\le1+\delta)\right]+\e\left[\ndata{Z}_1^\top(\ndata{Z}_*^\top\ndata{Z}_*)^{-1}\ndata{Z}_1\cdot\frac{\|\ndata{Z}_1\|^2}{p}\cdot\mathbb I(\frac{\|\ndata{Z}_1\|^2}{p}>1+\delta)\right]\\
&\le(1+\delta)\e\left[\ndata{Z}_1^\top(\ndata{Z}_*^\top\ndata{Z}_*)^{-1}\ndata{Z}_1\right]+\e\left[1\cdot\frac{\|\ndata{Z}_1\|^2}{p}\cdot\mathbb I(\frac{\|\ndata{Z}_1\|^2}{p}>1+\delta)\right]\\
&=(1+\delta)\frac{p}{n_*}+\e\left[\frac{\|\ndata{Z}_1\|^2}{p}\cdot\mathbb I(\frac{\|\ndata{Z}_1\|^2}{p}>1+\delta)\right]\\
&\le(1+\delta)\frac{p}{n_*}+\sqrt{\e\left[\frac{\|\ndata{Z}_1\|^4}{p^2}\right]\e\left[\mathbb I(\frac{\|\ndata{Z}_1\|^2}{p}>1+\delta)^2\right]}\\
&=(1+\delta)\frac{p}{n_*}+\sqrt{\frac{p^2+2p}{p^2}\p\left(\frac{\|\ndata{Z}_1\|^2}{p}>1+\delta\right)}\\
&=(1+\delta)\frac{p}{n_*}+\sqrt{1+\frac2p}\sqrt{\p\left(\frac{\|\ndata{Z}_1\|^2}{p}>1+\delta\right)}\\
&\le(1+\delta)\frac{p}{n_*}+\sqrt{1+\frac2p}\sqrt{\p\left(\left|\frac{\|\ndata{Z}_1\|^2}{p}-1\right|>\delta\right)}\\
&\le(1+\delta)\frac{p}{n_*}+\sqrt{1+\frac2p}\sqrt{\frac{\Var\left(\|\ndata{Z}_1\|^2/p\right)}{\delta^2}}\\
&=(1+\delta)\frac{p}{n_*}+\sqrt{1+\frac2p}\sqrt{\frac{2}{p\delta^2}}\to(1+\delta)\kappa\kappa_*.
\end{aligned}
\end{equation*}
Since $\delta$ can be arbitrarily small, this shows
\begin{equation}
\limsup\e\left[\ndata{Z}_1^\top(\ndata{Z}_*^\top\ndata{Z}_*)^{-1}\ndata{Z}_1\cdot\frac{\|\ndata{Z}_1\|^2}{p}\right]\le\kappa\kappa_*.
\label{equation:limsup}
\end{equation}

On the other hand,
\begin{equation*}
\begin{aligned}
&\e\left[\ndata{Z}_1^\top(\ndata{Z}_*^\top\ndata{Z}_*)^{-1}\ndata{Z}_1\cdot\frac{\|\ndata{Z}_1\|^2}{p}\right]\\
&=\e\left[\ndata{Z}_1^\top(\ndata{Z}_*^\top\ndata{Z}_*)^{-1}\ndata{Z}_1\cdot\frac{\|\ndata{Z}_1\|^2}{p}\cdot\mathbb I(\frac{\|\ndata{Z}_1\|^2}{p}\ge1-\delta)\right]+\e\left[\ndata{Z}_1^\top(\ndata{Z}_*^\top\ndata{Z}_*)^{-1}\ndata{Z}_1\cdot\frac{\|\ndata{Z}_1\|^2}{p}\cdot\mathbb I(\frac{\|\ndata{Z}_1\|^2}{p}<1-\delta)\right]\\
&\ge\e\left[\ndata{Z}_1^\top(\ndata{Z}_*^\top\ndata{Z}_*)^{-1}\ndata{Z}_1\cdot\frac{\|\ndata{Z}_1\|^2}{p}\cdot\mathbb I(\frac{\|\ndata{Z}_1\|^2}{p}\ge1-\delta)\right]\\
&\ge(1-\delta)\e\left[\ndata{Z}_1^\top(\ndata{Z}_*^\top\ndata{Z}_*)^{-1}\ndata{Z}_1\cdot\mathbb I(\frac{\|\ndata{Z}_1\|^2}{p}\ge1-\delta)\right]\\
&=(1-\delta)\e\left[\ndata{Z}_1^\top(\ndata{Z}_*^\top\ndata{Z}_*)^{-1}\ndata{Z}_1\right]-(1-\delta)\e\left[\ndata{Z}_1^\top(\ndata{Z}_*^\top\ndata{Z}_*)^{-1}\ndata{Z}_1\cdot\mathbb I(\frac{\|\ndata{Z}_1\|^2}{p}<1-\delta)\right]\\
&=(1-\delta)\frac{p}{n_*}-(1-\delta)\e\left[\ndata{Z}_1^\top(\ndata{Z}_*^\top\ndata{Z}_*)^{-1}\ndata{Z}_1\cdot\mathbb I(\frac{\|\ndata{Z}_1\|^2}{p}<1-\delta)\right]\\
&\ge(1-\delta)\frac{p}{n_*}-(1-\delta)\e\left[1\cdot\mathbb I(\frac{\|\ndata{Z}_1\|^2}{p}<1-\delta)\right]\\
&\ge(1-\delta)\frac{p}{n_*}-(1-\delta)\p\left(\left|\frac{\|\ndata{Z}_1\|^2}{p}-1\right|>\delta\right)\\
&\ge(1-\delta)\frac{p}{n_*}-(1-\delta)\frac{\Var\left(\|\ndata{Z}_1\|^2/p\right)}{\delta^2}\\
&=(1-\delta)\frac{p}{n_*}-(1-\delta)\frac{2}{p\delta^2}\to(1-\delta)\kappa\kappa_*.
\end{aligned}
\end{equation*}
Since $\delta$ can be arbitrarily small, this shows
\begin{equation}
\liminf\e\left[\ndata{Z}_1^\top(\ndata{Z}_*^\top\ndata{Z}_*)^{-1}\ndata{Z}_1\cdot\frac{\|\ndata{Z}_1\|^2}{p}\right]\ge\kappa\kappa_*.
\label{equation:liminf}
\end{equation}
Equations~\eqref{equation:limsup} and \eqref{equation:liminf} show $a_{p,n_*}\to\kappa\kappa_*$. This together with equation~\eqref{equation:central} shows
\begin{equation*}
\begin{aligned}
n^{-1}c(p,n,n_*)&=\left(1-\frac{n-1}{n_*-1}\right)a_{p,n_*}+\frac{n-1}{n_*-1}\\
&\to\left(1-\kappa_*\right)\kappa\kappa_*+\kappa_*\\
&=\kappa_*(1+\kappa(1-\kappa_*)).
\end{aligned}
\end{equation*}

To sum up, we have shown
\begin{equation*}
\frac1n\e\left[(\theta+\beta\eta)^\top \ndata{Z}^\top\ndata{Z}(\ndata{Z}_*^\top \ndata{Z}_*)^{-1}\ndata{Z}^\top\ndata{Z}(\theta+\beta\eta)\right]=\frac{c(p,n,n_*)}{n}(\theta+\beta\eta)^\top\Sigma_Z(\theta+\beta\eta)\to\kappa_*(1+\kappa(1-\kappa_*))v_Z^2.
\end{equation*}
We conjecture that actually (e.g., if we can show the variance converges to $0$)
\begin{equation*}
\frac1n(\theta+\beta\eta)^\top \ndata{Z}^\top\ndata{Z}(\ndata{Z}_*^\top \ndata{Z}_*)^{-1}\ndata{Z}^\top\ndata{Z}(\theta+\beta\eta)\cip\kappa_*(1+\kappa(1-\kappa_*))v_Z^2,
\end{equation*}
in which case \eqref{equation:ols-unlabeled-power} converges to
\begin{equation*}
\Phi\left(\frac{h\sqrt{1-\kappa\kappa_*}}{\sqrt{\sigma^2+v_Z^2(1-\kappa_*)}}-z_{1-\alpha}\right).
\end{equation*}
\end{proof}

\begin{thm}
Let $J_0$ and $J_1$ form a partition of $\{1,2,\dots,p\}$ and $|J_0|/p\to\gamma\in(0,1)$. Consider $p$ random variables $p_j=1-F_j(T_j)$, which should be thought of as $p$-values. Assume the following conditions.
\begin{enumerate}
    \item[(a)] For $j\in J_0$, $1-F_j(T_j)\sim\Unif[0,1]$; for any $t\in\rr$, $F_j(t)\cip F^{(0)}(t)$ and for $j\in J_1$, $F_j(t)\cip F^{(1)}(t)$. $F^{(0)}$ and $F^{(1)}$ are deterministic CDFs of random variables with common connected support and continuous densities on their support. For $j'\in J_1$, $T_{j'}\cid T^{(1)}$, which has the same support as $F^{(0)}$ and $F^{(1)}$ and continuous density on the support.
    \item[(b)] Within $J_0$ or $J_1$, $p_j$'s are exchangeable.
    \item[(c)] For 
    distinct $j_1,j_2\in J_0$ and distinct $j_3,j_4\in J_1$,
    the following two pairs of random variables are asymptotically pairwise independent: $(T_{j_1},T_{j_2})$ and $(T_{j_3},T_{j_4})$. That is, both pairs converge in distribution to a bivariate random vector (not necessarily the same random vector) with independent components.
\end{enumerate}
Then let $G$ be the CDF of $1-F^{(1)}(T^{(1)})$, $q\in(0,1)$, and
\begin{equation*}
t(g)=\max\{t\in(0,1]:g(t)\le q\}.
\end{equation*}
When the set is empty, define $t(g)=0$. Let
\begin{equation*}
g_\textnormal{BH}(t)=\frac{t}{\gamma t+(1-\gamma)G(t)}
\end{equation*}
and
\begin{equation*}
g_\textnormal{AdaPT}(t)=\frac{\gamma t+(1-\gamma)\left(1-G(1-t)\right)}{\gamma t+(1-\gamma)G(t)}.
\end{equation*}
The for $t_\textnormal{BH}=t(g_\textnormal{BH})$ and almost every $q\in(0,1)$, at least one of the following cases is true: (i) $t_\textnormal{BH}\in(0,1)$ and $g_\text{BH}'(t_\textnormal{BH})\ne0$, (ii) $t_\textnormal{BH}=1$ and $g_\textnormal{BH}(1)<q$, or (iii) $t_\text{BH}=0$. In cases (i) or (ii), for the BH procedure at level $q$, the FDP and realized power converge in probability to
\begin{equation*}
\frac{\gamma t_\textnormal{BH}}{\gamma t_\textnormal{BH}+(1-\gamma)G(t_\textnormal{BH}))}\quad\text{and}\quad G(t_\textnormal{BH}),
\end{equation*}
respectively. In case (i), the asymptotic realized power simplifies to $\gamma q$. In case (iii), the realized power converges in probability to $0$.

For $t_\textnormal{AdaPT}=t(g_\textnormal{AdaPT})$ and almost every $q\in(0,1)$, at least one of the following cases is true: (i) $t_\textnormal{AdaPT}\in(0,1)$ and $g_\textnormal{AdaPT}'(t_\textnormal{AdaPT})\ne0$, (ii) $t_\textnormal{AdaPT}=1$ and $g_\textnormal{AdaPT}(1)<q$, or (iii) $t_\textnormal{AdaPT}=0$.
In cases (i) or (ii), for the AdaPT procedure at level $q$, the FDP and realized power converge in probability to
\begin{equation*}
\frac{\gamma t_\textnormal{AdaPT}}{\gamma t_\textnormal{AdaPT}+(1-\gamma)G(t_\textnormal{AdaPT}))}\quad\text{and}\quad G(t_\textnormal{AdaPT}),
\end{equation*}
respectively. In case (iii), the realized power converges in probability to $0$.
\label{theorem:BH-AdaPT-CRT}
\end{thm}

\begin{proof}[Proof of Theorem~\ref{theorem:BH-AdaPT-CRT}]
We only prove the case for AdaPT, because the proof for BH is similar and slightly easier. We suppress the subsript in $g_\text{AdaPT}$.

If $q>g(1)$, then (ii) holds. If $q<\inf\{g(t):t\in(0,1]\}$, then (iii) holds. If $q\in(\inf\{g(t):t\in(0,1], g(1))$, then because $g$ is continuous, we can see that $t_\text{AdaPT}\in(0,1)$ (note that we are considering maximum and $(0,1]$ is closed on the right, so the maximum must exist and not equal to $1$). And in this case, $g(t_\text{AdaPT})=q$, since otherwise $g(t_\text{AdaPT})<q$ and $t_\text{AdaPT}$ could be smaller because of $g$'s continuity. Next we show for almost every $q\in(\inf\{g(t):t\in(0,1], g(1))$, case (i) holds. We just need to show that the set $\{g(t):g'(t)=0\}$ has measure zero, which is a simple application of Sard's theorem (take $n=m=k=1$, $f(x)=g(\arctan(x)/\pi+1/2)$, where $\arctan(x)/\pi+1/2$ is just some function that maps $\rr$ to $(0,1)$ with continuous nonzero derivative), where $g'$ is continuous on (0,1) because $F^{(1)}$, $T^{(1)}$ and thus $G$ have continuous densities on a common connected support.
\begin{lemma}[Sard's theorem, \citet{sard1942measure}]
Let $f:\rr^n\to\rr^m$ be $k$ times continuously differentiable, where $k\ge\max(n-m+1,1)$. Let $A$ be the set of points $x\in\rr^n$ such that the Jacobian matrix of $f$ has rank smaller than $m$. Then the image $f(A)$ has Lebesgue measure zero in $\rr^m$.
\label{lemma:sard}
\end{lemma}

In case (i), $t_\text{AdaPT}\in(0,1)$, we have established that $g(t_\text{AdaPT})=q$. Since $g'(t_\text{AdaPT})\ne0$, we must have $g'(t_\text{AdaPT})<0$, since otherwise $t_\text{AdaPT}$ could also be smaller. Thus, in case (i), $g'(t_\text{AdaPT})<0$ and for any sufficiently small $\varepsilon>0$ there exists a point $t^*$ in $(t_\text{AdaPT}-\varepsilon,t_\text{AdaPT})$ such that $g(t^*)<q$. In case (ii), since $g$ is continuous and $g(1)<q$, it also holds that for any sufficiently small $\varepsilon>0$ there exists a point $t^*$ in $(t_\text{AdaPT}-\varepsilon,t_\text{AdaPT})$ such that $g(t^*)<q$.

Let
\begin{equation*}
\hat t_p=\min\{t\in(0,1]:g_p(t)\equiv\frac{1/p+\#\{j:p_j\ge1-t\}/p}{\#\{j:p_j\le t\}/p}\le q\}.\footnote{The term $1/p$ does not matter asymptotically, in that we can still consider the numerator as an empirical CDF of the $p_j$'s.}
\end{equation*}
We analyze cases (i) and (ii). We begin by showing $\hat t_p\cip t_\text{AdaPT}$. Take any sufficiently small $\varepsilon>0$. We have established that there exists a point $t^*$ in $(t_\text{AdaPT}-\varepsilon,t_\text{AdaPT})$ such that $g(t^*)<q$. Then
\begin{equation*}
\begin{aligned}
\p(\hat t_p>t_\text{AdaPT}-\varepsilon)&\ge \p(\hat t_p\ge t^*)\\
&\ge\p(g_p(t^*)\le q)\\
&\ge\p(|g_p(t^*)-g(t^*)|<|g(t^*)-q|)\to1.
\end{aligned}
\end{equation*}
In case (ii), we get $\p(\hat t_p\le t_\text{AdaPT}=1)=1$ for free and the proof is concluded. Next we consider case (i) and assume $\varepsilon\in(0,1-t_\text{AdaPT})$. Choose $\delta_1\in\left(0,\min(1-t_\text{AdaPT}-\varepsilon,1-G(t_\text{AdaPT}+\varepsilon))\right)$. Let $\delta_3=\min\{g(t)-q:t\le[t_\text{AdaPT}+\varepsilon,1]\}$, which is positive since otherwise $g$ could attain a value no more than $q$ in $[t_\text{AdaPT}+\varepsilon,1]$, violating $t_\text{AdaPT}$'s definition.
Now observe that the function $\frac{\gamma x+(1-\gamma)y}{\gamma z+(1-\gamma)w}$ is continuous in $(x,y,z,w)$ on $\{(x,y,z,w)\in[0,1]^4:z,w\ge\min(1-t_\text{AdaPT}-\varepsilon,1-G(t_\text{AdaPT}+\varepsilon)-\varepsilon)-\delta_1\}$, and thus uniformly continuous. So we can choose $\delta_2\in(0,\delta_1)$ such that whenever $(x,y,z,w),(x',y',z',w')\in\{(x,y,z,w)\in[0,1]^4:z,w\ge\min(1-t_\text{AdaPT}-\varepsilon,1-G(t_\text{AdaPT}+\varepsilon))\}$ and $|x-x'|,|y-y'|,|z-z'|,|w-w'|<\delta_2$, $|\frac{\gamma x+(1-\gamma)y}{\gamma z+(1-\gamma)w}-\frac{\gamma x'+(1-\gamma)y'}{\gamma z'+(1-\gamma)w'}|<\delta_3$.

Now we show the empirical CDFs of the non-null and null $p$-values converge pointwise, which implies uniform convergence by Lemma~\ref{lemma:unif-conv-in-prob}. Let $\hat G^{(1)}_p(t)=\#\{j\in J_1:p_j\le t\}/|J_1|$. Due to exchangeability, for any $j\in J_1$,
\begin{equation*}
\begin{aligned}
\e[\hat G^{(1)}_p(t)]&=\p[p_j\le t]=\p(1-F_j(T_j)\le t)\\
&\to\p(1-F^{(1)}(T^{(1)})\le t)=G_1(t).
\end{aligned}
\end{equation*}
As for the variance, we have for any distinct $j,k\in J_1$,
\begin{equation*}
\Var[\hat G^{(1)}_p(t)]=\frac{1}{(1-\gamma)p}\Var[\ind(p_j\le t)]+\frac{(1-\gamma)p((1-\gamma)p-1)}{(1-\gamma)^2p^2}\Cov(\ind(p_j\le t),\ind(p_k\le t)).
\end{equation*}

Since all the $F_j$'s converges in distribution to deterministic and continuous CDFs, by the asymptotic pairwise independence, $\Var[\hat G^{(1)}_p(t)]\to0$, therefore establishing $\hat G^{(1)}_p(t)\cip G_1(t)$. We can show the empirical CDF of the null $p$-values converges in probability in the same way.

By Lemma~\ref{lemma:unif-conv-in-prob}, with probability converging to $1$, all the CDFs fall within a $\delta_2$-neighborhood around the limit CDFs, and by the previously showed uniform continuity, $|g_p(t)-g(t)|<\delta_3$ for $t\in[t_\text{AdaPT}+\varepsilon,1]$. By the definition of $\delta_3$, this means with probability converging to $1$, $g_p(t)>q$ for all $t\in[t_\text{AdaPT}+\varepsilon,1]$. Now we have
\begin{equation*}
\p(\hat t_p<t_\text{AdaPT}+\varepsilon)=\p(g_p(t)>q,\forall t\in[t_\text{AdaPT}+\varepsilon,1])\to1.
\end{equation*}
Combining the results, we have
\begin{equation*}
\lim_{p\to\infty}\p(|\hat t_p-t_\text{AdaPT}|<\varepsilon)=1.
\end{equation*}
Due to the uniform convergence of the CDFs, the results on the expressions on the asymptotic FDP and realized power then follow by noticing that the FDP and the realized power are
\begin{equation*}
\frac{\#\{j\in J_0:p_j\le\hat t_p\}}{\#\{j:p_j\le\hat t_p\}}\quad\text{and}\quad\frac{\#\{j\in J_1:p_j\le\hat t_p\}}{|J_1|},
\end{equation*}
respectively.

Finally, we consider case (iii). In this case, we have $q<\inf\{g(t):t\in(0,1]\}$. Take any small positive $\delta$, we have $q<\inf\{g(t):t\in[\delta,1]\}$. Since $0$ is excluded from $[\delta,1]$, we can again use the uniform continuity argument we used before, where we consider $g_p(t)$ as a function of four inputs, and the two inputs in the denominator are bounded away from $0$. In this way, we can show $g_p(t)\cip t$ uniformly for $t\in[\delta,1]$, so that we have
\begin{equation*}
\p(\hat t_p<\delta)=\p(g_p(t)\ge q, \forall t\in[\delta,1])\to1.
\end{equation*}
Thus, the asymptotic realized power satisfies
\begin{equation*}
\begin{aligned}
\frac{\#\{j\in J_1:p_j\le\hat t_p\}}{|J_1|}&=\ind_{\{\hat t_p<\delta\}}\frac{\#\{j\in J_1:p_j\le\hat t_p\}}{|J_1|}+\ind_{\{\hat t_p\ge\delta\}}\frac{\#\{j\in J_1:p_j\le\hat t_p\}}{|J_1|}\\
&\le\underbrace{\ind_{\{\hat t_p<\delta\}}}_{\cip1}\underbrace{\frac{\#\{j\in J_1:p_j\le\delta\}}{|J_1|}}_{\cip G(\delta)}+\underbrace{\ind_{\{\hat t_p\ge\delta\}}}_{\cip0}\underbrace{\frac{\#\{j\in J_1:p_j\le\hat t_p\}}{|J_1|}}_{\le1}\\
&\cip G(\delta).
\end{aligned}
\end{equation*}
Since $\delta$ can be arbitrarily small and $G$ has no point mass at $0$ because all distributions considered here are continuous, we have shown the realized power on the left hand side converges in distribution to $0$.

\end{proof}

\begin{lemma}
Let $\{G_n\}$ be a sequence of random CDFs. If for every $t\in\rr$, $G_n(t)\cip G(t)$, with $G$ being a continuous CDF, then the convergence is uniform in $t$, in the sense that
\begin{equation*}
\sup_{t\in\rr}|G_n(t)-G(t)|\cip0.
\end{equation*}
\label{lemma:unif-conv-in-prob}
\end{lemma}

\begin{proof}[Proof of Lemma~\ref{lemma:unif-conv-in-prob}]
Take any $\varepsilon>0$. Find a finite number of points $x_1,x_2,\dots,x_N$ such that $G(x_1)\le\varepsilon/2$, $G(x_{i})-G(x_{i-1})\le\varepsilon/2$ and $1-G(x_N)\le\varepsilon/2$. We have
\begin{equation*}
\p\left(\max_{1\le i\le N}|G_n(x_i)-G(x_i)|\le\varepsilon/2\right)\to1.
\end{equation*}
Now for any $x\in[x_{i-1},x_{i}]$ ($x_0=-\infty,x_{N+1}=\infty$), on the event $\max_{1\le i\le N}|G_n(x_i)-G(x_i)|\le\varepsilon/2$,
\begin{equation*}
G_n(x)\ge G_n(x_{i-1})\ge G(x_{i-1})-\varepsilon/2\ge G(x_i)-\varepsilon/2-\varepsilon/2\ge G(x)-\varepsilon,
\end{equation*}
\begin{equation*}
G_n(x)\le G_n(x_{i})\le G(x_{i})+\varepsilon/2\le G(x_{i-1})+\varepsilon/2+\varepsilon/2\le G(x)+\varepsilon.
\end{equation*}
Thus,
\begin{equation*}
\p(\sup_{t}|G_n(t)-G(t)|<\varepsilon)\ge\p(\max_{1\le i\le N}|G_n(x_i)-G(x_i)|\le\varepsilon/2)\to1.
\end{equation*}
\end{proof}

\noindent\textbf{Theorem~\ref{theorem:coro-BH-CRT}.} \textit{In Setting~\ref{model:lr-iid}, for Lebesgue-almost-every $q\in(0,1)$, BH or AdaPT at level $q$ using CRT $p$-values based on the statistics in Section~\ref{sec:CRT} (respectively, their absolute values) have the following one-sided (respectively, two-sided) effective $\pi_\mu$'s with respect to BH or AdaPT at level $q$:
\begin{enumerate}
    \item For the marginal covariance statistic, the effective $\pi_\mu$ is the distribution of $\frac{1}{\sqrt{\sigma^2+\kappa\e[B_0^2]}}{B_0}$.
    \item For the OLS statistic, assuming $\kappa<1$, the effective $\pi_\mu$ is the distribution of $\frac{\sqrt{1-\kappa}}\sigma B_0$.
    \item For the distilled lasso statistic, the effective $\pi_\mu$ is the distribution of $\frac1{\tau_\lambda}B_0$.
\end{enumerate}}

\begin{proof}[Proof of Theorem~\ref{theorem:coro-BH-CRT}]

We prove the case of one-sided $p$-values, while it is clear that the heart of these results is at the asymptotic pairwise independence of the variable important statistics $T_j$'s, and switching to two-sided $p$-values (or other reasonable $p$-values) only effectively changes using $T_j$ into using $|T_j|$, the results for which can be established almost identically. We use $\Phi_{a^2}$ to denote the CDF of $\mathcal N(0,a^2)$.

\begin{enumerate}
\item \textit{Marginal covariance}.
Let $T_j=n^{-1/2}\ndata{Y}^\top\ndata{X}_j$ and $F_j$ be the CDF of $\mathcal N(0,\|\ndata{Y}\|_2^2/n)$. Now we check the conditions of Theorem~\ref{theorem:BH-AdaPT-CRT}.
\begin{enumerate}
    \item It is clear that $1-F_j(T_j)\sim\Unif[0,1]$ for the null variables. Note that $\|\ndata{Y}\|^2/n\cip\sigma^2+\kappa\e[B_0^2]$ (see, e.g., the proof of Theorem~\ref{theorem:prop-power-mc}), so for any $j$ and $t$, $F_j(t)\cip\Phi_{\sigma^2+\kappa\e[B_0^2]}(t)$. We will find $\mathcal L(T^{(1)})$ with Lemma~\ref{lemma:wishart-char}.
    \item This is true because $\beta_j$'s are i.i.d.
    \item We verify this condition by introducing Lemma~\ref{lemma:wishart-char}, which also completes part (a).
\end{enumerate}
\begin{lemma}
In Setting~\ref{model:lr-iid}, for distinct $j$ and $k$ we have
\label{lemma:wishart-char}
\begin{equation*}
    \left(\begin{array}{c}
T_j\\
T_k
  \end{array}\right)-\left(\begin{array}{c}
\sqrt n\beta_j\\
\sqrt n\beta_k\\
  \end{array}\right)\cid\mathcal N\left(\left(\begin{array}{c}
0\\
0\\
  \end{array}\right),\left[\begin{array}{cc}
\sigma^2+\kappa\e[B_0^2] & 0\\
0 & \sigma^2+\kappa\e[B_0^2]\\
  \end{array}\right]\right).
    \end{equation*}
\end{lemma}
To sum up, the $p$-values we obtain from marginal covariance satisfy the conditions of Theorem~\ref{theorem:BH-AdaPT-CRT} with $F^{(1)}=\Phi_{\sigma^2+\kappa\e[B_0^2]}$ and $T^{(1)}\sim B_0+\sqrt{\sigma^2+\kappa\e[B_0^2]}W$, where $W$ is a standard Gaussian random variable independent of $B_0$. We obtain the effective $\pi_\mu$ by dividing the $T_j$'s by $\sqrt{\sigma^2+\kappa\e[B_0^2]}$.

\item \textit{OLS}.
Let $T_j=\sqrt n\hat\beta_j$ be the (normalized) OLS estimate for covariate $X_j$. Let $\ndata{X}_{\text{-}j}$ be $\ndata{X}$ with the $j$th column removed and $F_j$ be the CDF of $\mathcal L(T_j\,|\,\ndata{Y},\ndata{X}_{\text{-}j})$ under the null. We now check the conditions of Theorem~\ref{theorem:BH-AdaPT-CRT}.

\begin{enumerate}
    \item It is clear that $1-F_j(T_j)\sim\Unif[0,1]$ for the null variables, and we have shown that for any $j$ and $t$, $F_j(t)\cip\Phi_{\sigma^2/(1-\kappa)}(t)$ in the proof of Theorem~\ref{theorem:prop-power-mc}.
    \item This is true because $\beta_j$'s are i.i.d. We will find $\mathcal L(T^{(1)})$ in part (c).
    \item Since
    \begin{equation*}
    \left(\begin{array}{c}
\sqrt n\hat\beta_j\\
\sqrt n\hat\beta_k
  \end{array}\right)-\left(\begin{array}{c}
\sqrt n\beta_j\\
\sqrt n\beta_k
  \end{array}\right)\mid\ndata{X}\sim\mathcal N\left(\left(\begin{array}{c}
0\\
0
  \end{array}\right),\sigma^2\left[\left(\frac{\ndata{X}^\top\ndata{X}}{n}\right)^{-1}\right]_{\{j,k\},\{j,k\}}\right),
    \end{equation*}
    we can show
    \begin{equation*}
    \left(\begin{array}{c}
\sqrt n\hat\beta_j\\
\sqrt n\hat\beta_k
  \end{array}\right)-\left(\begin{array}{c}
\sqrt n\beta_j\\
\sqrt n\beta_k
  \end{array}\right)\cid\mathcal N\left(\left(\begin{array}{c}
0\\
0
  \end{array}\right),\frac{\sigma^2}{1-\kappa}I_2\right).
    \end{equation*}
    once we verify that any $2\times2$ sub-diagonal matrix of $\left(\ndata{X}^\top\ndata{X}/n\right)^{-1}$ converges in probability to $(1-\kappa)^{-1}I_2$, which follows directly from a computation of the first and second moments of the inverse Wishart distribution.
\end{enumerate}
To sum up, the $p$-values we obtain from OLS satisfy the conditions of Theorem~\ref{theorem:BH-AdaPT-CRT} with $F^{(1)}=\Phi_{\sigma^2/(1-\kappa)}$ and $T^{(1)}\sim B_0+\sigma W/\sqrt{1-\kappa}$, where $W$ is a standard Gaussian random variable independent of $B_0$. We obtain the effective $\pi_\mu$ by dividing the $T_j$'s by $\sigma/\sqrt{1-\kappa}$.

\item \textit{Distilled lasso}.
Let $T_j=(\ndata{Y}-\ndata{X}_{\text{-}j}\hat\beta_\lambda^{(\text{-}j)})^\top\ndata{X}_j$, where 
$\hat\beta_\lambda^{(\text{-}j)}$ is the lasso coefficient of fitting lasso with parameter $\lambda$ on $\ndata{Y}$ against $\ndata{X}_{\text{-}j}$. Our $F_j$ in this case is the CDF of $\mathcal N(0,\|\ndata{Y}-\ndata{X}_{\text{-}j}\hat\beta_\lambda^{(\text{-}j)}\|_2^2/n)$. We now check the conditions for Theorem~\ref{theorem:BH-AdaPT-CRT}.

\begin{enumerate}
    \item It is clear that $1-F_j(T_j)\sim\Unif[0,1]$ for the null variables. By Lemma~\ref{lemma:theorem-training-loss}, we have for any $j$, $F_j(t)\cip\Phi_{\lambda^2/\alpha_\lambda^2}(t)$. We will find $\mathcal L(T^{(1)})$ in part (c).
    \item This is true because $\beta_j$'s are i.i.d.
    \item We verify this condition by introducing Lemma~\ref{lemma:theorem-BH-distilled}.
\end{enumerate}

\begin{lemma}
\label{lemma:theorem-BH-distilled}
In Setting~\ref{model:lr-iid}, for distinct $j$ and $k$ we have
\begin{equation*}
\left(\begin{array}{c}
T_j\\
T_k
  \end{array}\right)-\frac{\lambda}{\alpha_\lambda\tau_\lambda}\left(\begin{array}{c}
\sqrt n\beta_j\\
\sqrt n\beta_k
  \end{array}\right)\cid\mathcal N\left(\left(\begin{array}{c}
0\\
0
  \end{array}\right),\frac{\lambda^2}{\alpha_\lambda^2}I_2\right).
\end{equation*}
\end{lemma}

To sum up, the $p$-values we obtain from the distilled lasso statistic satisfy the conditions of Theorem~\ref{theorem:BH-AdaPT-CRT} with $F^{(1)}=\Phi_{\lambda^2/\alpha_\lambda^2}$ and $T^{(1)}\sim \lambda B_0/(\alpha_\lambda\tau_\lambda)+\lambda W/\alpha_\lambda$, where $W$ is a standard Gaussian random variable independent of $B_0$. We obtain the effective $\pi_\mu$ by dividing the $T_j$'s by $\lambda/\alpha_\lambda$.

\end{enumerate}
\end{proof}

\begin{lemma}
\label{lemma:gaussian-covariate-conditional}
In Setting~\ref{model:lr-iid}, let $(X^\top,Y)$ represent a random row vector that has the same distribution as one generic row of $[\ndata X, \ndata Y]$ conditional on $\beta$, then
\begin{equation}
X\mid Y,\beta\sim\mathcal N\left(\frac{Y}{\|\beta\|^2+\sigma^2}\beta,I-\frac{1}{\|\beta\|^2+\sigma^2}\beta\beta^\top\right)
\label{equation:conditional-row}
\end{equation}
and
\begin{equation}
\ndata{X}^\top\ndata{Y}\mid\ndata{Y},\beta\sim\mathcal N\left(\frac{\|\ndata{Y}\|^2}{\|\beta\|^2+\sigma^2}\beta,\|\ndata{Y}\|^2\left(I-\frac{1}{\|\beta\|^2+\sigma^2}\beta\beta^\top\right)\right).
\label{equation:conditional-mc}
\end{equation}
\end{lemma}
\begin{proof}[Proof of Lemma~\ref{lemma:gaussian-covariate-conditional}]
Jointly, we have
\begin{equation*}
\left(\begin{array}{c}
X\\
Y
  \end{array}\right)\mid\beta\sim\mathcal N\left(\left(\begin{array}{c}
0\\
0
  \end{array}\right),\left[\begin{array}{cc}
I_p & \beta\\
\beta^\top & \|\beta\|^2_2+\sigma^2\\
  \end{array}\right]\right).
\end{equation*}
Apply the formula of the conditional Gaussian distribution and we get Equation~\eqref{equation:conditional-row}. Thus,
\begin{equation*}
YX\mid Y,\beta\sim\mathcal N\left(\frac{Y^2}{\|\beta\|^2+\sigma^2}\beta,Y^2\left(I-\frac{1}{\|\beta\|^2+\sigma^2}\beta\beta^\top\right)\right).
\end{equation*}
We notice that the left hand side of Equation~\eqref{equation:conditional-mc} is just a summation of $n$ independent Gaussian random vectors with their distributions given by the above formula, and the validity of Equation~\eqref{equation:conditional-mc} then follows.
\end{proof}

\begin{proof}[Proof of Lemma~\ref{lemma:wishart-char}]
Let $T_j=n^{-1/2}\ndata{X}_j^\top\ndata{Y}$. Applying Lemma~\ref{lemma:gaussian-covariate-conditional}, we have that for $j\ne k$,
\begin{multline*}
\left(\begin{array}{c}
T_j-\sqrt n\beta_j\\
T_k-\sqrt n\beta_k\\
\end{array}\right)\mid\ndata{Y},\beta\\
\sim\mathcal N\left(\left(\frac{\|\ndata{Y}\|^2/n}{\|\beta\|^2+\sigma^2}-1\right)\left(\begin{array}{c}
\sqrt n\beta_j\\
\sqrt n\beta_k\\
\end{array}\right),\frac{\|\ndata{Y}\|^2}{n}\left[\begin{array}{cc}
1-\frac{\beta_j^2}{\|\beta\|^2+\sigma^2} & -\frac{\beta_j\beta_k}{\|\beta\|^2+\sigma^2}\\
-\frac{\beta_j\beta_k}{\|\beta\|^2+\sigma^2} & 1-\frac{\beta_k^2}{\|\beta\|^2+\sigma^2}\\
\end{array}\right]\right),
\end{multline*}
which converges in distribution to $\mathcal N(0,(\sigma^2+\kappa\e[B_0^2])I_2)$ because
\begin{equation*}
    \|\ndata{Y}\|^2/n\cip\sigma^2+\kappa\e[B_0^2],
\end{equation*}
\begin{equation*}
    \|\beta\|^2\cip\kappa\e[B_0^2],
\end{equation*}
\begin{equation*}
\beta_j^2,\beta_k^2,\beta_j\beta_k\cip0,
\end{equation*}
and $\sqrt n\beta_j$'s are universally bounded.
\end{proof}

\begin{proof}[Proof of Lemma~\ref{lemma:theorem-BH-distilled}]
To use the results in \citet{bayati2011lasso}, we apply the following re-normalization to Setting~\ref{model:lr-iid}: assume $\ndata{X}$ is divided by $\sqrt n$ and $\beta$ is multiplied by $\sqrt n$. As explained in the proof of Lemma~\ref{lemma:theorem-training-loss}, we additionally assume that $\varepsilon_i$'s and $X_{ij}$'s do not change with $n,p$ as long as $n\ge i$ and $p\ge j$, which does not change the distribution of Setting~\ref{model:lr-iid} for each fixed pair $(n,p)$.

Let $\hat{\ndata{Y}}^{(\text{-}j)}=\ndata{X}_{\text{-}j}\hat\beta_\lambda^{(\text{-}j)}$, where $\ndata{X}_{\text{-}j}$ is $\ndata{X}$ with the $j$th column removed, and $\hat\beta_\lambda^{(\text{-}j)}$ is the lasso coefficient from regressing $\ndata{Y}$ on $\ndata{X}_{\text{-}j}$ with penalty parameter $\lambda$. We replace $j,k$ with $1,2$ in the proof, which we can do due to exchangeability. Note that
\begin{equation*}
\left(\begin{array}{c}
T_1\\
T_2
  \end{array}\right)=\left(\begin{array}{c}
(\ndata{Y}-\hat{\ndata{Y}}^{(\text{-}1)})^\top\ndata{X}_1\\
(\ndata{Y}-\hat{\ndata{Y}}^{(\text{-}2)})^\top\ndata{X}_2\\
  \end{array}\right).
\end{equation*}
We first consider the statistic
\begin{equation*}
\left(\begin{array}{c}
\tilde T_1\\
\tilde T_2
  \end{array}\right)=\left(\begin{array}{c}
(\ndata{Y}-\hat{\ndata{Y}}^{(\text{-}(1:2))})^\top\ndata{X}_1\\
(\ndata{Y}-\hat{\ndata{Y}}^{(\text{-}(1:2))})^\top\ndata{X}_2\\
  \end{array}\right),
\end{equation*}
where $\hat{\ndata{Y}}^{(\text{-}(1:2))}=\ndata{X}_{\text{-}(1:2)}\hat\beta_\lambda^{(\text{-}(1:2))}$, $\ndata{X}_{\text{-}(1:2)}$ is $\ndata{X}$ with its first two columns removed, and $\hat\beta_\lambda^{(\text{-}(1:2))}$ is the lasso coefficient from regressing $\ndata{Y}$ on $\ndata{X}_{\text{-}(1:2)}$ with penalty parameter $\lambda$.

Consider a random row vector $(X_1, X_2, X_{\text{-}(1:2)}^\top, Y)$ that has the same distribution of a generic row of $[\ndata{X}, \ndata{Y}]$. Applying Lemma~\ref{lemma:conditional-covariate-2} which we will introduce shortly, we have (note the re-normalization at the beginning of this proof)
\begin{equation*}
\left(\begin{array}{c}
X_1\\
X_2
  \end{array}\right)\mid X_{\text{-}(1:2)}, Y,\beta\sim\mathcal N\left(\frac{Y-X_{\text-(1:2)}^\top\beta_{\text-(1:2)}}{n\sigma^2+\beta_1^2+\beta_2^2}\left(\begin{array}{c}
\beta_1\\
\beta_2
  \end{array}\right),\frac1n\left[\begin{array}{cc}
1-\frac{\beta_1^2}{n\sigma^2+\beta_1^2+\beta_2^2} & -\frac{\beta_1\beta_2}{n\sigma^2+\beta_1^2+\beta_2^2}\\
-\frac{\beta_1\beta_2}{n\sigma^2+\beta_1^2+\beta_2^2} & 1-\frac{\beta_2^2}{n\sigma^2+\beta_1^2+\beta_2^2}\\
  \end{array}\right]\right).
\end{equation*}
It is then easy to see that (by writing $(\tilde T_1,\tilde T_2)$ as a sum of $n$ independent Gaussian random vectors)
\begin{multline*}
\left(\begin{array}{c}
\tilde T_1\\
\tilde T_2
  \end{array}\right)\mid\ndata{Y},\ndata{X}_{\text{-}(1:2)},\beta\sim\\
  \mathcal N\left(\frac{(\ndata{Y}-\ndata{X}_{\text{-}(1:2)}\hat\beta_\lambda^{(\text{-}(1:2))})^\top\varepsilon'}{n\sigma^2+\beta_1^2+\beta_2^2}\left(\begin{array}{c}
\beta_1\\
\beta_2
  \end{array}\right),\frac{\|\ndata{Y}-\ndata{X}_{\text{-}(1:2)}\hat\beta_\lambda^{(\text{-}(1:2))}\|^2}{n}\left[\begin{array}{cc}
1-\frac{\beta_1^2}{n\sigma^2+\beta_1^2+\beta_2^2} & -\frac{\beta_1\beta_2}{n\sigma^2+\beta_1^2+\beta_2^2}\\
-\frac{\beta_1\beta_2}{n\sigma^2+\beta_1^2+\beta_2^2} & 1-\frac{\beta_2^2}{n\sigma^2+\beta_1^2+\beta_2^2}\\
  \end{array}\right]\right),
\end{multline*}
where $\varepsilon'=\ndata{Y}-\ndata{X}_{\text-(1:2)}\beta_{\text-(1:2)}=\varepsilon+\beta_1\ndata{X}_1+\beta_2\ndata{X}_2$ is the effective error in the model $\ndata{Y}\sim\ndata{X}_{\text{-}(1:2)}$. Using the results from the proof of Lemma~\ref{lemma:theorem-training-loss}, we can find the limits of
\begin{equation*}
\frac{(\ndata{Y}-\ndata{X}_{\text{-}(1:2)}\hat\beta_\lambda^{(\text{-}(1:2))})^\top\varepsilon'}{n}\quad\text{and}\quad\frac{\|\ndata{Y}-\ndata{X}_{\text{-}(1:2)}\hat\beta_\lambda^{(\text{-}(1:2))}\|^2}{n},
\end{equation*}
and see that
\begin{equation*}
\left(\begin{array}{c}
\tilde T_1\\
\tilde T_2
  \end{array}\right)-\frac{\lambda}{\alpha_\lambda\tau_\lambda}\left(\begin{array}{c}
\beta_1\\
\beta_2
  \end{array}\right)\cid\mathcal N\left(\left(\begin{array}{c}
0\\
0
  \end{array}\right),\frac{\lambda^2}{\alpha_\lambda^2}\left[\begin{array}{cc}
1 & 0\\
0 & 1\\
  \end{array}\right]\right),
\end{equation*}
which is a bivariate Gaussian distribution with i.i.d. components. Now we just have to show
\begin{equation*}
\left(\begin{array}{c}
\tilde T_1\\
\tilde T_2
  \end{array}\right)-\left(\begin{array}{c}
T_1\\
T_2
  \end{array}\right)\cid0.
\end{equation*}
It suffices to show $T_1-\tilde T_1\cid0$ marginally (same for $T_2-\tilde T_2$ because of symmetry). Note that $T_1-\tilde T_1=\left(\hat{\ndata Y}^{(\text-(1:2))}-\hat{\ndata Y}^{(-1)}\right)^\top\ndata X_1$, and we can again apply Lemma~\ref{lemma:conditional-covariate-2} to get $\mathcal L(X_1\,|\, X_{\text-1},Y,\beta)$, where $(X_1, X_{\text{-}1}^\top, Y)$ has the same distribution of a generic row of $[\ndata{X}, \ndata{Y}]$. We would get
\begin{equation*}
T_1-\tilde T_1\mid\ndata{Y},\ndata{X}_{(\text-1)},\beta\sim\mathcal N\left(\frac{\beta_1}{n\sigma^2+\beta_1^2}(\hat{\ndata{Y}}^{(\text-1)}-\hat{\ndata{Y}}^{(\text-1:2)})^\top(\varepsilon+\beta_1\ndata{X}_1),\frac{\sigma^2}{n\sigma^2+\beta_1^2}\|\hat{\ndata{Y}}^{\text-1}-\hat{\ndata{Y}}^{(\text-1:2)}\|^2\right).
\end{equation*}
Now it remains to show $\|\hat{\ndata{Y}}^{(\text-1)}-\hat{\ndata{Y}}^{(\text-1:2)}\|^2/n\cip0$, which would imply the above variance and mean (use Cauchy--Schwartz) both go to zero. Note that we can simplify this problem to $\|\hat{\ndata{Y}}-\hat{\ndata{Y}}^{(\text-1)}\|^2/n\cip0$, because both regression processes ignore the first column of $\ndata X$ and we can just treat $\beta_1\ndata{X}_1$ as part of the error vector, which does not change the asymptotic distribution of the error.

Line (d) of the first decomposition used in the proof of Lemma~3.1 in \citet{bayati2011lasso} (we take their $x$ to be our $\hat\beta_\lambda^{(\text-1)}$ and their $r$ to be our $\hat\beta_\lambda-\hat\beta_\lambda^{(\text-1)}$. Here, we slightly abuse notation: $\hat\beta_\lambda^{(\text{-}1)}$ is originally $(p-1)$-dimensional, and we add a zero as its first coordinate to make it comparable with $\hat\beta_\lambda$) shows that the sum of four terms is non-positive, and immediately after it is shown that three of those terms, including $\frac{\|\hat{\ndata{Y}}-\hat{\ndata{Y}}^{(\text-1)}\|^2}{2p}$ (their $A$ is our $\ndata X$, and $\ndata X\hat\beta_\lambda=\hat{\ndata Y}$ and $\ndata X\hat\beta_\lambda^{(\text-1)}=\hat{\ndata Y}^{(\text-1)}$), are non-negative, guaranteeing that the remaining term, $\langle\text{sg}(\mathcal C,\hat\beta_\lambda^{(\text{-}1)}),\hat\beta_\lambda-\hat\beta_\lambda^{(\text{-}1)}\rangle$ is negative and has absolute value greater than each of the three non-negative terms. Thus:

\begin{equation}
\frac{\|\hat{\ndata{Y}}-\hat{\ndata{Y}}^{(\text-1)}\|^2}{2p}\le|\langle\text{sg}(\mathcal C,\hat\beta_\lambda^{(\text{-}1)}),\hat\beta_\lambda-\hat\beta_\lambda^{(\text{-}1)}\rangle|,
\label{equation:sub-gradient-inequ}
\end{equation}
where $\text{sg}(\mathcal C,\beta)$ is any subgradient of $\mathcal C(\beta)=\|\ndata Y-\ndata X\beta\|^2/2+\lambda\|\beta\|_1$, i.e.,
\begin{equation}
-\ndata X^\top(\ndata Y-\ndata X\beta)+\lambda\gamma
\label{equation:subgradient-C}
\end{equation}
for a $\gamma$ that is a subgradient of the $p$-dimensional $L_1$ norm. By the Cauchy--Schwarz inequality,
\begin{equation}
|\langle\text{sg}(\mathcal C,\hat\beta_\lambda^{(\text{-}1)}),\hat\beta_\lambda-\hat\beta_\lambda^{(\text{-}1)}\rangle|\le\sqrt{\frac{\|\text{sg}(\mathcal C,\hat\beta_\lambda^{(\text{-}1)})\|^2}{p}}\sqrt{\frac{\|\hat\beta_\lambda-\hat\beta_\lambda^{(\text{-}1)}\|^2}{p}}.
\label{equation:sub-gradient-inequ-2}
\end{equation}
Since $\hat\beta_\lambda^{(\text{-}1)}$ is a lasso solution, the Karush--Kuhn--Tucker (KKT) conditions imply
\begin{equation}
\ndata{X}_{\text-1}^\top(\ndata{Y}-\ndata X\hat\beta_\lambda^{(-1)})=\lambda\gamma^*,
\label{equation:gamma-beta-minus-1}
\end{equation}
where for $j=1,\dots,p-1$, (note again that we have added a zero to $\hat\beta_\lambda^{(\text{-}1)}$ to make it $p$-dimensional)
\begin{equation}
\gamma^*_j\in\left\{
\begin{aligned}
&\{1\}, &(j+1)\text{st coordinate of }\hat\beta_\lambda^{(\text-1)}>0,\\
&\{-1\}, &(j+1)\text{st coordinate of }\hat\beta_\lambda^{(\text-1)}<0,\\
&[-1,1], &(j+1)\text{st coordinate of }\hat\beta_\lambda^{(\text-1)}=0.\\
\end{aligned}\right.
\label{equation:gamma-subgradient}
\end{equation}
Since the first coordinate of $\hat\beta_\lambda^{(\text{-}1)}$ is zero, directly from the definition of a subgradient \eqref{equation:subgradient-C}, the first coordinate of $\text{sg}(\mathcal C,\hat\beta_\lambda^{(\text{-}1)})$ can be any number in
\begin{equation}
\left[-\ndata{X}_{1}^\top(\ndata{Y}-\ndata X\hat\beta_\lambda^{(\text-1)})-\lambda,-\ndata{X}_{1}^\top(\ndata{Y}-\ndata X\hat\beta_\lambda^{(\text-1)})+\lambda\right].
\label{equation:subg-1}
\end{equation}
For the remaining $(p-1)$ coordinates, they can be $-\ndata{X}_{\text-1}^\top(\ndata{Y}-\ndata X\hat\beta_\lambda^{(\text-1)})+\lambda\gamma$ for any $\gamma$ that satisfies \eqref{equation:gamma-subgradient} (i.e., a subgradient of the $(p-1)$-dimensional $L_1$ norm at $\hat\beta_\lambda^{(\text-1)}$), and specifically, we can let $\gamma$ be the one that satisfies \eqref{equation:gamma-beta-minus-1}, so that the subgradient in these $(p-1)$ dimensions cancels to $0$. This way, we have defined a $\text{sg}(\mathcal C,\hat\beta_\lambda^{(\text{-}1)})$ so that its first coordinate is (take the midpoint of \eqref{equation:subg-1})
\begin{equation}
-(\ndata{Y}-\underbrace{\hat{\ndata{Y}}^{(\text-1)}}_{=\ndata X\hat\beta_\lambda^{(\text-1)}})^\top\ndata{X}_1
\label{equation:sg-1st-coord}
\end{equation}
and all other coordinates are zero. Note that
\begin{equation*}
\mathcal L\left(-(\ndata{Y}-\hat{\ndata{Y}}^{(\text-1)})^\top\ndata{X}_1\mid\ndata{Y},\ndata{X}_{\text-1}\right)=\mathcal N(0,\|\ndata{Y}-\hat{\ndata{Y}}^{(\text-1)}\|^2/n)
\end{equation*}
and $\|\ndata{Y}-\hat{\ndata{Y}}^{(\text-1)}\|^2/n\cas\lambda^2/\alpha_\lambda^2$ by Lemma~\ref{lemma:theorem-training-loss}. Thus, \eqref{equation:sg-1st-coord} converges to $\mathcal N(0,\lambda^2/\alpha_\lambda^2)$ in distribution.
This way, the squared $L_2$ norm of the selected $\text{sg}(\mathcal C,\hat\beta_\lambda^{(\text-1)})$ divided by $p$ converges to zero. On the other hand,
\begin{equation*}
\frac{\|\hat\beta_\lambda-\hat\beta_\lambda^{(\text-1)}\|^2}{p}\le\frac2p(\|\hat\beta_\lambda\|^2+\|\hat\beta_\lambda^{(\text-1)}\|^2),
\end{equation*}
and the right hand side converges to a constant as a corollary of Theorem~1.5 in \citet{bayati2011lasso}. Now we get $\|\hat{\ndata{Y}}^{(\text-1)}-\hat{\ndata{Y}}^{(\text-1:2)}\|^2/n\cip0$ from \eqref{equation:sub-gradient-inequ} and \eqref{equation:sub-gradient-inequ-2}, because $n/p$ converges to a positive constant.
\end{proof}

\begin{lemma}
Let $W_1$ and $W_2$ be independent $q$-dimensional and $r$-dimensional standard multivariate Gaussian random vectors. Let $Y\mid W_1,W_2\sim\mathcal N(W_1^\top\zeta_q+W_2^\top\zeta_r,\sigma^2)$. Then,
\begin{equation}
W_1\mid\left(\begin{array}{c}
W_2\\
Y
  \end{array}\right)\sim\mathcal N\left(\frac{Y-W_2^\top\zeta_r}{\sigma^2+\|\zeta_q\|^2_2}\zeta_q,I_q-\frac{1}{\sigma^2+\|\zeta_q\|_2^2}\zeta_q^\top\zeta_q\right).
\label{equation:conditional-covariate-part}
\end{equation}
\label{lemma:conditional-covariate-2}
\end{lemma}
\begin{proof}[Proof of Lemma~\ref{lemma:conditional-covariate-2}]
Jointly, we have
\begin{equation*}
\left(\begin{array}{c}
W_1\\
W_2\\
Y
  \end{array}\right)\sim\mathcal N\left(\left(\begin{array}{c}
0\\
0\\
0
  \end{array}\right),\left[\begin{array}{ccc}
I_q & 0 & \zeta_q\\
0 & I_r & \zeta_r\\
\zeta_q^\top & \zeta_r^\top & \|\zeta_q\|_2^2+\|\zeta_r\|_2^2+\sigma^2\\
  \end{array}\right]\right).
\end{equation*}
Note that
\begin{equation*}
\left[\begin{array}{cc}
I_r & \zeta_r\\
\zeta_r^\top & \|\zeta_q\|_2^2+\|\zeta_r\|_2^2+\sigma^2\\
  \end{array}\right]^{-1}=\left[\begin{array}{cc}
I_r+\frac{\zeta_r\zeta_r^\top}{\sigma^2+\|\zeta_q\|_2^2} & -\frac{\zeta_r}{\sigma^2+\|\zeta_q\|_2^2}\\
-\frac{\zeta_r^\top}{\sigma^2+\|\zeta_q\|_2^2} & \frac{1}{\sigma^2+\|\zeta_q\|_2^2}\\
  \end{array}\right].
\end{equation*}
Thus, directly apply the formula for the conditional Gaussian distribution and we have \eqref{equation:conditional-covariate-part}.
\end{proof}

\begin{thm}
Let $J_0$ and $J_1$ form a partition of $\{1,2,\dots,p\}$ and $|J_0|/p\to\gamma\in(0,1)$. Consider $p$ random variables $W_j$, which can be thought of as the $W_j$'s in a knockoffs procedure. Assume the following conditions.   
\begin{enumerate}
    \item For $j\in J_0$, $W_j\cid W^{(0)}\sim G^{(0)}$ and for $j\in J_1$, $W_j\cid W^{(1)}\sim G^{(1)}$. $G^{(0)}$ and $G^{(1)}$ are deterministic CDFs of random variables with a common support which is connected and symmetric around $0$, and continuous densities on that support.
    \item Within $J_0$ or $J_1$, the $W_j$'s are exchangeable.
    \item For distinct $j_1,j_2\in J_0$ and distinct $j_3,j_4\in J_1$, the following two pairs of random variables are asymptotically pairwise independent: $(W_{j_1},W_{j_2})$ and $(W_{j_3},W_{j_4})$. That is, both pairs converge in distribution to a bivariate random vector (not necessarily the same random vector) with independent components.
\end{enumerate}
Let (min over an empty set is defined to be infinity)
\begin{equation*}
w_\textnormal{KF}=\min\{w\ge0:g(w)\le q\},
\end{equation*}
where
\begin{equation*}
g(w)=\frac{\gamma G^{(0)}(-w)+(1-\gamma)(G^{(1)}(-w))}{\gamma (1-G^{(0)}(w))+(1-\gamma)(1-G^{(1)}(w))}.
\end{equation*}
Then for almost every $q\in(0,1)$, at least one of the following cases is true: (i) $w_\textnormal{KF}>0$, $g'(w_\textnormal{KF})\ne0$, (ii) $w_\textnormal{KF}=0$, $g'(0)<0$, (iii) $w_\textnormal{KF}=0$, $g(0)<q$, or (iv) $w_\textnormal{KF}=\infty$. In cases (i), (ii), or (iii), for the knockoff filter \eqref{equation:knockoff-adapt} at level $q$ applied to $W_1,\dots,W_p$, the FDP and realized power converge in probability to
\begin{equation*}
\frac{\gamma G^{(0)}(-w_\textnormal{KF})}{\gamma (1-G^{(0)}(w_\textnormal{KF}))+(1-\gamma)(1-G^{(1)}(w_\textnormal{KF}))}\quad\text{and}\quad1-G^{(1)}(w_\textnormal{KF}),
\end{equation*}
respectively. In case (iv), the realized power converges in probability to $0$.
\label{theorem:knockoff-power}
\end{thm}

\begin{proof}[Proof of Theorem~\ref{theorem:knockoff-power}]
If $q>g(0)$, then (iii) holds. If $q<\inf\{g(w):w\ge0\}$, then (iv) holds. If $q\in(\inf\{g(w):w\ge0\}, g(0))$, then because $g$ is continuous, we can see that $w_\text{KF}\in(0,\infty)$ (note that we are considering minimum and $[0,\infty)$ is closed on the left, so the minimum must exist and not equal to $0$). And in this case, $g(w_\text{KF})=q$, since otherwise $g(w_\text{KF})<q$ and $w_\text{KF}$ could be smaller because of $g$'s continuity. Next we show for almost every $q\in(\inf\{g(w):w\ge0\}, g(0))$, condition (i) is met. We just need to show that the set $\{g(w):g'(w)=0\}$ has measure zero, which is a simple application of Sard's theorem (Lemma~\ref{lemma:sard}, take $n=m=k=1$, $f(x)=g(a\arctan(x))$ for a suitable $a$ such that $a\arctan(x)$ matches the support of $G^{(0)}$ and $G^{(1)}$ when that support is finite; otherwise just take $f=g$), where $g'$ is continuous because $G^{(0)}$ and $G^{(1)}$ have continuous densities with a common support.

In case (i), we have established that $g(w_\text{KF})=q$. Since $g'(w_\text{KF})\ne0$, we must have $g'(w_\text{KF})<0$, since otherwise $w_\text{KF}$ could also be smaller. Thus, in cases (i) and (ii), $g'(w_\text{KF})<0$ and for any sufficiently small $\varepsilon>0$ there exists a point $w^*$ in $(w_\text{KF},w_\text{KF}+\varepsilon)$ such that $g(w^*)<q$. In case (iii), since $g$ is continuous and $g(0)<q$, it also holds that for any sufficiently small $\varepsilon>0$ there exists a point $w^*$ in $(w_\text{KF},w_\text{KF}+\varepsilon)$ such that $g(w^*)<q$.

Let
\begin{equation*}
\hat w_p=\min\{w>0:g_p(w)\equiv\frac{1/p+\#\{j:W_j\le-w\}/p}{\#\{j:W_j\ge w\}/p}\le q\}.\footnote{Formally, we only take the minimum over $w\in\{|W_j|: W_j\ne0\}$. Such a difference is important when $g_p(w)\le q$ for all $w>0$. The term $1/p$ does not matter asymptotically, in that we can still consider the numerator as an empirical CDF of the $W_j$'s.}
\end{equation*}
We analyze cases (i), (ii), and (iii). We begin by showing $\hat w_p\cip w_\text{KF}$. Take any sufficiently small $\varepsilon>0$. We have established that there exists a point $w^*$ in $(w_\text{KF},w_\text{KF}+\varepsilon)$ such that $g(w^*)<q$. Then
\begin{equation*}
\begin{aligned}
\p(\hat w_p<w_\text{KF}+\varepsilon)&\ge \p(\hat w_p\le w^*)\\
&\ge\p(g_p(w^*)\le q)\\
&\ge\p(|g_p(w^*)-g(w^*)|<|g(w^*)-q|)\to1.
\end{aligned}
\end{equation*}
In case (iii), we get $\p(\hat w_p\ge w_\text{KF}=0)=1$ for free and the proof is concluded. Next we consider cases (i) and (ii) and assume $\varepsilon<w_\text{KF}$. Choose $\delta_1\in\left(0,\min(1-G^{(0)}(w_\text{KF}-\varepsilon),1-G^{(1)}(w_\text{KF}-\varepsilon))\right)$. Let $\delta_3=\min\{g(w)-q:0\le w\le w_\text{KF}-\varepsilon\}$, which is positive since otherwise $g$ could attain a value no more than $q$ in $[0,w_\text{KF}-\varepsilon]$, violating $w_\text{KF}$'s definition.
Now observe that the function $\frac{\gamma x+(1-\gamma)y}{\gamma z+(1-\gamma)t}$ is continuous in $(x,y,z,t)$ on $\{(x,y,z,t)\in[0,1]^4:z,t\ge\min(1-G^{(0)}(w_\text{KF}-\varepsilon),1-G^{(1)}(w_\text{KF}-\varepsilon))-\delta_1\}$, and thus uniformly continuous.
Choose $\delta_2\in(0,\delta_1)$ such that whenever $(x,y,z,t),(x',y',z',t')\in\{(x,y,z,t)\in[0,1]^4:z,t\ge\min(1-G^{(0)}(\tau_\text{KF}-\varepsilon),1-G^{(1)}(\tau_\text{KF}-\varepsilon))\}$ and $|x-x'|,|y-y'|,|z-z'|,|t-t'|<\delta_2$, $|\frac{\gamma x+(1-\gamma)y}{\gamma z+(1-\gamma)t}-\frac{\gamma x'+(1-\gamma)y'}{\gamma z'+(1-\gamma)t'}|<\delta_3$.

By Lemma~\ref{lemma:unif-conv-in-prob}, with probability converging to $1$, all the CDFs fall within a $\delta_2$-neighborhood around the limit CDFs, and by the previously shown uniform continuity, $|g_p(w)-g(w)|<\delta_3$ for $w\in[0,w_\text{KF}-\varepsilon]$. By the definition of $\delta_3$, this means with probability converging to $1$, $g_p(w)>q$ for all $w\in[0,w_\text{KF}-\varepsilon]$. Now we have
\begin{equation*}
\p(\hat w_p>w_\text{KF}-\varepsilon)=\p(g_p(w)>q,\forall w\in[0, w_\text{KF}-\varepsilon])\to1.
\end{equation*}
Combining the results, we have
\begin{equation*}
\lim_{p\to\infty}\p(|\hat w_p-w_\text{KF}|<\varepsilon)=1.
\end{equation*}
Similar to the proof of Theorem~\ref{theorem:BH-AdaPT-CRT}, we can show that
\begin{equation*}
\frac{\#\{j\in J_0:W_j\le t\}}{|J_0|}\cip G^{(0)}(t)\quad\text{and}\quad\frac{\#\{j\in J_1:W_j\le t\}}{|J_1|}\cip G^{(1)}(t),
\end{equation*}
and the convergence is uniform over $t\in\rr$ by Lemma~\ref{lemma:unif-conv-in-prob}. The result then follows by noticing that the FDP and realized power are
\begin{equation*}
\frac{\#\{j\in J_0:W_j\ge\hat w_p\}}{\#\{j:W_j\ge\hat w_p\}}\quad\text{and}\quad\frac{\#\{j\in J_1:W_j\ge\hat w_p\}}{|J_1|},
\end{equation*}
respectively.

Last, we look at case (iv). Similar to the end of the proof of Theorem~\ref{theorem:BH-AdaPT-CRT}, we can show that in this case the asymptotic realized power is $0$ by showing that for any $M>0$, $\p(w_\text{KF}\ge M)\to1$.
\end{proof}

\begin{lemma}
Let $J_0$ and $J_1$ form a partition of $\{1,2,\dots,p\}$ and $|J_0|/p\to\gamma\in(0,1)$. If $W_j=f(T_j,\tilde T_j)$ for a continuous antisymmetric function $f$, then the following conditions imply conditions 1, 2, and 3 of Theorem~\ref{theorem:knockoff-power}.
\begin{enumerate}
    \item For $j\in J_0$, $(T_j,\tilde T_j)\cid T^{(0)}$ and for $j\in J_1$, $(T_j,\tilde T_j)\cid T^{(1)}$. $T^{(0)}$, $T^{(1)}$, and $f$ are such that the distributions of $f(T^{(0)})$ and $f(T^{(1)})$ have a common support and continuous densities.
    \item Within $J_0$ or $J_1$, the $(T_j,\tilde T_j)$'s are exchangeable.
    \item For distinct $j_1,j_2\in J_0$ and distinct $j_3,j_4\in J_1$, the following two pairs of random vectors are asymptotically pairwise independent:
    \begin{equation*}
    \left(\left(  \begin{array}{c}
T_{j_1}\\
\tilde T_{j_1}\\
  \end{array}\right),\left(  \begin{array}{c}
T_{j_2}\\
\tilde T_{j_2}\\
  \end{array}\right)\right)\quad\text{and}\quad\left(\left(  \begin{array}{c}
T_{j_3}\\
\tilde T_{j_3}\\
  \end{array}\right),\left(  \begin{array}{c}
T_{j_4}\\
\tilde T_{j_4}\\
  \end{array}\right)\right).
    \end{equation*}
\end{enumerate}
\label{lemma:knockoff-conditions}
\end{lemma}

The proof of Lemma~\ref{lemma:knockoff-conditions} is immediate and thus omitted.

\noindent\textbf{Theorem~\ref{theorem:coro-knockoff}.} \textit{In Setting~\ref{model:lr-iid}, for almost every $q\in(0,1)$, knockoffs with $\tilde{\mathbf{X}}$ an i.i.d. copy of $\ndata{X}$ and the antisymmetric function $f(x,y)=x-y$ at level $q$ with marginal covariance or OLS test statistic has the following one-sided effective $\pi_\mu$'s with respect to the AdaPT procedure at level $q$:
\begin{enumerate}
    \item For the marginal covariance statistic, the effective $\pi_\mu$ is the distribution of $\frac{1}{\sqrt{2(\sigma^2+\kappa\e[B_0^2])}}B_0$.
    \item For the OLS statistic, assuming $\kappa<1/2$, the effective $\pi_\mu$ is the distribution of $\frac{\sqrt{1-2\kappa}}{\sqrt{2\sigma^2}}B_0$.
\end{enumerate}}

\begin{proof}[Proof of Theorem~\ref{theorem:coro-knockoff}]

Similar to the proof of Theorem~\ref{theorem:coro-BH-CRT}, we analyze the two statistics separately.

\begin{enumerate}
\item \textit{Marginal covariance}.
Consider using $T_j=n^{-1/2}\ndata{X}_j^\top\ndata{Y}$ and $\tilde T_j=n^{-1/2}\tilde{\mathbf X}_j^\top\ndata{Y}$. To utilize Lemma~\ref{lemma:knockoff-conditions}, we introduce Lemma~\ref{lemma:mc-knockoff}. The proof is based on a tedious yet straightforward computation of the characteristic function of the Wishart distribution.
\begin{lemma}
In Setting~\ref{model:lr-iid} with the knockoffs procedure that takes $\tilde{\mathbf{X}}$ to be an i.i.d. copy of $\ndata{X}$, let $T_j=n^{-1/2}\ndata{X}_j^\top\ndata{Y}$ and $\tilde T_j=n^{-1/2}\tilde{\mathbf X}_j^\top\ndata{Y}$. We have for $j\ne k$,
\begin{equation*}
    \left(\begin{array}{c}
T_j-\sqrt n\beta_j\\
T_k-\sqrt n\beta_k\\
\tilde T_j\\
\tilde T_k\\
  \end{array}\right)\cid\mathcal N\left(0,\left(\sigma^2+\kappa\e[B_0^2]\right)I_4\right).
\end{equation*}
\label{lemma:mc-knockoff}
\end{lemma}
This means that asymptotically, we can think of all the $T_j$'s and $\tilde T_j$'s as independent, $\tilde T_j\sim\mathcal N(0,\sigma^2+\kappa\e[B_0^2])$ and $T_j\sim B_0+\sqrt{\sigma^2+\kappa\e[B_0^2]}Z$, where $Z\sim\mathcal N(0,1)$ is independent of $B_0$. Thus, we can think of $W_j=T_j-\tilde T_j$ and $W_k=T_k-\tilde T_k$ as independent for distinct $j$ and $k$, with distribution $B_0+\sqrt{2(\sigma^2+\kappa\e[B_0^2]})Z$,  where $Z\sim\mathcal N(0,1)$ is independent of $B_0$. We obtain the effective $\pi_\mu$ by dividing the $W_j$'s by $\sqrt{2(\sigma^2+\kappa\e[B_0^2])}$.

\item \textit{OLS}.
We consider letting $T_j=\sqrt n\hat\beta_j$ and $\tilde T_j=\sqrt n\hat\beta_{j+p}$, where $\kappa<1/2$ and $\hat\beta$ is the OLS coefficient of $\ndata{Y}$ against $[\ndata{X},\tilde{\ndata{X}}]$. We just check the conditions in Lemma~\ref{lemma:knockoff-conditions}. The second condition is obvious. For the other two conditions, notice that for $j\ne k$,
\begin{equation*}
    \left(\begin{array}{c}
\sqrt n\hat\beta_j\\
\sqrt n\hat\beta_{j+p}\\
\sqrt n\hat\beta_k\\
\sqrt n\hat\beta_{k+p}
  \end{array}\right)-\left(\begin{array}{c}
\sqrt n\beta_j\\
0\\
\sqrt n\beta_k\\
0\\
  \end{array}\right)\mid\ndata{X},\tilde{\mathbf{X}}\sim\mathcal N\left(\left(\begin{array}{c}
0\\
0\\
0\\
0\\
  \end{array}\right),\sigma^2\left[\left(\begin{array}{cc}
\ndata{X}^\top\ndata{X} & \ndata{X}^\top\tilde{\mathbf{X}}\\
\tilde{\mathbf{X}}^\top\ndata{X} & \tilde{\mathbf{X}}^\top\tilde{\mathbf{X}} \\
  \end{array}\right)^{-1}\right]_{(j,j+p,k,k+p),(j,j+p,k,k+p)}\right).
\end{equation*}
Thus, we can show
\begin{equation*}
    \left(\begin{array}{c}
\sqrt n\hat\beta_j\\
\sqrt n\hat\beta_{j+p}\\
\sqrt n\hat\beta_k\\
\sqrt n\hat\beta_{k+p}
  \end{array}\right)-\left(\begin{array}{c}
\sqrt n\beta_j\\
0\\
\sqrt n\beta_k\\
0\\
  \end{array}\right)\cid\mathcal N\left(\left(\begin{array}{c}
0\\
0\\
0\\
0\\
  \end{array}\right),\frac{\sigma^2}{1-2\kappa}I_4\right),
\end{equation*}
once we verify that any $4\times4$ sub-diagonal matrix of
\begin{equation*}
\left(\begin{array}{cc}
\ndata{X}^\top\ndata{X} & \ndata{X}^\top\tilde{\mathbf{X}}\\
\tilde{\mathbf{X}}^\top\ndata{X} & \tilde{\mathbf{X}}^\top\tilde{\mathbf{X}} \\
  \end{array}\right)^{-1}
\end{equation*}
converges in probability to $(1-2\kappa)^{-1}I_4$, which follows directly from a computation of the first and second moments of the inverse Wishart distribution. This means that asymptotically, we can think of all the $T_j$'s and $\tilde T_j$'s as independent, $\tilde T_j\sim\mathcal N(0,\sigma^2/(1-2\kappa))$ and $T_j\sim B_0+\sigma Z/\sqrt{1-2\kappa}$, where $Z\sim\mathcal N(0,1)$ is independent of $B_0$. Thus, we can think of $W_j=T_j-\tilde T_j$ and $W_k=T_k-\tilde T_k$ as independent for distinct $j$ and $k$, with distribution $B_0+\sqrt2\sigma Z/\sqrt{1-2\kappa}$, where $Z\sim\mathcal N(0,1)$ is independent of $B_0$. We obtain the effective $\pi_\mu$ by dividing the $W_j$'s by $\sqrt2\sigma/\sqrt{1-2\kappa}$.
\end{enumerate}
\end{proof}

\begin{proof}[Proof of Lemma~\ref{lemma:mc-knockoff}]
By Lemma~\ref{lemma:gaussian-covariate-conditional},
\begin{equation*}
\ndata{X}^\top\ndata{Y}\mid\ndata{Y},\beta\sim\mathcal N\left(\frac{\|\ndata{Y}\|^2}{\|\beta\|^2+\sigma^2}\beta,\|\ndata{Y}\|^2\left(I-\frac{1}{\|\beta\|^2+\sigma^2}\beta\beta^\top\right)\right),
\end{equation*}
and since $\tilde{\ndata X}^\top\ndata Y\mid\ndata Y,\beta\sim\mathcal N(0,\|\ndata Y\|^2I_p)$, $\tilde{\ndata X}\ci\ndata X\mid\ndata Y$,
\begin{equation*}
\left(\begin{array}{c}
\ndata{X}^\top\ndata{Y}\\
\tilde{\ndata{X}}^\top\ndata{Y}\\
\end{array}\right)\mid\ndata{Y},\beta\sim\mathcal N\left(\frac{\|\ndata{Y}\|^2}{\|\beta\|^2+\sigma^2}\left(\begin{array}{c}
\beta\\
0\\
\end{array}\right),\|\ndata{Y}\|^2\left[\begin{array}{cc}
   I-\frac{1}{\|\beta\|^2+\sigma^2}\beta\beta^\top  & 0 \\
   0  & I\\
\end{array}\right]\right)
\end{equation*}
Let $T_j=n^{-1/2}\ndata{X}_j^\top\ndata{Y}$ and $\tilde T_j=n^{-1/2}\tilde{\ndata{X}}_j^\top\ndata{Y}$. For $j\ne k$,
\begin{multline*}
\left(\begin{array}{c}
T_j-\sqrt n\beta_j\\
T_k-\sqrt n\beta_k\\
\tilde T_j\\
\tilde T_k\\
\end{array}\right)\mid\ndata{Y},\beta\\
\sim\mathcal N\left(\left(\frac{\|\ndata{Y}\|^2/n}{\|\beta\|^2+\sigma^2}-1\right)\left(\begin{array}{c}
\sqrt n\beta_j\\
\sqrt n\beta_k\\
0\\
0\\
\end{array}\right),\frac{\|\ndata{Y}\|^2}{n}\left[\begin{array}{cccc}
1-\frac{\beta_j^2}{\|\beta\|^2+\sigma^2} & -\frac{\beta_j\beta_k}{\|\beta\|^2+\sigma^2} & 0 & 0\\
-\frac{\beta_j\beta_k}{\|\beta\|^2+\sigma^2} & 1-\frac{\beta_k^2}{\|\beta\|^2+\sigma^2} & 0 & 0\\
0 & 0 & 1 & 0\\
0 & 0 & 0 & 1\\
\end{array}\right]\right),
\end{multline*}
which converges in distribution to $\mathcal N(0,(\sigma^2+\kappa\e[B_0^2])I_4)$.
\end{proof}

\begin{thm}
In Setting~\ref{model:lr-iid} with the knockoff procedure that takes $\tilde{\mathbf{X}}$ to be an i.i.d. copy of $\ndata{X}$, let $W_j$ be $f(\sqrt n\hat\beta_j^\lambda,\sqrt n\hat\beta_{j+p}^\lambda)$ for a continuous antisymmetric function $f$ that is not almost everywhere $0$, where $\hat\beta^\lambda$ is the lasso estimate with penalty parameter $\lambda$. Assume $\alpha_\lambda$ and $\tau_\lambda$ are defined as in \citet{bayati2011lasso} (note that the number of covariates is $2p$ instead of $p$). Let $G^{(0)}$ be the CDF of $f(\eta(\tau_\lambda Z_1;\alpha_\lambda\tau_\lambda),\eta(\tau_\lambda Z_2;\alpha_\lambda\tau_\lambda))$ and $G^{(1)}$ be the CDF of $f(\eta(B^{\textnormal{alt}}_0+\tau_\lambda Z_1;\alpha_\lambda\tau_\lambda),\eta(\tau_\lambda Z_2;\alpha_\lambda\tau_\lambda))$, where $Z_1,Z_2\simiid\mathcal N(0,1)$, independent of $B_0^\textnormal{alt}$, and $B_0^\textnormal{alt}$ has the same distribution as $(B_0\mid B_0\ne0)$. Assume $f$ is such that $G^{(0)}$ and $G^{(1)}$ are CDFs that only have a point mass at $0$ and have continuous densities elsewhere. Let (min over an empty set is defined to be infinity)
\begin{equation*}
w_\textnormal{KF}=\min\{w\ge0:g(w)\le q\},
\end{equation*}
where
\begin{equation*}
g(w)=\left\{
\begin{aligned}
&\frac{\gamma G^{(0)}(-w)+(1-\gamma)(G^{(1)}(-w))}{\gamma (1-G^{(0)}(w))+(1-\gamma)(1-G^{(1)}(w))}, &w>0,\\
&\lim_{w\to0^+}g(w), &w=0.
\end{aligned}\right.
\end{equation*}
For almost every $q\in(0,1)$, one of the following cases is true:
\begin{enumerate}
    \item[(a)] $w_\textnormal{KF}>0$ and $g'(w_\textnormal{KF})\ne0$, then the FDP and realized power of the knockoffs procedure at level $q$ converge in probability to
\begin{equation*}
\frac{\gamma G^{(0)}(-w_\textnormal{KF})}{\gamma (1-G^{(0)}(w_\textnormal{KF}))+(1-\gamma)(1-G^{(1)}(w_\textnormal{KF}))}\quad\text{and}\quad1-G^{(1)}(w_\textnormal{KF}),
\end{equation*}
respectively;

\item[(b)] $w_\textnormal{KF}=0$ and either the right derivative $g'(0)<0$ or $g(0)<q$, then the FDP and realized power of the knockoffs procedure at level $q$ converge in probability to
\begin{equation*}
\lim_{w\to0^+}\frac{\gamma G^{(0)}(-w)}{\gamma (1-G^{(0)}(w))+(1-\gamma)(1-G^{(1)}(w))}\quad\text{and}\quad\lim_{w\to0^+}1-G^{(1)}(w),
\end{equation*}
respectively;

\item[(c)] $w_\textnormal{KF}=\infty$, then the realized power of the knockoffs procedure at level $q$ converges in probability to $0$.
\end{enumerate}

\label{theorem:lasso-fdr-power}
\end{thm}

\begin{proof}[Proof of Theorem~\ref{theorem:lasso-fdr-power}]
Let
\begin{equation*}
\hat w_p=\min\{w>0:g_p(w)\equiv\frac{1/p+\#\{j:f(\sqrt n\hat\beta^\lambda_{j},\sqrt n\hat\beta^\lambda_{j+p})\le-w\}/p}{\#\{j:f(\sqrt n\hat\beta^\lambda_{j},\sqrt n\hat\beta^\lambda_{j+p})\ge w\}/p}\le q\}.\footnote{Formally, we only take the minimum over $w\in\{|f(\sqrt n\hat\beta^\lambda_{j},\sqrt n\hat\beta^\lambda_{j+p})|: f(\sqrt n\hat\beta^\lambda_{j},\sqrt n\hat\beta^\lambda_{j+p})\ne0\}$. Such a difference is important when $g_p(w)\le q$ for all $w>0$. The term $1/p$ does not matter asymptotically, in that we can still consider the numerator as an empirical CDF of the $f(\sqrt n\hat\beta^\lambda_{j},\sqrt n\hat\beta^\lambda_{j+p})$'s.}
\end{equation*}
Similar to the proof of Theorem~\ref{theorem:knockoff-power}, if we can show the convergence of the empirical CDFs, we can show $\hat w_p$ converges in probability to $w_\textnormal{KF}$ and the results of the theorem then follow. Hence, we only need the results from Lemma~\ref{lemma:theorem-lasso-central}.
\end{proof}

\begin{lemma}
\label{lemma:theorem-lasso-central}
Under the setting in Theorem~\ref{theorem:lasso-fdr-power}, for any nonzero $t\in\rr$,
\begin{equation*}
\frac{1}{p}\sum_{j=1}^p\mathbb I(f(\sqrt n\hat\beta^\lambda_{j},\sqrt n\hat\beta^\lambda_{j+p})\le t)\cip\p\left(f(\eta(B_0+\tau_\lambda Z_1;\alpha_\lambda\tau_\lambda),\eta(\tau_\lambda Z_2;\alpha_\lambda\tau_\lambda))\le t\right)
\end{equation*}
and
\begin{equation*}
\frac{1}{\#\{j\in[p]:\beta_j=0\}}\sum_{j=1}^p\mathbb I(f(\sqrt n\hat\beta^\lambda_{j},\sqrt n\hat\beta^\lambda_{j+p})\le t,\beta_j=0)\cip\p\left(f(\eta(\tau_\lambda Z_1;\alpha_\lambda\tau_\lambda),\eta(\tau_\lambda Z_2;\alpha_\lambda\tau_\lambda))\le t\right),
\end{equation*}
where $B_0\sim\gamma\delta_0+(1-\gamma)\pi_1$ and $Z_1,Z_2\simiid\mathcal N(0,1)$ are independent of $B_0$. These imply
\begin{equation*}
\frac{1}{\#\{j\in[p]:\beta_j\ne0\}}\sum_{j=1}^p\mathbb I(f(\sqrt n\hat\beta^\lambda_{j},\sqrt n\hat\beta^\lambda_{j+p})\le t,\beta_j\ne0)\cip\p\left(f(\eta(B_0^\textnormal{alt}+\tau_\lambda Z_1;\alpha_\lambda\tau_\lambda),\eta(\tau_\lambda Z_2;\alpha_\lambda\tau_\lambda))\le t\right).
\end{equation*}
\end{lemma}

\begin{proof}[Proof of Lemma~\ref{lemma:theorem-lasso-central}]
To use the results in \citet{bayati2011lasso}, we apply the following re-normalization to Setting~\ref{model:lr-iid}: assume $\ndata{X}$ is divided by $\sqrt n$ and $\beta$ is multiplied by $\sqrt n$.

For simplicity of notation, we relabel the covariates, so all odd-labeled covariates correspond to real covariates, and all even-labeled covariates correspond to knockoffs. We condition on $\beta_{1:\infty}$. Note that the relabeling means only odd $\beta_j$'s correspond to draws from $\gamma\delta_0 +(1-\gamma)\pi_1$, and the even $\beta_j$'s are just zero.
\begin{equation*}
\begin{aligned}
&\e\left[\frac{1}{p}\sum_{j=1}^p\mathbb I(f(\hat\beta^\lambda_{2j-1},\hat\beta^\lambda_{2j})\le t)\right]\\
&=\p(f(\hat\beta^\lambda_{2J-1},\hat\beta^\lambda_{2J})\le t)\qquad\text{($J\sim\Unif([p])$)}\\
&=\p(f(\hat\beta^\lambda_{2J-1},\hat\beta^\lambda_{2J'})\le t)\qquad\text{($J,J'\simiid\Unif([p])$, by exchangeability)}.
\end{aligned}
\end{equation*}
We know the limit of this probability if we can show that
\begin{equation*}
(\hat\beta^\lambda_{2J-1},\hat\beta^\lambda_{2J'})\cid(\eta(B_0+\tau_\lambda Z_1;\alpha_\lambda\tau_\lambda),\eta(\tau_\lambda Z_2;\alpha_\lambda\tau_\lambda)),
\end{equation*}
where $Z_1,Z_2\simiid\mathcal N(0,1)$, independent of $B_0$.
\begin{equation*}
\p(\hat\beta^\lambda_{2J-1}\le s,\hat\beta^\lambda_{2J'}\le t)=\e[\hat F^\text{odd}_p(s)\hat F^\text{even}_p(t)],
\end{equation*}
where
\begin{equation*}
\hat F^\text{odd}_p(s)=\frac1p\sum_{j=1}^p\mathbb I(\hat\beta^\lambda_{2j-1}\le s),\hat F^\text{even}_p(t)=\frac1p\sum_{j=1}^p\mathbb I(\hat\beta^\lambda_{2j}\le t).
\end{equation*}
We need these two terms to converge in probability, which would give us asymptotic independence of $(\hat\beta^\lambda_{2J-1},\hat\beta^\lambda_{2J'})$ by convergence of CDF via the bounded convergence theorem. Note that we only have to analyze $t\ne0$, which corresponds to the continuity points. By exchangeability,
\begin{equation*}
\sum_{j=1}^p\mathbb I(\hat\beta^\lambda_{2j}\le t)\mid\sum_{j=1}^{2p}\mathbb I(\hat\beta^\lambda_{j}\le t,\beta_j=0)\sim\text{Hypergeometric}(p+p\gamma_p,{\sum_{j=1}^{2p}\mathbb I(\hat\beta^\lambda_{j}\le t,\beta_j=0)},p).
\end{equation*}
Here,
\begin{equation*}
\gamma_p=\frac{\#\{j\in[p]:\beta_{2j-1}=0\}}{p}\to\gamma.
\end{equation*}
The hypergeometric distribution divided by $p$ has mean
\begin{equation*}
\frac{\sum_{j=1}^{2p}\mathbb I(\hat\beta^\lambda_{j}\le t,\beta_j=0)}{p+p\gamma_p}
\end{equation*}
and variance
\begin{equation*}
\frac{\sum_{j=1}^{2p}\mathbb I(\hat\beta^\lambda_{j}\le t,\beta_j=0)}{p+p\gamma_p}\left(1-\frac{\sum_{j=1}^{2p}\mathbb I(\hat\beta^\lambda_{j}\le t,\beta_j=0)}{p+p\gamma_p}\right)\frac{\gamma_p}{p-1}.
\end{equation*}

\begin{lemma}
In Setting~\ref{model:lr-iid} with the knockoff procedure that takes $\tilde{\ndata X}$ to be an i.i.d. copy of $\ndata X$, assume the $\varepsilon_i$'s, $X_{ij}$'s and $\tilde X_{ij}$'s do not change with $n,p$ as long as $n\ge i$ and $p\ne j$, then
\begin{equation*}
\frac{1}{2p}\sum_{j=1}^{2p}\mathbb I(\hat\beta^\lambda_{j}\le t,\beta_j=0)\cas\frac{\gamma+1}{2}\p(\eta(\tau_\lambda Z;\alpha_\lambda\tau_\lambda)\le t)
\end{equation*}
for any $t\ne0$, where $Z\sim\mathcal N(0,1)$.
\label{lemma:lasso-amp-even}
\end{lemma}

\begin{proof}[Proof of Lemma~\ref{lemma:lasso-amp-even}]
For $\varepsilon>0$, let $\phi(x,y)=\mathbb I(x\le t,|y|\le\varepsilon)$. Take
\begin{equation*}
\phi_{1,k}(x,y)=1-\min(1,k\times\inf_{z\le t\text{ and }|w|\le\varepsilon}\|(z,w)-(x,y)\|),
\end{equation*}
and
\begin{equation*}
\phi_{2,k}(x,y)=\min(1,k\times\inf_{z\ge t\text{ or }|w|\ge\varepsilon}\|(z,w)-(x,y)\|).
\end{equation*}
It is clear that
\begin{equation*}
\phi_{2,k}(x,y)\le\phi(x,y)\le\phi_{1,k}(x,y),
\end{equation*}
and for each $k>0$, $\phi_{1,k}$ and $\phi_{2,k}$ are uniformly continuous, 
so we have almost surely \citep{bayati2011lasso},
\begin{equation}
\lim_{p\to\infty}\frac1{2p}\sum_{j=1}^{2p}\psi(\hat\beta^\lambda_j,\beta_{j})
=\e[\psi(\eta(B_\textnormal{all}+\tau_\lambda Z;\alpha_\lambda\tau_\lambda),B_\text{all})], \qquad\psi=\phi_{1,k},\phi_{2,k},
\label{eq:lasso-amp-dist}
\end{equation}
where $Z\sim N(0,1)$ independent of $B_\textnormal{all}\sim\frac{\gamma+1}2\delta_0+\frac{1-\gamma}{2}\pi_1$.

Now we assume $\varepsilon$ is such that $B_\textnormal{all}$ does not have point mass at $\varepsilon$, which holds for almost every $\varepsilon>0$, then
\begin{equation}
\begin{aligned}
&\e[\phi_{1,k}(\eta(B_\textnormal{all}+\tau_\lambda Z;\alpha_\lambda\tau_\lambda),B_\text{all})]\\
&=\p(B_\text{all}=0)\e[\phi_{1,k}(\eta(\tau_\lambda Z;\alpha_\lambda\tau_\lambda),0)]+\p(0<|B_\textnormal{all}|\le\varepsilon+\frac1k)\e[\phi_{1,k}(\eta(B_\textnormal{all}+\tau_\lambda Z;\alpha_\lambda\tau_\lambda),B_\textnormal{all})\mid0<|B_\textnormal{all}|\le\varepsilon+\frac1k]\\
&\to\frac{1+\gamma}{2}\e[\phi(\eta(\tau_\lambda Z;\alpha_\lambda\tau_\lambda),0)]+\p(0<|B_\textnormal{all}|\le\varepsilon)\e[\phi(\eta(B_\textnormal{all}+\tau_\lambda Z;\alpha_\lambda\tau_\lambda),B_\textnormal{all})\mid0<|B_\textnormal{all}|\le\varepsilon]\\
&\e[\phi_{2,k}(\eta(B_\textnormal{all}+\tau_\lambda Z;\alpha_\lambda\tau_\lambda),B_\textnormal{all})]\\
&=\p(B_\textnormal{all}=0)\e[\phi_{2,k}(\eta(\tau_\lambda Z;\alpha_\lambda\tau_\lambda),0)]+\p(0<|B_\textnormal{all}|\le\varepsilon)\e[\phi_{2,k}(\eta(B_\textnormal{all}+\tau_\lambda Z;\alpha_\lambda\tau_\lambda),B_\textnormal{all})\mid0<|B_\textnormal{all}|\le\varepsilon]\\
&\to\frac{1+\gamma}{2}\e[\phi(\eta(\tau_\lambda Z;\alpha_\lambda\tau_\lambda),0)]+\p(0<|B_\textnormal{all}|\le\varepsilon)\e[\phi(\eta(B_\textnormal{all}+\tau_\lambda Z;\alpha_\lambda\tau_\lambda),B_\textnormal{all})\mid0<|B_\textnormal{all}|\le\varepsilon]
\end{aligned}
\label{eq:sup-inf}
\end{equation}
as $k\to\infty$, by the bounded convergence theorem. Since
\begin{equation*}
\frac{1}{2p}\sum_{j=1}^{2p}\phi_{2,k}(\hat\beta^\lambda_j,\beta_j)\le\frac{1}{2p}\sum_{j=1}^{2p}\phi(\hat\beta^\lambda_j,\beta_j)\le\frac{1}{2p}\sum_{j=1}^{2p}\phi_{1,k}(\hat\beta^\lambda_j,\beta_j),
\end{equation*}
we have,
\begin{equation*}
\limsup_{p\to\infty}\frac{1}{2p}\sum_{j=1}^{2p}\phi(\hat\beta^\lambda_j,\beta_j)\le\lim_{p\to\infty}\frac{1}{2p}\sum_{j=1}^{2p}\phi_{1,k}(\hat\beta^\lambda_j,\beta_j)=\e[\phi_{1,k}(\eta(B_\textnormal{all}+\tau_\lambda Z;\alpha_\lambda\tau_\lambda),B_\text{all})]
\end{equation*}
and
\begin{equation*}
\liminf_{p\to\infty}\frac{1}{2p}\sum_{j=1}^{2p}\phi(\hat\beta^\lambda_j,\beta_j)\ge\lim_{p\to\infty}\frac{1}{2p}\sum_{j=1}^{2p}\phi_{2,k}(\hat\beta^\lambda_j,\beta_j)=\e[\phi_{2,k}(\eta(B_\textnormal{all}+\tau_\lambda Z;\alpha_\lambda\tau_\lambda),B_\text{all})].
\end{equation*}
Then it follows from equations~\eqref{eq:lasso-amp-dist} and \eqref{eq:sup-inf} that
\begin{multline*}
\frac{1}{2p}\sum_{j=1}^{2p}\mathbb I(\hat\beta^\lambda_j\le t,|\beta_j|\le\varepsilon)\\
\to\frac{\gamma+1}{2}\e[\phi(\eta(\tau_\lambda Z;\alpha_\lambda\tau_\lambda),0)]+\p(0<|B_\textnormal{all}|\le\varepsilon)\e[\phi(\eta(B_\textnormal{all}+\tau_\lambda Z;\alpha_\lambda\tau_\lambda),B_\textnormal{all})\mid0<|B_\textnormal{all}|\le\varepsilon].
\end{multline*}
The second term goes to $0$ as $\varepsilon\to0$ since $|\phi|\le1$.
Now we want to show
\begin{equation*}
\lim_{\varepsilon\to0}\lim_{p\to\infty}\frac{1}{2p}\sum_{j=1}^{2p}\mathbb I(\hat\beta^\lambda_j\le t,|\beta_j|\le\varepsilon)=\lim_{p\to\infty}\frac{1}{2p}\sum_{j=1}^{2p}\mathbb I(\hat\beta^\lambda_j\le t,|\beta_j|=0),
\end{equation*}
while the difference is
\begin{equation*}
\lim_{p\to\infty}\frac{1}{2p}\sum_{j=1}^{2p}\mathbb I(\hat\beta^\lambda_j\le t,0<|\beta_j|\le\varepsilon)\le\lim_{p\to\infty}\frac{1}{2p}\sum_{j=1}^{2p}\mathbb I(0<|\beta_j|\le\varepsilon)=\frac{1-\gamma}{2}\pi_1((0,\varepsilon]),
\end{equation*}
which converges to $0$ as $\varepsilon\to0$.

Now we have shown for any $t\ne0$,
\begin{equation*}
\frac{1}{2p}\sum_{j=1}^{2p}\mathbb I(\hat\beta^\lambda_j\le t,|\beta_j|=0)\cas\frac{\gamma+1}{2}\p(\eta(\tau_\lambda Z;\alpha_\lambda\tau_\lambda)\le t).
\end{equation*}
\end{proof}

By Lemma~\ref{lemma:lasso-amp-even}, letting $B_p=\frac{1}{2p}\sum_{j=1}^{2p}\mathbb I(\hat\beta^\lambda_{j}\le t,\beta_j=0)$,
\begin{equation*}
\e\left[\frac{1}{p}\sum_{j=1}^{p}\mathbb I(\hat\beta^\lambda_{2j}\le t)\right]=\frac{2}{1+\gamma_p}\e[B_p]\to\p(\eta(\tau_\lambda Z;\alpha_\lambda\tau_\lambda)\le t)
\end{equation*}
by the bounded convergence theorem.
\begin{equation*}
\begin{aligned}
\Var\left[\frac{1}{p}\sum_{j=1}^{p}\mathbb I(\hat\beta^\lambda_{2j}\le t)\right]&=\e\left[\frac{2B_p}{1+\gamma_p}(1-\frac{2B_p}{1+\gamma_p})\frac{\gamma_p}{p-1}\right]+\Var\left[\frac{2B_p}{1+\gamma_p}\right]\\
&\le\frac{\gamma_p}{p-1}+\frac{4}{(1+\gamma_p)^2}\Var[B_p]\to0,
\end{aligned}
\end{equation*}
where $\Var[B_p]\to0$ by the bounded convergence theorem since $B_p$ converges to a constant. Now we have shown for $t\ne0$,
\begin{equation*}
\hat F^\text{even}_p(t)\cip\p(\eta(\tau_\lambda Z;\alpha_\lambda\tau_\lambda)\le t),
\end{equation*}
and convergence of $\hat F^\text{odd}_p(s)$ follows from
\begin{equation*}
\hat F^\text{odd}_p(s)=2\times\frac{1}{2p}\sum_{j=1}^{2p}\mathbb I(\hat\beta^\lambda_j\le s)-\hat F^\text{even}_p(s),
\end{equation*}
where the convergence of $\frac{1}{2p}\sum_{j=1}^{2p}\mathbb I(\hat\beta^\lambda_j\le s)$ for $s\ne0$ can be established using uniformly continuous functions as upper and lower bounds on the indicator function by the same technique as in Lemma~\ref{lemma:lasso-amp-even}.

We have now proved the first result of Lemma~\ref{lemma:theorem-lasso-central} in expectation, and proceed to the variance. We need to analyze the following to apply the Markov inequality.
\begin{equation*}
\begin{aligned}
\e[(\frac{1}{p}\sum_{j=1}^p\mathbb I(f(\hat\beta^\lambda_{2j-1},\hat\beta^\lambda_{2j})\le t))^2]&=\frac1p\p(f(\hat\beta^\lambda_{2J_1-1},\hat\beta^\lambda_{2J_1'})\le t)\\
&\quad+\frac{p-1}p\p(f(\hat\beta^\lambda_{2J_1-1},\hat\beta^\lambda_{2J_1'})\le t,f(\hat\beta^\lambda_{2J_2-1},\hat\beta^\lambda_{2J_2'})\le t),
\end{aligned}
\end{equation*}
where
\begin{equation*}
(J_1,J_1',J_2,J_2')\simiid\Unif(\{(i,j,k,\ell)\in[p]^4:i\ne k,j\ne \ell\}).
\end{equation*}
We can evaluate this by showing that asymptotically $(\hat\beta^\lambda_{2J_1-1},\hat\beta^\lambda_{2J_1'},\hat\beta^\lambda_{2J_2-1},\hat\beta^\lambda_{2J_2'})$ converges in distribution to four independent random variables. Per the results we have shown, we can define CDFs $F_\text{even}$ and $F_\text{odd}$ such that for $t\ne0$, $\hat F_p^\text{even}(t)\cip F_\text{even}(t)$ and $\hat F_p^\text{odd}(t)\cip F_\text{odd}(t)$. Thus,
\begin{equation*}
\begin{aligned}
\p(\hat\beta^\lambda_{2J_1-1}\le s,\hat\beta^\lambda_{2J_1'}\le t,\hat\beta^\lambda_{2J_2-1}\le s',\hat\beta^\lambda_{2J_2'}\le t')&=\e\left[\hat F_\text{odd}(s)\hat F_\text{even}(t)\left(\frac{p\hat F_\text{odd}(s')-1}{p-1}\right)\left(\frac{p\hat F_\text{even}(t')-1}{p-1}\right)\right]\\
&\to F_\text{odd}(s)F_\text{even}(t)F_\text{odd}(s')F_\text{even}(t')
\end{aligned}
\end{equation*}
by the bounded convergence theorem, assuming $s,s',t,t'\ne0$ and $s\le s'$ and $t\le t'$ without loss of generality. It follows immediately that $(\hat\beta^\lambda_{2J_1-1},\hat\beta^\lambda_{2J_1'},\hat\beta^\lambda_{2J_2-1},\hat\beta^\lambda_{2J_2'})$ converges in distribution to four independent random variables. We are now able to claim convergence in probability for
\begin{equation*}
\frac{1}{p}\sum_{j=1}^p\mathbb I(f(\hat\beta^\lambda_{2j-1},\hat\beta^\lambda_{2j})\le t).
\end{equation*}
For
\begin{equation*}
\frac{1}{|\{j\in[p]:\beta_{2j-1}=0\}|}\sum_{j=1}^p\mathbb I(f(\hat\beta^\lambda_{2j-1},\hat\beta^\lambda_{2j})\le t,\beta_{2j-1}=0),
\end{equation*}
let $N_p=\{j\in[p]:\beta_{2j-1}=0\}$. To show its convergence, we use the same technique where we compute its first and second moments by finding the asymptotic distribution of a four-dimensional random vector. Specifically, We just need to show that
\begin{equation*}
(\hat\beta^\lambda_{2J_1-1},\hat\beta^\lambda_{2J_1'},\hat\beta^\lambda_{2J_2-1},\hat\beta^\lambda_{2J_2'})\cid(\eta(\tau_\lambda Z_1;\alpha_\lambda\tau_\lambda),\eta(\tau_\lambda Z_2;\alpha_\lambda\tau_\lambda),\eta(\tau_\lambda Z_3;\alpha_\lambda\tau_\lambda),\eta(\tau_\lambda Z_4;\alpha_\lambda\tau_\lambda)),
\end{equation*}
where
\begin{equation*}
(J_1,J_1',J_2,J_2')\sim\Unif(\{(i,j,k,\ell)\in N_p^4: i\ne k,j\ne \ell\})\text{ and }Z_1,Z_2,Z_3,Z_4\simiid\mathcal N(0,1).
\end{equation*}
Note that
\begin{multline*}
\p(\hat\beta^\lambda_{2J_1-1}\le s,\hat\beta^\lambda_{2J_1'}\le t,\hat\beta^\lambda_{2J_2-1}\le s',\hat\beta^\lambda_{2J_2'}\le t')\\
=\e\left[\hat F_p^\text{odd null}(s)\hat F_p^\text{even null}(t)\left(\frac{p\gamma_p\hat F_p^\text{odd null}(s')-1}{p\gamma_p-1}\right)\left(\frac{p\gamma_p-\hat F_p^\text{even null}(t')}{p\gamma_p-1}\right)\right],   
\end{multline*}
assuming $s,s',t,t'\ne0$ and $s\le s'$ and $t\le t'$ without loss of generality, where
\begin{equation*}
\hat F_p^\text{odd null}(s)=\frac1{|N_p|}\sum_{j\in N_p}\mathbb I(\hat\beta^\lambda_{2j-1}\le s),\qquad \hat F_p^\text{even null}(t)=\frac1{|N_p|}\sum_{j\in N_p}\mathbb I(\hat\beta^\lambda_{2j}\le t).
\end{equation*}
The result follows if we show that for $t\ne0$
\begin{equation*}
\hat F_p^\text{odd null}(t),\hat F_p^\text{even null}(t)\cip\p(\eta(\tau_\lambda Z;\alpha_\lambda\tau_\lambda)\le t).
\end{equation*}
This is true because
\begin{equation*}
\gamma_p\hat F_p^\text{odd null}(t)=2B_p-\hat F_\text{even}(t),
\end{equation*}
and $\hat F_p^\text{odd null}(t)\eqd\hat F_p^\text{even null}(t)$ by exchangeability.
\end{proof}

\noindent\textbf{Theorem~\ref{theorem:retro-mc}.} \textit{Consider using the test statistics $T=n^{-1}\ndata{X}^\top\ndata{Y}$ for the CRT, and $T_j=n^{-1}\ndata{X}_j^\top\ndata{Y}$ for multiple testing with CRT $p$-values and knockoffs. Let $M_{\textnormal{retro}}^2$ be the asymptotic second moment of the retrospectively collected $Y_i$, i.e.,
\begin{equation*}
M_{\textnormal{retro}}^2=\frac{\e[Y_\textnormal{raw}^2g(Y_\textnormal{raw})]}{\e[g(Y_\textnormal{raw})]},
\end{equation*}
where $Y_\textnormal{raw}\sim\mathcal N(0,\sigma^2+v_Z^2)$ is drawn from the asymptotic distribution of $Y$ without rejection.\footnote{$M_\textnormal{retro}$ always exists because $g(y)\in[0,1]$ and is not almost everywhere zero.} Note that in Setting~\ref{model:retro-crt}, the corresponding $v_Z^2$ (or $v_{X_{\text{-}j}}^2$) is equal to $\kappa\e[B_0^2]$.
\begin{enumerate}
    \item In Setting~\ref{model:retro}, the asymptotic power of the CRT is equal to that of a $z$-test with standardized effect size
    \[ \frac{h{M_{\textnormal{retro}}}}{v_Z^2+\sigma^2}. \]
\item In Setting~\ref{model:retro-crt}, for almost all $q\in(0,1)$, BH or AdaPT at level $q$ applied to CRT $p$-values using $T_j$ (or $|T_j|$) have one-sided (or two-sided) effective $\pi_\mu$ given by the distribution of $\frac{M_{\textnormal{retro}}}{\sigma^2+\kappa\e[B_0^2]}B_0$ with respect to BH or AdaPT at level $q$.
\item In Setting~\ref{model:retro-crt}, for almost all $q\in(0,1)$, knockoffs with $\tilde{\ndata{X}}$ an i.i.d. copy of $\ndata{X}$, antisymmetric function $f(x,y)=x-y$, test statistic $T_j$, and level $q$ has one-sided effective $\pi_\mu$ given by the distribution of $\frac{M_{\textnormal{retro}}}{\sqrt{2}(\sigma^2+\kappa\e[B_0^2])}B_0$ with respect to AdaPT at level $q$.
\end{enumerate}}

\begin{proof}[Proof of Theorem~\ref{theorem:retro-mc}]
We prove the three statements in the theorem one by one. In the proof, we rescale $T$ and $\tilde T$ by $\sqrt n$, as we will make explicit later when needed.
\begin{enumerate}
    \item The retrospective sampling does not affect how we run the CRT, since the CRT is carried out using the distribution $\tilde X\mid Z\sim\mathcal N(Z\xi,1)$. Hence, we should still analyze \eqref{equation:mc-power}. Same as the rest of the proof of Theorem~\ref{theorem:prop-power-mc}, we only need to show that instead of $\|Y\|/n\cip\sigma^2+v_Z^2$, we have
$\|\ndata{Y}\|^2/n\cip M_\text{retro}^2$, which holds because the $Y_i$'s are i.i.d. with all absolute moments satisfying ($k=0,1,\dots$)
\begin{equation}
\e[|Y_i|^k]=\frac{\e[|Y_\text{non-asymptotic}|^kg(Y_\text{non-asymptotic})]}{\e[g(Y_\text{non-asymptotic})]}\to\frac{\e[|Y_\text{raw}|^kg(Y_\text{raw})]}{\e[g(Y_\text{raw})]},
\label{equation:y-moment}
\end{equation}
where
\begin{equation*}
Y_\text{non-asymptotic}\sim\mathcal N\left(0,\frac{h^2}{n}+(\theta+h\eta/\sqrt n)^\top\Sigma_Z(\theta+h\eta/\sqrt n)+\sigma^2\right).
\end{equation*}
Equation~\eqref{equation:y-moment} holds because we can write, for $k=0,1,\dots$,
\begin{multline*}
\e[|Y_\text{non-asymptotic}|^kg(Y_\text{non-asymptotic})]=\left(\frac{h^2}{n}+(\theta+h\eta/\sqrt n)^\top\Sigma_Z(\theta+h\eta/\sqrt n)+\sigma^2\right)^{k/2}\\
\times\e\left[|W|^kg\left(\sqrt{\frac{h^2}{n}+(\theta+h\eta/\sqrt n)^\top\Sigma_Z(\theta+h\eta/\sqrt n)+\sigma^2}W\right)\right], W\sim\mathcal N(0,1),
\end{multline*}
use \eqref{equation:mc-to-show} to establish
\begin{equation*}
\sqrt{\frac{h^2}{n}+(\theta+h\eta/\sqrt n)^\top\Sigma_Z(\theta+h\eta/\sqrt n)+\sigma^2}W\cid Y_\text{raw},
\end{equation*}
and finally apply the dominated convergence theorem with $|W|^k$ as the dominating function.

\item Applying Lemma~\ref{lemma:gaussian-covariate-conditional}, we have
\begin{equation*}
\ndata{X}^\top\ndata{Y}\mid\ndata{Y},\beta\sim\mathcal N\left(\frac{\|\ndata{Y}\|^2}{\|\beta\|^2+\sigma^2}\beta,\|\ndata{Y}\|^2\left(I-\frac{1}{\|\beta\|^2+\sigma^2}\beta\beta^\top\right)\right).
\end{equation*}
Let $T_j=n^{-1/2}\ndata{X}_j^\top\ndata{Y}$. For $j\ne k$,
\begin{multline*}
\left(\begin{array}{c}
T_j-\frac{M_\textnormal{retro}^2\sqrt n\beta_j}{\kappa\e[B_0^2]+\sigma^2}\\
T_k-\frac{M_\textnormal{retro}^2\sqrt n\beta_k}{\kappa\e[B_0^2]+\sigma^2}\\
\end{array}\right)\mid\ndata{Y},\beta\\
\sim\mathcal N\left(\left(\frac{\|\ndata{Y}\|^2/n}{\|\beta\|^2+\sigma^2}-\frac{M_\textnormal{retro}^2}{\kappa\e[B_0^2]+\sigma^2}\right)\left(\begin{array}{c}
\sqrt n\beta_j\\
\sqrt n\beta_k\\
\end{array}\right),\frac{\|\ndata{Y}\|^2}{n}\left[\begin{array}{cc}
1-\frac{\beta_j^2}{\|\beta\|^2+\sigma^2} & -\frac{\beta_j\beta_k}{\|\beta\|^2+\sigma^2}\\
-\frac{\beta_j\beta_k}{\|\beta\|^2+\sigma^2} & 1-\frac{\beta_k^2}{\|\beta\|^2+\sigma^2}\\
\end{array}\right]\right),
\end{multline*}
which converges in distribution to $\mathcal N(0,M^2_\text{retro}I_2)$. We get the desired result by dividing both sides by $M_\textnormal{retro}$.

\item We have shown in Lemma~\ref{lemma:mc-knockoff} that
\begin{equation*}
\left(\begin{array}{c}
\ndata{X}^\top\ndata{Y}\\
\tilde{\ndata{X}}^\top\ndata{Y}\\
\end{array}\right)\mid\ndata{Y},\beta\sim\mathcal N\left(\frac{\|\ndata{Y}\|^2}{\|\beta\|^2+\sigma^2}\left(\begin{array}{c}
\beta\\
0\\
\end{array}\right),\|\ndata{Y}\|^2\left[\begin{array}{cc}
   I-\frac{1}{\|\beta\|^2+\sigma^2}\beta\beta^\top  & 0 \\
   0  & I\\
\end{array}\right]\right).
\end{equation*}
Let $T_j=n^{-1/2}\ndata{X}_j^\top\ndata{Y}$, $\tilde T_j=n^{-1/2}\tilde{\ndata{X}}_j^\top\ndata{Y}$. For $j\ne k$,
\begin{multline*}
\left(\begin{array}{c}
T_j-\frac{M_\textnormal{retro}^2\sqrt n\beta_j}{\kappa\e[B_0^2]+\sigma^2}\\
T_k-\frac{M_\textnormal{retro}^2\sqrt n\beta_k}{\kappa\e[B_0^2]+\sigma^2}\\
\tilde T_j\\
\tilde T_k\\
\end{array}\right)\mid\ndata{Y},\beta\\
\sim\mathcal N\left(\left(\frac{\|\ndata{Y}\|^2/n}{\|\beta\|^2+\sigma^2}-\frac{M_\textnormal{retro}^2}{\kappa\e[B_0^2]+\sigma^2}\right)\left(\begin{array}{c}
\sqrt n\beta_j\\
\sqrt n\beta_k\\
0\\
0\\
\end{array}\right),\frac{\|\ndata{Y}\|^2}{n}\left[\begin{array}{cccc}
1-\frac{\beta_j^2}{\|\beta\|^2+\sigma^2} & -\frac{\beta_j\beta_k}{\|\beta\|^2+\sigma^2} & 0 & 0\\
-\frac{\beta_j\beta_k}{\|\beta\|^2+\sigma^2} & 1-\frac{\beta_k^2}{\|\beta\|^2+\sigma^2} & 0 & 0\\
0 & 0 & 1 & 0\\
0 & 0 & 0 & 1\\
\end{array}\right]\right),
\end{multline*}
which converges in distribution to $\mathcal N(0,M^2_\text{retro}I_4)$. We get the desired result by dividing both sides by $M_\textnormal{retro}$.
\end{enumerate}
\end{proof}
\section{Fixed-X test with the OLS coefficient}
\label{sec:fixed-X-OLS}
Consider the fixed-X test in Section~\ref{sec:CRT-OLS}. Under the sequence of alternatives $\beta=h/\sqrt{n}$, the power is
\begin{equation*}
\p_{\beta=h/\sqrt n}\left(\hat\beta_1>z_\alpha\sigma\sqrt{{\hat\Omega}_{11}}\right)=\e\left[\p_{\beta=h/\sqrt n}\left(\hat\beta_1>z_\alpha\sigma\sqrt{{\hat\Omega}_{11}}\mid\ndata{X},\ndata{Z}\right)\right]=\e\left[\Phi\left(\frac{h}{\sigma\sqrt{n\hat\Omega_{11}}}-z_\alpha\right)\right].
\end{equation*}
Since $\hat\Omega\sim\text{Inv-Wishart}(\Omega=\Sigma^{-1},n)$, where $\Sigma$ is the joint covariance matrix of $(X,Z)$, we have
$n\hat\Omega_{11}\cip\Omega_{11}/(1-\kappa)$ by moment cacluations. Note that $\Omega_{11}=1$, and it then follows that the asymptotic power of the fixed-X test with OLS coefficient is
\begin{equation*}
\p_{\beta=h/\sqrt n}(\hat\beta_1>z_\alpha\sigma\sqrt{{\hat\Omega}_{11}})\to\Phi\left(\frac{h}{\sigma}\sqrt{\frac{1-\kappa}{\Omega_{11}}}-z_\alpha\right)=\Phi\left(\frac{h}{\sigma}\sqrt{1-\kappa}-z_\alpha\right),
\end{equation*}
the same as CRT with the OLS coefficient!
\section{The CRT with unlabeled data}
\label{sec:conditional-crt}
We discuss the conditional CRT in detail using the concrete example in Section~\ref{sec:CRT-conditional}, i.e., Setting~\ref{model:moderate-dim-lr} but with the following changes: $\xi$ is unknown; $\Var(X\,|\,Z)=1$ but is unknown to the CRT.

We notice that we could write
\begin{equation*}
\ndata{X}_*=\ndata{Z}_*\xi+\varepsilon^X_*,\varepsilon^X_*\sim\mathcal N(0,I_{n_*}),
\end{equation*}
where $\varepsilon_*^X$ is independent of $\ndata{Z}_*$ and $\varepsilon$. It can be seen that under the null hypothesis $H_0:{X}\ci {Y}\mid {Z}$, we have $\varepsilon^X_*\ci\ndata{Y}\mid\ndata{Z}_*$. Then
\begin{equation*}
\begin{aligned}
\ndata{X}_*&=\left(\ndata{Z}_*\left(\ndata{Z}_*^\top\ndata{Z}_*\right)^{-1}\ndata{Z}_*^\top\right)\ndata{X}_*+\left(I_{n_*}-\ndata{Z}_*\left(\ndata{Z}_*^\top\ndata{Z}_*\right)^{-1}\ndata{Z}_*^\top\right)\ndata{X}_*\\
&=\left(\ndata{Z}_*\left(\ndata{Z}_*^\top\ndata{Z}_*\right)^{-1}\ndata{Z}_*^\top\right)\ndata{X}_*+\left(I_{n_*}-\ndata{Z}_*\left(\ndata{Z}_*^\top\ndata{Z}_*\right)^{-1}\ndata{Z}_*^\top\right)\varepsilon^X_*\\
&=\left(\ndata{Z}_*\left(\ndata{Z}_*^\top\ndata{Z}_*\right)^{-1}\ndata{Z}_*^\top\right)\ndata{X}_*+A_{\ndata{Z}_*}A_{\ndata{Z}_*}^\top\varepsilon^X_*\\
&=\left(\ndata{Z}_*\left(\ndata{Z}_*^\top\ndata{Z}_*\right)^{-1}\ndata{Z}_*^\top\right)\ndata{X}_*+A_{\ndata{Z}_*}\|A_{\ndata{Z}_*}^\top\varepsilon_*^X\|\frac{A_{\ndata{Z}_*}^\top\varepsilon_*^X}{\|A_{\ndata{Z}_*}^\top\varepsilon_*^X\|},
\end{aligned}
\end{equation*}
where $A_{\ndata{Z}_*}$ is an $n_*\times(n_*-p)$ matrix that satisfies
\begin{equation*}
A_{\ndata{Z}_*}A_{\ndata{Z}_*}^\top=I_{n_*}-\ndata{Z}_*\left(\ndata{Z}_*^\top\ndata{Z}_*\right)^{-1}\ndata{Z}_*^\top,\quad A_{\ndata{Z}_*}^\top A_{\ndata{Z}_*}=I_{n_*-p}.
\end{equation*}
Such an $A_{\ndata{Z}_*}$ exists because $I_{n_*}-\ndata{Z}_*\left(\ndata{Z}_*^\top\ndata{Z}_*\right)^{-1}\ndata{Z}_*^\top$ is a projection matrix.
Hence, $A_{\ndata{Z}_*}^\top\varepsilon_*^X\mid\ndata{Z}_*\sim\mathcal N(0,I_{n_*-p})$. Let $H_{\ndata{Z}_*}=\ndata{Z}_*\left(\ndata{Z}_*^\top\ndata{Z}_*\right)^{-1}\ndata{Z}_*^\top$ and it follows that under the null,
\begin{equation*}
\ndata{X}_*\eqd H_{\ndata{Z}_*}\ndata{X}_*+A_{\ndata{Z}_*}\|A_{\ndata{Z}_*}^\top\varepsilon_*^X\|\cdot\tilde{U}\Bigm|\ndata{Z}_*,\ndata{Y},H_{\ndata{Z}_*}\ndata{X}_*,\|A_{\ndata{Z}_*}^\top\varepsilon_*^X\|,
\end{equation*}
where $\mathcal L\left(\tilde U\mid\ndata{Z}_*,\ndata{Y},H_{\ndata{Z}_*}\ndata{X}_*,\|A_{\ndata{Z}_*}^\top\varepsilon_*^X\|\right)$ is the uniform distribution on the sphere $\mathbb S^{n_*-p-1}$. Armed with this observation, we now know under the null,
\begin{equation*}
T(\ndata{X}_*,\ndata{Y},\ndata{Z}_*)\eqd T\left(H_{\ndata{Z}_*}\ndata{X}_*+A_{\ndata{Z}_*}\|A_{\ndata{Z}_*}^\top\varepsilon_*^X\|\cdot\tilde{U},\ndata{Y},\ndata{Z}_*\right)\Bigm|\ndata{Z}_*,\ndata{Y},H_{\ndata{Z}_*}\ndata{X}_*,\|A_{\ndata{Z}_*}^\top\varepsilon_*^X\|.
\end{equation*}
We know exactly the conditional distribution on the right hand side (at least, we can sample from it to get an empirical estimate), and a cutoff can thus be obtained.

As a concrete example, consider the marginal correlation $T_\text{MC}(\ndata{X}_*,\ndata{Y},\ndata{Z}_*)=\ndata{Y}^\top\ndata{X}$. Same as the previous derivation, we write
\begin{equation*}
T_\text{MC}(\ndata{X}_*,\ndata{Y},\ndata{Z}_*)=\ndata{Y}^\top\left(\ndata{Z}\left(\ndata{Z}_*^\top\ndata{Z}_*\right)^{-1}\ndata{Z}_*^\top\right)\ndata{X}_*+\ndata{Y}^\top\left(\ndata{X}-\left(\ndata{Z}\left(\ndata{Z}_*^\top\ndata{Z}_*\right)^{-1}\ndata{Z}_*^\top\right)\ndata{X}_*\right).
\end{equation*}
The first term is a discardable constant conditional on $\ndata{Z}_*,\ndata{Y},H_{\ndata{Z}_*}\ndata{X}_*,\|A_{\ndata{Z}_*}^\top\varepsilon_*^X\|$, making the second term the essential part, which admits an interesting interpretation as a generalization of the OLS coefficient with unlabeled data. To elaborate, note that we can write
\begin{equation*}
\beta\ndata{X}+\ndata{Z}\theta=\beta(I-\ndata{Z}(\ndata{Z}^\top \ndata{Z})^{-1}\ndata{Z}^\top)\ndata{X}+\ndata{Z}(\beta(\ndata{Z}^\top \ndata{Z})^{-1}\ndata{Z}^\top \ndata{X}+\theta)),
\end{equation*}
and thus run OLS equivalently by solving
\begin{equation*}
(\hat\beta,\hat\theta)=\argmin\limits_{\beta,\theta}\|\ndata{Y}-\beta(I-\ndata{Z}(\ndata{Z}^\top \ndata{Z})^{-1}\ndata{Z}^\top)\ndata{X}+\ndata{Z}(-\beta(\ndata{Z}^\top \ndata{Z})^{-1}\ndata{Z}^\top \ndata{X}-\theta))\|^2_2.
\end{equation*}
Then we have
\begin{equation}
\hat\beta=[\ndata{X}^\top(I-\ndata{Z}(\ndata{Z}^\top \ndata{Z})^{-1}\ndata{Z}^\top)^2\ndata{X}]^{-1}\ndata{X}^\top(I-\ndata{Z}(\ndata{Z}^\top \ndata{Z})^{-1}\ndata{Z}^\top)\ndata{Y}=\frac{\sum_{i=1}^n\tilde\tau_iY_i}{\sum_{i=1}^n\tilde\tau_i^2},
\label{eq:model-x-beta1}
\end{equation}
where
\begin{equation*}
\tilde\tau=\ndata{X}-\ndata{Z}(\ndata{Z}^\top \ndata{Z})^{-1}\ndata{Z}^\top \ndata{X}\mid \ndata{Z}\sim\mathcal N(0,I-\ndata{Z}(\ndata{Z}^\top \ndata{Z})^{-1}\ndata{Z}^\top).
\end{equation*}
Since the denominator of \eqref{eq:model-x-beta1} satisfies $n^{-1}\sum_{i=1}^n\tilde\tau_i^2\cip(1-\kappa)$, the essence of OLS statistic without unlabeled data is $\ndata{Y}^\top(I-\ndata{Z}(\ndata{Z}^\top \ndata{Z})^{-1}\ndata{Z}^\top)\ndata{X}$. A natural generalization is thus $\ndata{Y}^\top\left(I_{n\times n_*}-\ndata{Z}\left(\ndata{Z}_*^\top\ndata{Z}_*\right)^{-1}\ndata{Z}_*^\top\right)\ndata{X}_*$.

On the other hand, if we consider the original OLS statistic, its essence $\ndata{Y}^\top(I-\ndata{Z}(\ndata{Z}^\top \ndata{Z})^{-1}\ndata{Z}^\top)\ndata{X}$ is equal to $\ndata{Y}^\top(I-\ndata{Z}(\ndata{Z}^\top \ndata{Z})^{-1}\ndata{Z}^\top)\varepsilon^X$, and the only parameter unknown for its distribution given $\ndata{Y},\ndata{Z}_*$ is the scalar $\Var(X\,|\,Z)$. As we eventually learn $\Var(X\,|\,Z)$ asymptotically, even from the labeled data alone, it makes no difference by conditioning on the unlabeled data.

Now we proceed by removing the constant from $T_\text{MC}$,
\begin{multline*}
T^\text{ess}_\text{MC}(\ndata{X}_*,\ndata{Y},\ndata{Z}_*)=\ndata{Y}^\top\left(\ndata{X}-\left(\ndata{Z}\left(\ndata{Z}_*^\top\ndata{Z}_*\right)^{-1}\ndata{Z}_*^\top\right)\ndata{X}_*\right)\\
=\ndata{Y}^\top\left(I_{n\times n_*}-\ndata{Z}\left(\ndata{Z}_*^\top\ndata{Z}_*\right)^{-1}\ndata{Z}_*^\top\right)\varepsilon^X_*\eqd\ndata{Y}^\top I_{n\times n_*}A_{\ndata{Z}_*}\|A_{\ndata{Z}_*}^\top\varepsilon_*^X\|\cdot\tilde{U}\Bigm|\ndata{Z}_*,\ndata{Y},H_{\ndata{Z}_*}\ndata{X}_*,\|A_{\ndata{Z}_*}^\top\varepsilon_*^X\|.
\end{multline*}
The upper quantile of the last distribution could be obtained by Monte Carlo simulations. We notice that if we know $\Var(X\,|\,Z)=1$, we do not have to condition on $\|A_{\ndata{Z}_*}^\top\varepsilon_*^X\|$ and we can directly use the quantile of $\ndata{Y}^\top\left(I_{n\times n_*}-\ndata{Z}\left(\ndata{Z}_*^\top\ndata{Z}_*\right)^{-1}\ndata{Z}_*^\top\right)\varepsilon^X_*$, which simply follows an explicit Gaussian distribution conditional on $\ndata{Z}_*,\ndata{Y},H_{\ndata{Z}_*}\ndata{X}_*$. In fact, in our power analysis, we first make the above observation formal and show that we can indeed assume $\Var(X\,|\,Z)$ is known. 
\section{Comparison of the two $\tau_\lambda$'s for the CRT and knockoffs}
\label{sec:tau}
For presentational simplicity, we write $\tau_{\lambda_\text{CRT}}$ and $\tilde\tau_{\lambda_\text{KF}}$ for $\tau_{\lambda_\text{CRT}}^\text{CRT}$ and $\tau_{\lambda_\text{KF}}^\text{KF}$. We would like to show that the lowest $\tau_{\lambda_\text{CRT}}$ is smaller than the lowest $\tilde\tau_{\lambda_\text{KF}}$. Before we start, we make an important note on the notation: there are three parameters $\alpha$, $\tau$, and $\lambda$ in \citet{bayati2011lasso} to characterize the lasso asymptotics, and $\tau$ and $\lambda$ are defined as functions of $\alpha$ over a suitable range. In the rest of this article, we add subscript $\lambda$ to $\alpha$ and $\tau$ ($\alpha_\lambda$ and $\tau_\lambda$) to indicate their association with $\lambda$. In this section, we need to consider the change of $\tau$ as $\alpha$ varies, and as mentioned at the beginning of this section, there are two different $\tau$'s for CRT and knockoffs, so we use $\alpha$ without a subscript as a variable that varies, $\tau_{\lambda_\text{CRT}}$ and $\tilde\tau_{\lambda_\text{KF}}$ as functions of $\alpha$ in CRT and knockoffs settings, respectively, and $\tau$ as a dummy variable that does not depend on $\alpha$. We emphasize that the $\alpha$ in this section is purely an AMP parameter and has nothing to do the level of a hypothesis test.

For an $\alpha$ in a suitable range, we have $\tau_{\lambda_\text{CRT}}$ and $\tilde\tau_{\lambda_\text{KF}}$ as implicit functions of $\alpha$ via equations
\begin{equation}
\begin{aligned}
\tau_{\lambda_\text{CRT}}^2&=\sigma^2+\kappa\e[(\eta(B_0+\tau_{\lambda_\text{CRT}}Z;\alpha\tau_{\lambda_\text{CRT}})-B_0)^2],\\
\tilde\tau_{\lambda_\text{KF}}^2&=\sigma^2+2\kappa\e[(\eta(IB_0+\tilde\tau_{\lambda_\text{KF}}Z;\alpha\tilde\tau_{\lambda_\text{KF}})-IB_0)^2],
\end{aligned}
\label{equation:two-taus}
\end{equation}
where $I$ is a Bernoulli random variable with parameter $1/2$ independent of $B_0$ and $Z$. By Proposition~1.3 in \citet{bayati2011lasso}, every valid $\alpha$ for the knockoff setting is a valid $\alpha$ for the CRT setting (since the left hand side of Equation~(1.14) is non-increasing in $\alpha$ and the CRT setting doubles the $\delta$ parameter there compared to knockoffs). Therefore, if we can show that for every $\alpha$ valid in the knockoff setting, it defines a $\tau_{\lambda_\text{CRT}}$ smaller than $\tilde\tau_{\lambda_\text{KF}}$, which is valid and thus actually corresponds to a $\lambda$, then we have shown that the lowest $\tau_{\lambda_\text{CRT}}$ is smaller than the lowest $\tilde\tau_{\lambda_\text{KF}}$.

From now on, we fix $\alpha$ and thus fix $\tau_{\lambda_\text{CRT}}$ and $\tilde\tau_{\lambda_\text{KF}}$ as functions of $\alpha$. We first see that for any $\tau$,
\begin{equation*}
\e[(\eta(B_0+\tau Z;\alpha\tau)-B_0)^2]=\gamma\e[\eta(\tau Z;\alpha\tau)^2]+(1-\gamma)\e_{B\sim\pi_1}[(\eta(B+\tau Z;\alpha\tau)-B^2)],
\end{equation*}
and
\begin{equation*}
\e[(\eta(IB_0+\tau Z;\alpha\tau)-IB_0)^2]=\frac{1+\gamma}{2}\e[\eta(\tau Z;\alpha\tau)^2]+\frac{1-\gamma}{2}\e_{B\sim\pi_1}[(\eta(B+\tau Z;\alpha\tau)-B^2)],
\end{equation*}
so
\begin{equation*}
\e[(\eta(B_0+\tilde\tau_{\lambda_\text{KF}}Z;\alpha\tilde\tau_{\lambda_\text{KF}})-B_0)^2]<2\e[(\eta(IB_0+\tilde\tau_{\lambda_\text{KF}}Z;\alpha\tilde\tau_{\lambda_\text{KF}})-IB_0)^2].
\end{equation*}
By \eqref{equation:two-taus}, we immediately have
\begin{equation*}
\tilde\tau_{\lambda_\text{KF}}^2>\sigma^2+\kappa\e[(\eta(B_0+\tilde\tau_{\lambda_\text{KF}}Z;\alpha\tilde\tau_{\lambda_\text{KF}})-B_0)^2].
\end{equation*}
When $\tau^2\to0$, we have
\begin{equation*}
\tau^2<\sigma^2<\sigma^2+\kappa\e[(\eta(B_0+\tau Z;\alpha\tau)-B_0)^2].
\end{equation*}
Now consider the function $f_\alpha(\tau)=\sigma^2+\kappa\e[(\eta(B_0+\tau Z;\alpha\tau)-B_0)^2]-\tau^2$, where we have seen $f_\alpha(\tilde\tau_{\lambda_\text{KF}})<0$ and $f_\alpha(0^+)>0$. We can show $f_a$ is a continuous function of $\tau$ on $(0,\tilde\tau_{\lambda_\text{KF}}]$ with the dominated convergence theorem, because
\begin{equation*}
(\eta(B_0+\tau Z;\alpha\tau)-B_0)^2\le\max\left((\tau Z+\alpha\tau)^2,(\tau Z-\alpha\tau)^2\right)\le\underbrace{\tilde\tau_{\lambda_\text{KF}}^2\max((Z+\alpha)^2,(Z-\alpha)^2)}_{\text{dominating function}}.
\end{equation*}
By continuity, there is at least one $\tau\in(0,\tilde\tau_{\lambda_\text{KF}})$ that satisfies $f_a(\tau)=0$. i.e.,
\begin{equation*}
\tau^2=\sigma^2+\kappa\e[(\eta(B_0+\tau Z;\alpha\tau)-B_0)^2].
\end{equation*}
Due to uniqueness (Proposition~1.3 in \citet{bayati2011lasso}), this solution is $\tau_{\lambda_\text{CRT}}$ and thus $\tau^2_{\lambda_\text{CRT}}\in(0,\tilde\tau^2_{\lambda_\text{KF}})$.

\section{Simulation details}

\subsection{Oracle methods}
Before presenting our simulation results, we discuss some Bayesian methods as baselines that we compare against. They are referred to as oracle methods, because they require the knowledge of the prior distribution of the parameters, which our methods do not use and we generally do not expect to be available in practice.
\subsubsection{Controlling the Bayesian FDR}
\label{sec:bayes-fdr}
In the multiple testing problem, when we know the prior distribution of the parameters (e.g., when we know $\gamma$ and $\pi_1$ in Setting~\ref{model:lr-iid}), we can run a Bayesian procedure to incorporate the prior knowledge that we have. 
Suppose we have the posterior probabilities of the covariates being non-null as $p^B_1,\dots,p^B_p$. Without loss of generality, we assume they are ordered from large to small. Then we find $k$ such that
\begin{equation*}
\frac{\sum_{j=1}^kp^B_j}{k}>1-q>\frac{\sum_{j=1}^{k+1}p^B_j}{k+1}
\end{equation*}
and reject $1,2,\dots,k$. Finally, we reject $k+1$ with probability $r$, where $r$ satisfies
\begin{equation*}
(1-r)\frac{\sum_{j=1}^kp^B_j}{k}+r\frac{\sum_{j=1}^{k+1}p^B_j}{k+1}=1-q.
\end{equation*}
It is straightforward to see that this procedure controls the Bayesian FDR at level $q$. In fact, the above procedure controls
\begin{equation*}
\e[\text{FDP}\mid\ndata{X},\ndata{Y},\ndata{Z}]\le q.
\end{equation*}
This is neither stronger nor weaker than the FDR control conditional on the parameters elsewhere in the article:
\begin{equation*}
\e[\text{FDP}\mid\text{parameters of the model}]\le q,
\end{equation*}
while they both control the unconditional FDR.
In the simulations, we run a Gibbs sampler to estimate those posterior probabilities.

\subsubsection{Bayesian statistic is optimal for the CRT}
\label{sec:crt-opt-bayes}
Here, we show that in Setting~\ref{model:moderate-dim-lr}, the posterior probability is the optimal statistic to use for the CRT. Suppose we have a true prior $\pi$ for $\theta$, where $\pi$ is a mixture of $\delta_0$ and $\pi_1$ and $\pi_1$ has no point mass at $0$. Then by the Neyman--Pearson Lemma, the optimal level-$\alpha$ test for $H_0:\beta=0$ against $H_1:\beta\sim\pi_1$ among valid CRTs (i.e., tests conditional on $\ndata{Y},\ndata{Z}$) is the likelihood ratio test that rejects when
\begin{equation*}
T_\text{opt}(\ndata{X},\ndata{Y},\ndata{Z})=\frac{\iint\p(\ndata{X},\ndata{Z})\p_{\theta,\beta}(\ndata{Y}\mid\ndata{X},\ndata{Z})\pi(\theta)\di\theta\pi_1(\beta)\di\beta}{\int\p(\ndata{X},\ndata{Z})\p_{\theta,\beta=0}(\ndata{Y}\mid\ndata{X},\ndata{Z})\pi(\theta)\di\theta}>c_\alpha(\ndata{Y},\ndata{Z}),
\end{equation*}
where $c_\alpha$ is an appropriate cutoff (if $X$ is discrete, randomize when $T_\text{opt}=c_\alpha$). Interestingly, if we have an almost-correct prior on $\beta$, i.e., for a $\gamma\in(0,1)$, $\beta\sim\gamma\delta_0+(1-\gamma)\pi_1$, and $\beta\ci\theta$ {a priori}, then the posterior probability of $H_1$ is
\begin{multline*}
\p(H_1\text{ holds}\mid\ndata{X},\ndata{Y},\ndata{Z})\\
=\frac{(1-\gamma)\iint\p(\ndata{X},\ndata{Z})\p_{\theta,\beta}(\ndata{Y}\mid\ndata{X},\ndata{Z})\pi(\theta)\di\theta\pi_1(\beta)\di\beta}{\gamma\int\p(\ndata{X},\ndata{Z})\p_{\theta,\beta=0}(\ndata{Y}\mid\ndata{X},\ndata{Z})\pi(\theta)\di\theta+(1-\gamma)\iint\p(\ndata{X},\ndata{Z})\p_{\theta,\beta}(\ndata{Y}\mid\ndata{X},\ndata{Z})\pi(\theta)\di\theta\pi_1(\beta)\di\beta},
\end{multline*}
which is a monotone function of $T_\text{opt}(\ndata{X},\ndata{Y},\ndata{Z})$. Hence, the posterior probability is equivalent to the likelihood ratio and thus optimal, regardless of whether $\gamma$ is correctly specified.

We point out that the Bayesian methods shown in Figure~\ref{figure:comp_all} are not the BH-CRT using the oracle statistic introduced in this section, which is prohibitively expensive to simulate for large $n$ and $p$, and we expect it to have similar performance to the BH-CRT with distilled lasso based on experiments in Section~\ref{sec:simu-crt}.

\subsection{Comparison of BH and AdaPT applied to CRT $p$-values}
\label{sec:bh-adapt}

\begin{figure}[H]
    \centering
\includegraphics[width = 0.9\textwidth]{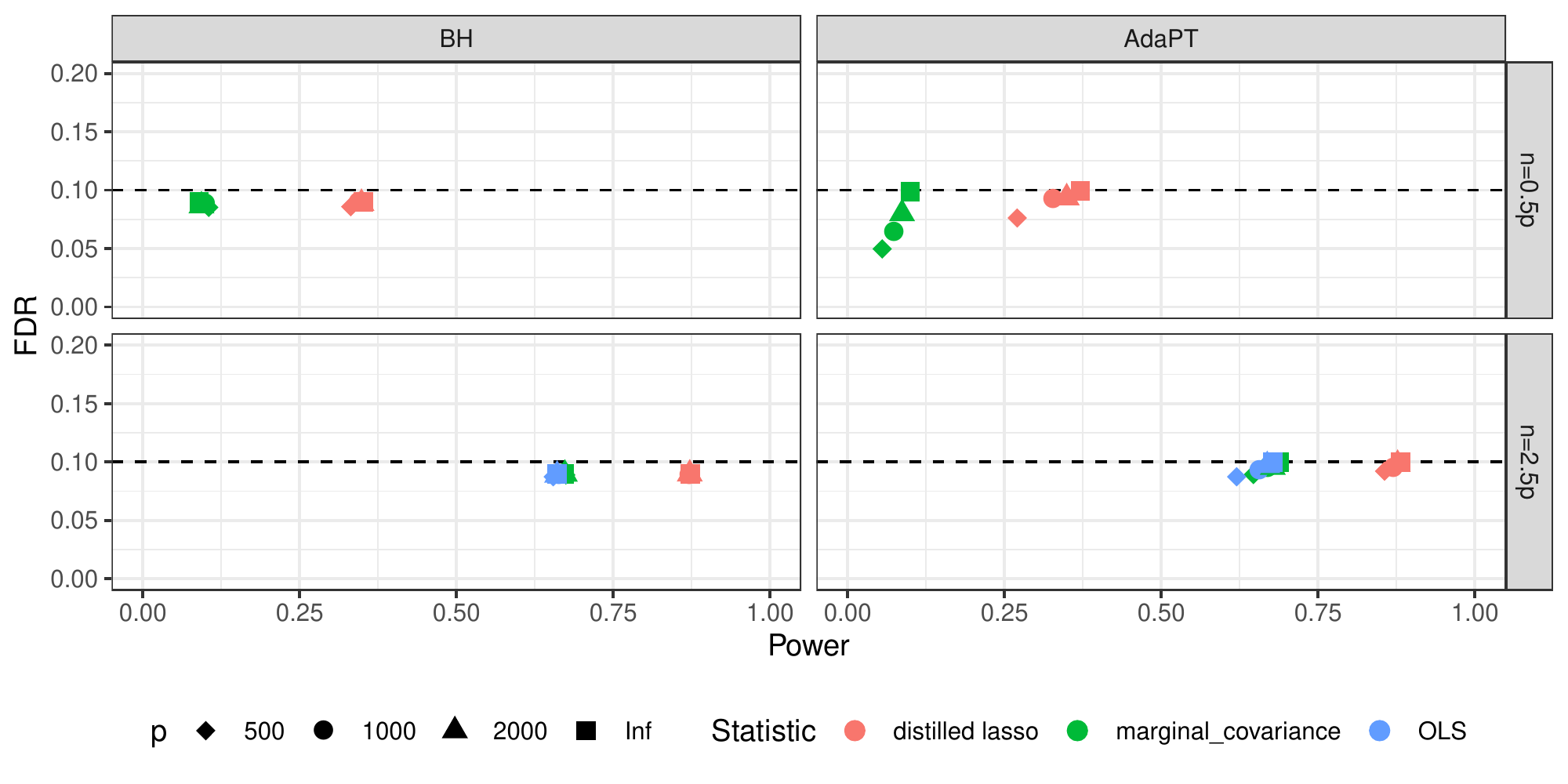}
    \caption{Comparison of BH and AdaPT applied to CRT $p$-values at FDR level $0.1$. The settings are the same as in Figure~\ref{figure:comp_all}. All standard errors are below $0.01$.} 
    \label{figure:BH-AdaPT}
\end{figure}

\end{document}